\documentclass[amsthm,11pt]{elsarticle}

\usepackage{amsmath,amssymb,amsthm}

\DeclareMathAlphabet\mathbfcal{OMS}{cmsy}{b}{n}
\usepackage{multicol}
\usepackage{mathrsfs,fullpage,url}
\usepackage[shortlabels]{enumitem}
\usepackage{stmaryrd,bm}
\usepackage[linesnumbered,vlined,ruled,noresetcount]{algorithm2e}
\SetAlgorithmName{Algorithm}{Algorithm}{List of Algorithms}

\usepackage{hyperref}

\usepackage{graphics}
 \usepackage{color}

\def\lc{\ensuremath{\mathsf{lc}}}

\def\LM{\ensuremath{\textsc{lm}}}

\def\prem{\ensuremath{\mathsf{prem}}}
\def\pquo{\ensuremath{\mathsf{pquo}}}
\def\coeff{\ensuremath{\mathsf{coeff}}}
\def\Sr{\ensuremath{\mathsf{Sres}}}
\def\S{\ensuremath{\mathsf{S}}}

\def\dreg{\ensuremath{\mathsf{d}_{\mathrm{reg}}}}

\def\DEG{\ensuremath{\mathsf{DEG}}}

\def\tdeg{\ensuremath{\mathrm{tdeg}}}

\def\cat{\ensuremath{\mathsf{~cat~}}}

\def\l{\ensuremath{\langle}}
\def\r{\ensuremath{\rangle}}

\def\calP{\ensuremath{\mathcal{P}}}

\def\calF{\ensuremath{\mathcal{F}}}

\def\calG{\ensuremath{\mathcal{G}}}

\def\calL{\ensuremath{\mathcal{L}}}

\def\xgcd{\ensuremath{\mathrm{xgcd}}}
\def\gcd{\ensuremath{\mathrm{gcd}}}

\def\out{\ensuremath{\mathtt{OUT}}}

\def\SToGB{\ensuremath{\mathtt{SubresToGB}}}

\def\SToGBD5{\ensuremath{\mathtt{SubresToGB\_D5}}}
\def\SToGBA{\ensuremath{\mathtt{SubresToGB\_Aux}}}
\def\Res{\ensuremath{\mathrm{Res}}}
\def\LNN{\ensuremath{\mathtt{LastNonNil}}}
\def\LNND5{\ensuremath{\mathtt{LastNonNil\_D5}}}
\def\invNil{\ensuremath{\mathtt{invertNil}}}
\def\WF{\ensuremath{\mathtt{WeierstrassForm}}}
\def\MFD5{\ensuremath{\mathtt{MonicForm\_D5}}}
\def\WFD5{\ensuremath{\mathtt{WeierstrassForm\_D5}}}
\def\MF{\ensuremath{\mathtt{MonicForm}}}

\def\DecD5{\ensuremath{\mathtt{ExactDec\_D5}}}

\def\exD5{\ensuremath{\mathtt{Extract\_D5}}}
\def\isol{\ensuremath{\mathtt{IsolFactor}}}

\def\HL{\ensuremath{\mathtt{HenselLift}}}
\def\MQ{\ensuremath{\mathtt{MusserQ}}}

\def\bfS{\ensuremath{\mathbf{S}}}

\def\bfZ{\ensuremath{\mathbf{Z}}}

 \def\bfW{\ensuremath{\mathbf{W}}}
 \def\bMb{\ensuremath{\mathbf{Mb}}}
 \def\bMa{\ensuremath{\mathbf{Ma}}}
 \def\bfG{\ensuremath{\mathbf{G}}}

\def\Z{\ensuremath{\mathbb{Z}}}

\def\N{\ensuremath{\mathbb{N}}}

\def\sqfp{\ensuremath{\mathrm{sqfp}}}

\def\fa{\ensuremath{\mathfrak{a}}}

\def\fm{\ensuremath{\mathfrak{m}}}

\newtheorem{Lem}{Lemma}
\newtheorem{Prop}{Proposition}
\newtheorem{Theo}{Theorem}
\newtheorem{Cor}{Corollary}
\theoremstyle{definition}
\newtheorem{Def}{Definition}

\theoremstyle{definition}
\newtheorem{ex}{Example}
\theoremstyle{remark}
\newtheorem{Rmk}{Remark}

\begin{document}

\begin{frontmatter}
  \title{Lexicographic Gr\"obner bases of  bivariate polynomials modulo a univariate one}

\author{Xavier Dahan}
\address{Tohoku University, IEHE. 41 Kawauchi, Aoba ward, Sendai. 980-8576, Japan}

\ead{xdahan@gmail.com}



\begin{abstract}
Let $T(x)\in k[x]$ be a monic non-constant polynomial
and write $R=k[x]/\l T\r$
the quotient ring. 
Consider two bivariate polynomials $a(x, y), b(x, y)\in
R[y]$.
In a first part,  $T = p^e$ is assumed to be the power of an irreducible
polynomial $p$. A new algorithm  that computes  a minimal lexicographic Gr\"obner basis of the ideal
$\l a,\ b, \ p^e\r$, is introduced.
A second part
extends this algorithm when $T$ is general through the ``local/global''
principle realized by a generalization of ``dynamic evaluation'',
restricted so far to a polynomial $T$ that is squarefree.
The algorithm produces splittings according to the case distinction ``invertible/nilpotent'',
extending the usual ``invertible/zero'' in classic dynamic evaluation.
This algorithm belongs to the Euclidean family, 
the core being  a  subresultant sequence of $a$
and $b$  modulo $T$. 
In particular no factorization or Gr\"obner basis computations are necessary.
The theoretical background relies on Lazard's structural
theorem for lexicographic Gr\"obner bases in two variables.
An implementation is realized in Magma. Benchmarks
 show  clearly the benefit, sometimes important, of this approach
compared to the  Gr\"obner bases approach.
\end{abstract}

\end{frontmatter}


\section{Introduction}

\subsection{Context and results}
Gr\"obner bases are a major tool to solve and manipulate
systems of polynomial equations, as well  as computing in
their quotient algebras.
Modern and most efficient algorithms rely on linear
algebra on variants of Macaulay matrices~\cite{F4, F5}.
Another class of methods, the triangular decomposition, rely on some broad generalizations
of the Euclidean algorithm, as initiated by Ritt~\cite{Ri32}, Wu~\cite{WWT86}, and pursued by many researchers
(see surveys~\cite{Hu1, ChenMMM11, wang2012elimination} for details).
Starting from an input system $\calF$,
these methods produce a family of {\em triangular sets}  $(T_i)
_i$ which enjoy  the elimination property and satisfies $\cap_i \l T_i\r\simeq  \l \calF\r$. 
In dimension zero, these sets form a particular subclass of lexicographic Gr\"obner bases
(lexGb thereafter).
This representation is well-understood and implemented since already more than two decades,
especially for radical ideals. Although several attempts to {\em represent multiplicities}
have come out~\cite{Cheng2014,Li2003multiplicities,Marcus2012-short,Alvandi2015-short}
they are limited in scope, and resort to sophisticated concepts
while obtaining partial information only.
Triangular sets, hence standard triangular decomposition methods,
 cannot in general 
produce an  ideal preserving decomposition: for example a mere
primary ideal in dimension zero is not triangular in general
(think of 
the primary ideal  $\l y^2 + 3 x , x y + 2 x , x^2\r$
of associated prime $\l x,y\r$;
it is a  reduced lexGb for $x \prec y$ and does not generate a  radical ideal).
However, the underlying key algorithmic concepts
may still be relevant, and the present work
shows it is the case, yet in a particular situation.

The algorithms proposed in this work
consider lexGbs instead of triangular sets.
Although it also builds upon the Euclidean algorithm
and hence can be compared to triangular decomposition
algorithms,
it goes only half-way when decomposing. A {\em splitting}
follows from those of polynomials in $x$ only, by generalizing dynamic evaluation.
In particular, no decomposition is produced in the first part, when $T=p^e$
is the power of an irreducible polynomial.
One can  expect  to decompose along the $y$-variable too, but this requires
dynamic evaluation in two variables.
Nonetheless, the methods introduced here  pave the way toward such decompositions.
Let us illustrate this by a toy example.

 \begin{ex}[Computing a lexGb]\label{ex:trois}
From now on p.r.s stands for (subresultant) pseudo-remainder sequence.
Below it is computed modulo $t_1$ (Eq.~\eqref{eq:deux}) or modulo $t_1'$ (Eq.~\eqref{eq:2bis}), $\S_i(a,\ b)$ refers to the $i$-th subresultant of $a \bmod t_1$ and $b \bmod t_1$.
See Definition~\ref{def:sub} for details.
{\small \begin{equation}\label{eq:deux}
\hspace{-8pt} \left\{
\begin{array}{l}
a:=(y+x) y (y+1+x)(y-1) \\
b:= (y+x) (y+1-x)\\
t_1 := x^2
\end{array}
\right.
\left\|
\begin{array}{l}
\underline{{\rm prs}(a,\ b) \bmod t_1}\\
    a \quad \rightarrow\quad b\quad \rightarrow\quad  \S_1(a,\ b)=- 4 x y \quad \rightarrow\quad  \S_0(a,\ b)=0
\end{array}
\right.
\end{equation}}

The last non-zero subresultant  $\S_1(a,\ b)=-4 x y$ is nilpotent modulo $t_1$.
We factor out $-4 x$ (it is obvious here, in general realized through
a removal of nilpotent coefficients, and Weierstrass factorization),
and  divide $- 4 x y$ by $-4 x$ and $t_1$ by $\frac{-4x}{-4}=x$
in order to make them monic.
We obtain now a monic polynomial $c:=y=\frac{-  4 x y}{-4 x}$ modulo $t_1':=\frac{t_1}x = x$.
We restart  then a subresultant p.r.s
of $(b \bmod t_1'(x))$ and $(c \bmod t_1'(x))$
modulo $t_1'(x)=x$,
while keeping record of the division by $x$.
{\small \begin{equation}\label{eq:2bis}
\hspace{-8pt} \left\{
\begin{array}{l}
\bar{b}:=b \bmod t_1'  = y  (y+1) \\
\bar{c}:=c \bmod t_1'  = y \\
t_1'(x) := x
\end{array}
\right.
\left\|
\begin{array}{l}
\underline{{\rm prs}(\bar{b},\ \bar{c}) \bmod t_1'}\\
  y^2 + y \quad \to
  \quad     \S_1(\bar{b},\ \bar{c})=y \quad \to\quad   \S_0(\bar{b},\ \bar{c})=0.
\end{array}
\right.
\end{equation}
}

The last non-zero subresultant $\S_1(\bar{b},\ \bar{c})=y$ is a gcd of $\bar{b}$ and $\bar{c}$ (modulo $t'_1(x)=x$) 
so that $\l \bar{b},\ \bar{c},\ x\r = \l y, \ x\r$.
Note that $[\S_1(\bar{b},\ \bar{c}),\ x] = [y,\ x]$ is a lexGb of $\l x, \ \bar{c},\ \bar{b}\r$.
A minimal lexGb of $\l a,\ b,\ x^2\r$ is then obtained by multiplying by $x$ (of which we have kept record)
the lexGb $[y,\ x]$
and concatenating $b$ (this is Line~\ref{SToGB:ret4} of Algorithm~\ref{algo:SToGB}).
$$
\left[x \cdot x,\  x \cdot y \right] \cat \left[b\right] =  \left[ x^2,\ x y\ , \ y^2 + (3x +1 ) y + x \right].
$$
This lexGb of $\l a,\ b,\ x^2\r$ is  minimal but not reduced. 
If necessary, it suffices to compute
adequate normal forms to obtain the reduced one.
We will consider lexGbs as output by the algorithms,
hence minimal Gr\"obner bases but not necessarily reduced.
\hfill \qed
\end{ex}

As this example shows, several new steps come into play to handle the computation
of the subresultant p.r.s. modulo $T=p^e=x^2$: removal of nilpotent (Algorithm~\ref{algo:WF}) and
Weierstrass factorization (Algorithm~\ref{algo:MQ}).
These two algorithms transform a non monic polynomial modulo $p^e$
to an equivalent monic one (Algorithm~\ref{algo:MF}). In this way, pseudo divisions can be carried through
to retrieve the ``first nilpotent'' and the ``last non nilpotent'' polynomials in a modified subresultant p.r.s
(Algorithm~\ref{algo:LNL}).
It suffices then to iterate the above process to compute a minimal lexGb (Algorithm~\ref{algo:SToGB}).

\paragraph{Generalizing dynamic evaluation}
In a second part of the article, $T$ can be any monic polynomial.
The most interesting cases are when a (monic) large degree factor of $T$
is also a factor of the resultant of $a$ and $b$. 
If $T$ has no common factor or only a small common factor,
the degree of the ideal $\l a,\ b,\ T\r$ is small compared to the degrees of the input polynomials
and 
while the algorithms proposed work in this case, they are slower than Gr\"obner bases.

This however covers most important situations, for example when
the input are two polynomials $a,\ b$ without a modulus $T$.
Then we take $T=\Res_y(a,b)$ the resultant itself.
In the experimental section~\ref{sec:exp} (Columns 4 \&  7 of the tables),
timings show indeed that computing the resultant (Column~3) and applying the algorithms of this work (Column~4)
reveals  faster or simply competitive with computing a lexGb  of $\l a,\ b\r$ (Column~5).
Computing the squarefree decomposition of the resultant (Column~6)
and applying the algorithms of this work is always much faster.
 
When $T$ is not the power of an irreducible polynomial,
Weierstrass factorization does not apply in general and making polynomials monic in order to restart
subresultant computations becomes impossible ``globally'' (that is modulo $T$, while modulo primary factors  $p^e$,
it does).
Applying the ``local/global'' philosophy of classic dynamic evaluation fails here too, the polynomial $T$ being not
squarefree. We show that it is possible to extend it to a general $T$
with splittings of type ``invertible/nilpotent'', instead of the standard ``invertible/zero''.
The algorithms developed in the first part to treat the local case are then rewritten under this new dynamic evaluation
paradigm. The output are {\em families} of lexGbs, deduced from the splittings
that occur when attempting a division by a non-invertible element.

\begin{ex}\label{ex:newD5} (See also Algorithm~\ref{algo:invNil} ``$\invNil$'', Example~\ref{ex:invNil})
Consider a non-zero polynomial $f\in k[x]/\l T\r$.
If $T$ is irreducible, $f$ is invertible, being non-zero. The Extended Euclidean algorithm permits to find it.
If $T$ is squarefree not irreducible and $\gcd(f,\ T)=g$ then $f$ is invertible modulo $\frac{T}{g}$ and zero modulo $g$
(the ``invertible/zero'' dichotomy).
This leads to an isomorphism $k[x]/\l T\r \simeq k[x]/\l T/g\r \times k[x]/\l g\r $,
as in classic dynamic evaluation.

If $T$ is not squarefree like $T=x^3(x+1)^2$, then $f$ may still not be invertible modulo $\frac{T}{\gcd(f,\ T)}$.
Take for example $f=x^2(x+2)$. Then $g_0=\gcd(T,\ f)=x^2$ and $\frac{T}{g_0}=x (x+1)^2$.
$f$ is not invertible modulo $\frac{T}{g_0}$. Consider next $g_1=\gcd(\frac{T}{g_0},\ g_0)=\gcd(x (x+1)^2, \ x^2)=x$.
This time $f$ is invertible modulo $\frac{T}{g_1 g_0} =(x+1)^2$,
and nilpotent modulo $g := g_0 g_1 = x^3$
(the ``invertible/nilpotent'' dichotomy). Moreover, $k[x]/\l T\r \simeq k[x]/\l g\r \times k[x]/\l T/g\r$.
In the invertible branch, we  can pursue computations (typically invert a coefficient
of a polynomial in $(k[x])[y]$), and in the nilpotent branch, work as introduced in the first part when $T=p^e$.
\hfill $\Box$
\end{ex}



\paragraph{Main results}
In the first part, Algorithm~\ref{algo:SToGB} ``$\SToGB$'' and Theorem~\ref{th:SToGB}
show how to compute a minimal lexGb of an ideal $\l a,\ b,\ T\r$,
where $a,\ b\in k[x,y]$ and $T=p^e$ is the power of an irreducible polynomial $p\in k[x]$.

In the second part, the input are polynomials $a,\ b\in k[x, y]$ and a monic non-constant polynomial $T\in k[x]$.
They verify the assumption
\begin{equation}\label{tag:H}
\text{for any primary factor } p^e\text{ of } T,
\quad
a{\rm~and~} b {\rm~are~ not~nilpotent~modulo~} p^e.\tag{H}
\end{equation}
Algorithm~\ref{algo:SToGBD5} ``$\SToGBD5$''  and Theorem~\ref{th:SToGBD5}
computes, on input $a,\ b,\ T$, a family of minimal lexGbs $(\calG_i)_i$
such that $\prod_i \l \calG_i\r = \l a,\ b,\ T\r$. The product is a direct one.
More precisely, $\l \calG_i  \cap k[x] \r + \l \calG_j \cap k[x]\r =\l 1\r$ for $i\not=j$.

The outcome translates to  faster computations of (a direct product of)
lexGbs of an ideal $\l a,\ b,\ T\r$
than the  Gr\"obner engine of Magma~\cite{Magma}, one of the  fastest available,
especially when $T$ is a factor of the resultant of
$a$ and $b$ having multiplicities.
\medskip

There is no difference if in Algorithm~\ref{algo:SToGB} ``$\SToGB$'' $p$ is a prime number instead of an irreducible polynomial,
and if  $a, b\in \Z[y]$.
The  obtained family  of lexGbs is made of  {\em strong} Gr\"obner bases of the ideal $\l a,\ b,\ p^e\r \subset  \Z[y]$.
Indeed, Hensel lifting, Weierstrass factorization, Lazard's structural theorem~\cite[Section~4.6]{AL94}
all hold in this context too.

In Algorithm~\ref{algo:SToGBD5} it is similarly possible to replace $T$ by an integer $n$ and to consider polynomials
$a,\ b\in \Z[y]$. The output is then a family of coprime lexGbs whose (direct)
product equals the ideal $\l a,\ b,\ n\r\subset \Z[y]$.
These lexGbs are {\em strong} Gr\"obner bases.

\paragraph{Motivation}
There exists a variety of methods to represent the {\em solutions} of a system of polynomials.
The most widespread are probably  Gr\"obner bases,
triangular-decomposition methods, primitive element representations (among
which the RUR~\cite{Ro99} and the Kronecker representation~\cite{GiLeSa01} have received the most attention),
homotopy continuation method~\cite{PHCpack,Bates2013}, multi-resultant/eigenvector methods~\cite{Auz88,Emiris02}. However, when it comes to representing faithfully the ideal generated by the input polynomial,
it remains essentially the Gr\"obner bases only. Some of the aforementioned methods
have the ability to represent a multiplicity of a solution, which is just a number,
however this is nowhere near to be ideal preserving.
This work is somehow affiliated to the triangular-decomposition method,
having the Euclidean algorithm at its core.
It thus constitutes  a first incursion to  ideal preserving methods based on the Euclidean algorithm.

Primary decomposition constitutes another clear motivation.
Being factorization free, this work nor its generalizations can compute
directly a primary decomposition. Rather, it yields a decomposition at a cheap cost,
so that  a primary decomposition algorithm can then be run on each component. 
Finding a decomposition efficiently is indeed a well-known speed-up in the realm
of  primary
decomposition~\cite[Remark~2]{decker1999primary}.
For example, the {\tt PrimaryDecomposition} command of Magma tries to compute a
triangular decomposition
from a Gr\"obner basis (apparently following~\cite{La92})
in order to reduce the cost of internal subsequent routines such as factorization.
It is not known how to do this for non-radical ideals.
In addition, the lexicographic monomial order plays a crucial role in GTZ-like
primary decomposition algorithms~\cite{GTZ88}.
As mentioned above, in this work decompositions follow only from polynomials in $x$,
so comparisons would be premature. We would rather wait that decompositions following the $y$ variables
are developed, which essentially amounts to dynamic evaluation in two variables.

\paragraph{Direct product of lexGbs} While our algorithms naturally compare to Gr\"obner bases' ones,
the output differs in that it is a {\em family} of lexGbs.
However, this is not a drawback. Indeed it carries more information, like an intermediate representation toward the primary decomposition.
Normal forms, for example to test ideal membership,  can be done componentwise.
It is also possible to perform other standard ideal operations.
Besides, when the coefficients are rational numbers,
their size are smaller: this stems from the fact that one lexGb can be reconstructed from
the family output by our algorithm,
through the Chinese remainder theorem, which induces a growth in the size of the coefficients.
See~\cite{DaSc04} for details.
Additionally, decomposing a lexGb for solving has been known to be advantageous since
a long time~\cite{La92}.

\subsection{Previous and related work}
\paragraph{On the general method}
There exists a quite dense literature about representations of the {\em solutions}
of a system of polynomials $a,\ b\in k[x,y]$, whether they are {\em simple}
or not. 

One line of research builds upon a subresultant sequence.
It started with~\cite{gonzalez1996improved} by Gonz\`alez-Vega and El Kahoui,
with several forthcoming articles improving the idea. In the background of these works lies
the topology of plane (real) curves defined by polynomials over the rationals.
The analysis of the bit-complexity is thus a central aspect, and the forthcoming works aim at
improving it \cite{diochnos2009asymptotic,BoLaMoPoRoSa16,LaPoRo17,li2011modpn}.
From a subresultant sequence of $a$ and $b$,
it computes a {\em triangular decomposition} of the common set of solutions of $a$ and $b$.
A key step is to take the squarefree part of the resultant. It
thus does not consider representation of multiplicities.

Another direction considers  {\em shearing} of coordinates, for representations
in term of primitive element (RUR)~\cite{BoLaMoPoRoSa16,MeSc16}. It should be noted
that primitive element representations cannot be ideal preserving~\cite[Remark~3.1]{Ro99}.

\medskip

As for general triangular decomposition methods, some have studied the representation
of multiplicities. In most references below the multiplicity is equal to the dimension
of a certain local algebra.
In the bivariate case, \cite{Cheng2014}
proposes algebraic cycles and primitive p.r.s. to find the multiplicity of each primary components.
A preliminary study of deformation techniques is found in~\cite{Li2003multiplicities} that computes  a certain multiplicity number.
The {\em intersection multiplicity} through Fulton's algorithm was investigated in
the bivariate case in~\cite{Marcus2012-short}. The complexity of this approach remains unclear.
It was extended in~\cite{Alvandi2015-short} to multivariate system by reducing to the bivariate ones, through a method
dealing with the tangent cone.

The present work distinguishes from the above in that it faithfully produces an ideal preserving representation,
which contains far more information than a multiplicity number,
yet relying ultimately on the classic routine: the subresultant.
To compare with other formal techniques supplied with the ideal preserving feature,
I am only aware of Gr\"obner bases computations.

\paragraph{About the subresultant algorithm}
Algorithm~\ref{algo:LNL} ``$\LNN$'' presents a subresultant algorithm computed modulo the power
of an irreducible polynomial $T=p^e$.  If the input is written $a,\ b,\ T$ then
the output written $u, \ v,\ T_1,\ U$ with $ U T_1=T$, and $u, \ v\in k[x,y]$ monic (in $y$)
polynomials, satisfy  $\l a,\ b,\ T\r = \l u, \ v T_1,\ T\r$. 
$u$ is a monic gcd of $a$ and $b$ modulo $T$ if there exists one, in which case $v=0$
(such a gcd does not necessarily  exist see~\cite{Da17}).
Otherwise, $v\not=0$  and this algorithm provides a weak version of the gcd, with the ideal preserving feature.
The terminology $\LNN$ reflects the fact that $u$ is the last non-nilpotent polynomial ``made monic'' in the modified p.r.s,
while $v$ if the first nilpotent one ``made monic'' too (we assume that $0$ is a nilpotent element in this article).

Moroz and Schost~\cite{MoSc16} propose to compute the resultant modulo $x^e$, that is when $p(x)=x$
with our notation. It shares some similar tools,  Weierstrass factorization and Hensel lifting.
However, their purpose is to write a quasi-linear complexity estimate,
rather than producing a practical and neat algorithm. The idea amounts to adapt the half-gcd algorithm that runs
asymptotically in quasi-linear time,
to this particular context --- whereas here it is the standard quadratic time subresultant algorithm that is used.
The possibility to compute subresultants along with the resultant, required by Algorithm~\ref{algo:LNL}, is left open in~\cite{MoSc16}.
The present article aims rather at  practical algorithms and implementations.
The algorithms presented here being new, they must demonstrate their practical efficiency,
and be as simple as possible to  check their correctness and to be implemented.
Despite this, the description is already quite technical, and the possibility
to incorporate faster routines such as the half-gcd for computing subresultants
appears 
to be an interesting challenge for future work.
The algorithms described are all implemented ``as such''  in Magma~\cite{Magma} (see \url{http://xdahan.sakura.ne.jp/gb.html}),
and in many cases outperform the Gr\"obner engine of Magma, yet equipped with  one
of the best implementation of the F4 and FGLM algorithms.

Recently, faster algorithms for the computation of the resultant of $a,b$ have come
out~\cite{Vi-resultant-2018,VdH-resultant2020}.
It does so by bypassing the computation of a p.r.s sequence --- whereas here the p.r.s. is essential.
These outcomes imply directly an algorithm that computes a lexGb of the ideal $\l a,\ b\r$
within the same complexity. However, they require  strong generic assumptions: the lexGb
must consist  of two polynomials, and the  unique monic (in $y$)
polynomial shall be of degree one. These assumptions lie
far away of the present work that deal with lexGbs containing many polynomials
of arbitrary degree. 

\medskip

The algorithms of this article adapt straightforwardly to  univariate polynomials $a$ and $b$ that have coefficients in $\Z$.
When $T=p^e$ with $p$ a prime number, one can think  of $p$-adic polynomials, $e$ being the precision.
A related work~\cite{caruso2017numerical} in this situation  studies the stability of the subresultant
algorithm to compute a ($p$-adic) approximate gcd. It should be mentioned that this gcd is not monic, nor
have an invertible leading coefficient, hence is of limited practical interest.
The subresultant algorithm therein does not hold the crucial functionalities
(Hensel lifting, Weierstrass factorization) that make polynomials monic.

\paragraph{About dynamic evaluation}
The second part of this article involves dynamic evaluation.
In its broad meaning, this technique automatically produces case distinctions depending
on the values of some parameters in an equation.
More precisely here, the two algebraic equations are the polynomials $a$ and $b$ in the variable $y$,
the parameter being $x$.
The case distinctions comes from  the image of their coefficients (in $x$)
modulo each primary factor of $T$.
In this context, classic dynamic evaluation considers $T$ squarefree;
A primary factor is then of the form $x-\alpha$ (at least in the algebraic closure)
for a root $\alpha$ of $T$: this corresponds to  the evaluation of the parameter at the value $\alpha$.
This work
allows {\em Taylor expansions} at $\alpha$ at order $e-1$ if the primary
factor of $T$ is of degree $e$ namely is $(x-\alpha)^e$.
The case distinction still follows from the different roots of $T$,
thus the different Taylor expansions considered should be at different expansion points.

The computational paradigm of dynamic evaluation, introduced in~\cite{D5} is quite simple.
A  {\em squarefree} polynomial $T$ is split by a gcd computation
when attempting to perform an inversion modulo $T$, see Example~\ref{ex:newD5}.
It is not surprising that such splittings  actually appeared  before~\cite{D5} in the realm
of integer factorization, as  Pollard did with his rho method~\cite{pollard1974}.
In the algebraic context, dynamic evaluation, also coined ``D5 principle''~\cite{D5},
has been quite thoroughly studied, especially toward
manipulation and representation of algebraic numbers~\cite{Du94},
generalization to multivariate polynomials through the triangular representation~\cite{MMMOo95}, becoming 
a major component of some triangular decomposition algorithms~\cite{LeMMMXi05}. 
Another standard reference is~\cite{Mo03}.
Concerning complexity results, we refer to~\cite{DaMMMScXi06} that brings fast univariate arithmetic
to this multivariate context, or more recently to~\cite{van2020directed}.
The splitting can also be performed in an ideal-theoretic way through Gr\"obner bases computations, see~\cite{Noro06}, but at
a probably  too expensive cost to be competitive with what gcd computations afford.
Except this latter article,
all these  previous works manipulate squarefree polynomials only, the present work being to the best of my knowledge, the first addressing general univariate polynomials.

\subsection{Organization}
The first part (Sections~\ref{sec:prem}-\ref{sec:lF}) considers $T$ equal to the power of an irreducible polynomial.
Section~\ref{sec:prem} is made of preliminaries: Lazard's structural theorem,
Weierstrass factorization through Hensel lifting, subresultant p.r.s are recalled.
The following  section~\ref{sec:lF}
introduces the main subresultant routine ``\LNN'',
and shows how to deduce a minimal (not reduced in general)
lexGb of $\l a,\ b ,\ p^e\r$.
The second part considers $T$ general and is treated in Sections~\ref{sec:D5}-\ref{sec:lexGbD5}.
This consists essentially in translating the algorithms presented in the first part 
into a dynamic evaluation that works when $T$ is not necessarily squarefree.
Section~\ref{sec:D5} introduces the dichotomy ``invertible/nilpotent''
and applies it to the algorithms that  make polynomials monic.
It is followed in Section~\ref{sec:lexGbD5} by the generalization
of the algorithm ``$\LNN$'' that computes the ``monic form'' of the last nilpotent/first nilpotent
polynomials in the modified subresultant p.r.s of $a$ and $b$, and by the computation of families of lexGbs
whose (direct) product is $\l a,\ b,\ T\r$.
An experimental  section~\ref{sec:exp}  shows the benefits of an implementation realized in Magma
of this approach, compared to Gr\"obner basis computations. 
A special emphasize is put on the case of an input $a$, $b$
without a modulus $T$ in Section~\ref{sec:compl} due to the importance
of this case and the questions it raises. 
The last section~\ref{sec:concl}
describes some improvements for future work and  various generalizations.

As one can see, the second part encompasses the algorithms proposed in the first part,
restricted to when $T=p^e$. Is this first part then really necessary ?
Well, rewriting {\em known} algorithms to the {\em classic} dynamic evaluation paradigm is one thing,
rewriting  {\em new} algorithms in a {\em new} dynamic evaluation is another one.
Without the first part, the algorithms of the second
part would become  too obscure.
Besides, some of these algorithms  rely on those restricted to the ``local case'' $T=p^e$ dealt with in the first part,
in order to alleviate the proofs.

\paragraph{Notations and convention} Below are gathered some notations used in this paper:
\begin{itemize}[-]
  \setlength\itemsep{-2pt}
\item $\sqfp(f)$: squarefree part of a polynomial $f\in k[x]$.
\item a polynomial $p$ {\em strictly} divides another polynomial $q$ if $p|q$ and $p\not=q$.
\item $\calP$ denotes the set of irreducible polynomials in $k[x]$. Given $p
\in \calP$ and a polynomial $T\in k[x]$, $v_p(T)\in \N$  is the largest integer such that $p^{v_p(T)} | T$.
\item an ideal generated by a family of polynomials $\calL$ is denoted $\l \calL\r$.
\item $k$ will denote a perfect field: irreducible polynomials are squarefree in any algebraic extension of $k$.
\item a product of ideals is direct if the ideals are pairwise coprime. Unless otherwise specified, all
  products of ideals will be direct.
\item all lexicographic Gr\"obner bases are monic, meaning that their leading coefficient is $1$.
\item although zero is usually not considered a nilpotent element,
  it is convenient to assume it is one.
\item a  value  output by an algorithm is ignored by writing ``$\_\_$'' (underscore) instead of that value.
\item algorithms are written as pseudo-codes. There is no obvious shortcut that would allow
  to simplify them furthermore.  Most routines are indeed elementary, and 
it essentially amounts to recursive calls,
and case distinctions.
The proof of correctness follows the lines of the algorithms.
  They are all self-contained though, copiously commented which should allow to read them
  without too much effort.
\end{itemize}

\section{Preliminaries\label{sec:prem}}
We consider a perfect field $k$, 
and a monic non-constant polynomial $T\in k[x]$.
Write $R=k[x]/\l T\r$ the quotient ring.
The other input are two polynomials $a$, $b$ in $k[x,y]$ reduced modulo $T$.
The lexicographic monomial order with $x \prec y$ is put on the monomials of $k[x,y]$.
In the first part Sections~\ref{sec:prem}-\ref{sec:lF}
$T$ is a power of an irreducible polynomial $p$,
$T=p^e$.

\subsection{Nilpotents}

\begin{Def}
  Let $T$ be a non-constant monic polynomial in $k[x]$. An element $r\in R = k[x]/\l T\r$
  is invertible if and only if $\gcd(T,r)=1$.
  It is nilpotent if all irreducible factors of $T$ divide $r$. Equivalently,
if the squarefree part $\sqfp(T)$ of $T$ divides $r$.
\end{Def}

The following result is classical and elementary:

\medskip

\noindent {\bf Proposition.}
{\em A polynomial $a\in R[y]$ is {\em nilpotent} if 
and only if all its coefficients are  nilpotent in $R$, or equivalently,
if $\sqfp(T)$ divides $a$.}

\begin{ex}[nilpotent element]
Let $r=x^2 (x+1)$. Then $r$ is  nilpotent modulo $T=x^4$, modulo $T=(x+1)^4$, 
and modulo $T=x^3 (x+1)^2$,  but not modulo $T=x^2 (x-1)^2$. \hfill \qed 
\end{ex}

\begin{ex}[nilpotent polynomial]\label{ex:nilcont}
Let $T=x^2(x+1)^2$ and $P=x(x+1) y + 2 x (x+1) (x-1)$. Then $P$ is nilpotent since its coefficients
$x(x+1)$ and $2 x (x+1)(x-1)$ are all nilpotents. \hfill \qed  
\end{ex}

\subsection{About lexicographic Gr\"obner bases}
The lexicographic monomial order $\prec$
with $x\prec y$ is a total order on the monomials of $k[x,y]$
defined as $x^a y^b \prec x^c y^d$ if
$b<d$, or $b=d$ and $a<c$. Given $f\in k[x,y]$ a non-zero polynomial, the notation $\LM(f)$ will refer
to the leading monomial for $\prec$ among the monomials occuring in $f$.
A lexGb of a polynomial system $(f_1,\ldots,f_m)$ generating an ideal $I$
is any family of polynomials $(g_1,\ldots,g_s)\subset I$ such that 
for any $f\in I$ there is a $g_j$ such that $\LM(g_j)|\LM(f)$. If $k$ is simply an Euclidean ring,
variations exist, and the one corresponding to this definition is called a {\em strong Gr\"obner basis}.
The Gr\"obner basis is {\em minimal} if for all $1\le i \le s$, $\LM(g_i)$ is not divided by $\LM(g_j)$ for any $j\not=i$.

In 1985, Lazard~\cite{Laz85}  completely characterized lexGbs in $k[x,y]$,
result now known as ``Lazard's structural theorem''.
It will be used in the following form:
\begin{Theo}\label{th:Laz}
 \ {\upshape (A)}\  Let monic non-constant  polynomials   $h_1,\ h_2 ,\ldots,\  h_{\ell-1}\in k[x]$ 
such that $h_i$ divides $h_{i-1}$ for $i=2,\ldots,\ \ell-1$. 
Let also  monic (in $y$) polynomials   $g_2,\ldots,\ g_{\ell}\in (k[x])[y]$
  of {\em increasing degree}: $\deg_y(g_2)<\deg_y(g_3)<\cdots < \deg_y(g_\ell)$.
  The list of polynomials 
$$
\calL = [h_1,\ h_2 g_2,\ \ldots,\ h_{\ell-1} g_{\ell-1}  ,\  g_\ell]
$$
is a lexGb if and only if:
\begin{equation}\label{eq:Laz}
g_i  \in \l g_{i-1} , \frac{h_{i-2}}{h_{i-1}} g_{i-2} ,\ \ldots, \frac{h_1}{h_{i-1}}\r,\ \text{for}~3 \le i\le \ell
 \end{equation}
It is moreover {\em minimal} if and only if $h_i$ {\em strictly} divides $h_{i-1}$ for $i=2,\ldots, \ell-1$.

{\upshape (B)} Assume now that $\calL$ above is  already a  lexGb.
Given $h_0\in k[x]$ a monic non-constant polynomial,
and $g_{\ell+1}\in (k[x])[y]$ monic (in $y$) of  degree  larger than that of all the 
polynomials $g_1,\ldots, g_\ell$,
then
  $$\calL' := [\ h_0\, f\ : \ f\in \calL\ ] \ \cat \ [\ g_{\ell+1}\ ]$$
is a  lexGb if and only if $g_{\ell+1} \in \l \calL\r$. 
It is moreover minimal if and only if $\calL$ is minimal.
\end{Theo}
\begin{ex}\label{ex:Laz}
  The system of polynomials $(x^2, \ x y,\ y^2 + x)$ is a minimal
lexGb. Indeed, with the notation of the theorem we have $h_1=x^2,\ h_2 = x$ and $f_2 =y,\ f_3=y^2+x$.
Moreover $h_2$ strictly divides $h_1$.
  The condition  $f_3 \in \l \frac{h_1}{h_2},\ f_2\r$ is the only one that we must verify. It translates to the condition $y^2  + x \in \l x,\ y\r$ which is clearly true.
  
  However, the system $(x^2, \ x y,\ y^2 + 1)$ is not a lexGb. Indeed, $y^2+1 \notin \l x,\ y\r$. \hfill $\Box$
 \end{ex}
\begin{proof}
  Part~(A) is a restatement\footnote{
    with the notations $k,\ f_i,\ H_i, \ G_i$s of Lazard's article:
    $k=\ell-1$,
$f_0 = h_1,\ f_1 = h_2 g_2, \ \ldots \ , \ f_{k-1} = h_{\ell-1} g_{\ell-1} , \ f_k = g_\ell$, and 
$H_i = g_{i+1}$ and $G_i = \frac{h_{i+1}}{h_{i+2}}$.}
  of Lazard's theorem~\cite[Thm 1]{Laz85} in our particular context.

  Part~(B) is straightforward when we translate the conditions of Eq.~\eqref{eq:Laz} made on the list $\calL$
  to the list $\calL' =  [\ h_0\, f\ : \ f\in \calL\ ] \ \cat \ [\ g_{\ell+1}\ ] $.
  Indeed, all these conditions, except the one mentioned that is ($g_{\ell+1} \in \l \calL\r$), are verified since $\calL$ is assumed to be a  lexGb. If $\calL$ is minimal, then the condition for minimality is guaranteed since $h_0$ is assumed to
be non-constant.
  \end{proof}

\subsection{Making polynomials monic}
\label{sec:Weier}
In order to cope with subresultants which do not have an invertible leading  coefficient modulo $T$,
a routine that transforms them into equivalent monic polynomials is presented here.
If the polynomial is not nilpotent, Weierstrass factorization allows to realize this.
Else, we need first to ``remove'' the nilpotent part.
When $T=p^e$,
it suffices to factor out the largest power of $p$ that divides it.
When $T$ is general however, a polynomial 
which is not nilpotent may not necessarily have a coefficient that is invertible.
We need then dynamic evaluation as introduced in Section~\ref{sec:D5}.
The main algorithm~\ref{algo:MF} ``$\MF$''  is a wrapping  of two subroutines, Algorithm~\ref{algo:WF} ``$\WF$''
and Algorithm~\ref{algo:MQ} ``$\MQ$''. The first one puts the polynomial that we want to ``make'' monic
into a form of application of the second algorithm. This second one realizes Weierstrass factorization
through Hensel lifting. 
In the subsequent algorithms, only the ``$\MF$'' wrapping algorithm will appear
(and no more  ``$\WF$'' nor ``$\MQ$'').
  
\paragraph{Overview of Algorithm~\ref{algo:WF} ``$\WF$''} We assume in this section and the next one
that $T=p^e$ is the power of an irreducible polynomial.
The algorithm ``$\WF$'' has two effects: find the largest power of $p$ that divides the input polynomial (case of return at Line~\ref{WF:ret2}),
or, if there is none, find the coefficient of highest degree that is not divided by $p$ (case of a return at Line~\ref{WF:ret1}); This coefficient is then invertible modulo $T$.
The algorithm scans the coefficients by decreasing degree and updates the computation with a gcd (Line~\ref{WF:gcd}).
If the coefficient $c_i(x)$ (of $y^i$) is invertible modulo $T$, the algorithm also outputs
the largest power of $p$ that divides all coefficients of degree higher than $i$, 
as well as the inverse of $c_i$ modulo $T$. These output are indeed required to perform Weierstrass factorization through Hensel
lifting (Algorithm~\ref{algo:MQ} ``$\MQ$'').

It may thus be necessary to call 
Algorithm~\ref{algo:WF} twice: one to remove the nilpotent part, and one
to find the largest coefficient that has become invertible.
Then, Weierstrass factorization Algorithm~\ref{algo:MQ}  can be called.
Let us see this through an example:
\begin{ex}[Algorithm~\ref{algo:WF} ``$\WF$'']\label{ex:nilFac}
Let $f = 3 x^2  y^2 + (x^2  + 2 x) y + x$ and $p(x)=x$, $T=p^3=x^3$. 
This polynomial is nilpotent, and to ``remove'' the nilpotent
part, running Algorithm~\ref{algo:WF} ``$\WF$'' gives (see the specifications):
\begin{equation}\label{eq:NF-ex}
 (\, x,\ \ - 1,\ \ \_\_ \,) \ \leftarrow \ \WF(f, \ T).
\end{equation}
Indeed, $x\mid f$ but $x^2\nmid f$.
The second output is $-1$ meaning that there is no coefficient of $f$ that is invertible modulo $x$  (since
$f$ is divided by $x$). The third output has no importance in this case, and according to the conventions
is replaced by an underscore ``{\_\_}''.

\noindent Then $f/x$ is not nilpotent, so that running $\WF$ again  will output:
\begin{equation}\label{eq:WM-ex}
 (\, x,\  \  1, \ \ -\frac 1 4 x + \frac 1 2\,) \leftarrow \WF(f/x,\ x^{3-1}).
\end{equation}
The second output $1$ indicates the degree one in $f/x$. Indeed, the coefficient of $y$ in $f/x$
is $\frac{x^2+2x}x = x+2$, which is indeed invertible modulo $x^3/x=x^2$. The inverse modulo $x^2$ of $x+2$ is then 
 $-\frac 1 4 x + \frac 1 2$ which is the third output.

Note that the coefficient of $y^2$ in $f/x$ is not invertible hence is divided by $x$ (this coefficient being $3x$).
The first output returns the highest power of $x$ that divides this coefficient:   $x$ divides $\frac{3 x^2}x$.
\hfill \qed\end{ex}

\begin{algorithm}[H]
\DontPrintSemicolon
\KwIn{1. polynomial $f=c_0(x) + c_1(x) y + \dots+ c_\delta(x) y^\delta \in k[x,y]$,
  
  2. $T=p^e \in k[x]$ power of an irreducible polynomial $p$.
}
\KwOut{1. $g=p^\nu$ the largest power of $p$ that divides all the coefficients
$c_{d+1},\ldots, c_\delta$ (where $d$ is defined in 2. below).

  2.   the largest index $d$ such that the coefficient $c_d(x)$ is invertible modulo $T$.
  If there is none, then $d=-1$.
  
  3. $\alpha$ is the inverse of $c_d$ if $d \ge 0$,  and an arbitrary value if $d=-1$.
}

\BlankLine

$d\leftarrow \deg_y(f)$~~;\ \ \   $g_{new}\leftarrow T$\;
\While{$d \ge 0$ and $g_{new}\not=1$\nllabel{WF:while1}}{
  $g \leftarrow g_{new}$~;~~ $c_d\leftarrow \coeff(f,\ d)$\;

  $(\, g_{new},\ \_\_\, ,\ \alpha \, ) \leftarrow \xgcd(g,\ c_d)$
  \nllabel{WF:gcd}\tcp*[r]{$\xgcd \ =$ Extended gcd: $\_\_ g + \alpha c_d = g_{new}$,}

  $d \leftarrow d-1$ \tcp*[r]{the underscore \_\_ replaces an unimportant output}
}
\eIf(\tcp*[f]{case $f$ not nilpotent}){$g_{new} = 1$}{
 
  \Return{$(\, g,\ d+1,\ \alpha\, )$}\nllabel{WF:ret1}\;
}(\tcp*[f]{case $f$ nilpotent})
{
  \Return{$(\, g_{new},\ -1,\ 0\, )$}\nllabel{WF:ret2}\;
}
\caption{\label{algo:WF} \quad $(\, g,\ d,\ \alpha \,) \leftarrow  \WF(f,\ T)$}
\end{algorithm}

\begin{Rmk}\label{rmk:nilFac}
  $f$ is nilpotent if and only if $d=-1$ where $ (\, \_\_,\ d,\ \_\_ \,) \leftarrow \WF(f,\ T)$
  (again, underscores ``$\_\_$'' replace unimportant output).
\end{Rmk}

\begin{Lem}[Correctness of Algorithm~\ref{algo:WF} ``$\WF$'']\label{lem:WF}
The three output $(\, g_{new},\ d,\ \alpha \, )$ satisfy the specifications 1-2-3.
  \end{Lem}

\begin{proof}
 We must prove that the output at Line~\ref{WF:ret1} or~\ref{WF:ret2}
 verifies the specifications 1.-2.-3. of the output.

 {\em Case of return at Line~\ref{WF:ret1}.} Here $g_{new}=1$
 so that the while loop (Lines~\ref{WF:while1}-\ref{WF:gcd})
 exited on that condition and not on the condition $d <0$. Therefore $d \ge 0$.
 This means that
 $$
 g_{new}=\gcd(T,\ c_{d+1},\ c_{d+2}, \ldots, \ c_\delta) = 1,
 \quad \text{and} \quad g=\gcd(T,\ c_{d+2}, \ldots, \ c_\delta)\not=1.
 $$
 Moreover the extended gcd ``\xgcd'' computation at Line~\ref{WF:gcd}   
 gives $\alpha c_{d+1} \equiv 1 \bmod T$. We have proved that the output $(\, g,\ d+1,\ \alpha \, )$
 verifies the specifications 1.-2.-3.
 
 {\em Case of return at Line~\ref{WF:ret2}.}
 Here $g_{new}\not=1$, meaning that the while loop (Lines~\ref{WF:while1}-\ref{WF:gcd})
 exited on the condition $d =-1$. Thus $g_{new}=\gcd(T,\ c_{0},\ c_{1}, \ldots, \ c_\delta)\not=1$
 is the largest power of $p$ that divides all the coefficients of $f$, hence that divides $f$.
 This proves that the output $(\, g_{new},\ -1,\ 0\, )$ verifies the specifications 1.-2.-3.
  \end{proof}
\paragraph{Weierstrass factorization through Hensel lifting}
The Weierstrass preparation theorem  states that
a formal power series $f = \sum_i c_i y^i \in \fa [[y]]$ with coefficients $(c_i)_i$ in a
local complete ring $(\fa,\fm)$, not all of them lying in $\fm$,
factorizes uniquely as $f = q u$ where $q = q_0 + \cdots + q_{n-1} y^{n-1} + y^n$ is
monic and $q_i \in \fm$, and where $u \in \fa[[y]]^\star$ is an invertible power series.
In the traditional version, the degree $n$ of $q$ is equal to the {\em smallest} degree coefficient
of the series $f$ not in $\fm$.
If $f$ is a {\em polynomial}, it is easy to adapt proofs so that $n=\deg(q)$ is the degree
of the {\em highest} degree coefficient $c_n$ of $f$ {\em not} in $\fm$.
Indeed, it suffices to consider the reverse polynomials to switch back to the the smallest degree coefficient.
In our context the local complete ring is $ \fa = k[x]/\l p^e\r$, $\fm=\l p\r$
(indeed it is isomorphic to $k[[x]]/\l p^e\r$, which is a finite quotient of the local  complete ring
$k[[x]]$). 

\begin{Lem}[Weierstrass]\label{lem:WM}
Given  $f\in R[y]$, written $f = c_d(x) y^d + \cdots +c_0(x)$, with $c_n$ the coefficient
of highest degree not in $\l p\r$ (that such a coefficient exists is a pre-requirement),
there exist two polynomials $q$ and $u$ defined below 
such that:

\quad (1) $f= q \cdot u$.

\quad (2) $q$ is monic of degree $n$, and

\quad (3)  $u = u_0 + u_1 y + \cdots + u _{d-n}  y^{d-n}$ with $u_0\not \in \l p \r$ and $u_i \in \l p\r$ for $i\ge 1$. In particular
$u$ is a unit of $R[[y]]$, so that: \underline{$\l f,\ p^e\r = \l q, \ p^e \r$}.
\end{Lem}

 To compute the monic polynomial $q$ in practice, ``special'' Euclidean division~\cite{lang2002algebra},
linear algebra,
and Hensel lifting~\cite[Algo~Q]{musser1975} are available.
The latter is the most efficient and works in more general situations that
are required in the second part of this article.
This is what is used in~\cite[Thm.~1]{MoSc16} for their ``normalization'' (whose
purpose is to make polynomials monic too).
Since a proof (and more) can be found  in~\cite{musser1975}, it is not reproduced here.
That  proof  relies entirely on Hensel lifting and does
not involve Weierstrass preparation theorem, however Algorithm~(Q) given by  Musser is really 
a Weierstrass factorization when the modulus is the power of an irreducible polynomial.
This implementation has the advantage to extend straightforwardly 
when the ring of coefficients is not necessarily a local ring, but a direct product of thereof --- and that no division by zero occur.
This feature helps to  generalize the dynamic valuation as undertaken in Section~\ref{sec:D5},
in that we can rely again on the same algorithm~\ref{algo:MQ} when the modulus is not a power of an irreducible polynomial.

This algorithm essentially reduces to quadratic Hensel lifting (QHL). It lifts a factorization, and,
as always with QHL, a B\'ezout identity. A standard form of QHL is described in Algorithm~15.10 of~\cite{GaGe03}
and we refer to it for details. Given $N\in k[x]$ and $f,a,b,\alpha,\beta\in k[x,y]$,
the notation:
$$(\, a^\star,\ b^\star\, ) \leftarrow  \HL(f,\ a,\ b,\ \alpha,\  \beta,\  N \leadsto N^{2^\epsilon}), $$
assumes that:
$$
f\equiv a b \bmod N, \quad  \alpha a + \beta b \equiv 1 \bmod N,
\quad f\equiv a^\star b^\star \bmod N^{2^\epsilon},
\quad a^\star \equiv a \bmod N,
\quad b^\star \equiv b \bmod N.
$$

Here, the initial two factors of the input polynomial
$f$ are the coefficient $c_d \equiv \alpha^{-1} \bmod T$ of  $f$,
and $b \equiv \alpha f \bmod N$
(so that $b $ is a monic polynomial of degree $d$ modulo $N$). The lifting produces
the same equality modulo $N^{2^\epsilon}$. This integer $\epsilon$ is the smallest
integer for which $N^{2^\epsilon} $ divides $T$. That such  an integer exists follows from $\sqfp(N)=\sqfp(T)$.
The initial B\'ezout identity is given for free here: indeed $c_d \alpha  + b \cdot 0 = 1$.

\begin{algorithm}[H]
\DontPrintSemicolon
\KwIn{1. polynomial $f = c_0(x) + c_1(x) y + \dots+ c_\delta(x) y^\delta
  \in k[x,y]$, with $\deg_y(f)=\delta \ge d$.
  
  Its coefficient $c_d(x)$ of degree $d$  is invertible mod $T$ of inverse $\alpha$.\\
  Moreover, $c_{d+1},\ \ldots,\ c_\delta$ are nilpotent modulo $T$, that is $\sqfp(T)|\sqfp(c_i)$ for $i=d+1,\ \ldots, \delta$
  
  2. A monic non-constant polynomial $T\in k[x]$.
  
  3. $d$ a non-negative integer (defined in 1.)
  
  4. $\alpha\equiv c_d(x)^{-1}  \bmod T$
  
  5. a monic polynomial $N\in k[x]$ that divides $T$ and such that $\sqfp(N)=\sqfp(T)$,
  equal to $\gcd(T, \ c_{d+1},\ \ldots, \ c_\delta)$.
  So that if $N\not=1$ then $\deg_y(f \bmod N) = d$.
}
\KwOut{$b^\star\in k[x,y]$, monic (in $y$) of degree $d$,
  and verifying \underline{$\l b^\star,\ T\r=\l f,\ T\r$}.}

\BlankLine

$b \leftarrow  \alpha f \bmod N$ \tcp*[r]{$b$ is monic of degree $d$}

$c_d\leftarrow \coeff(f,\ d)$\;

Let $\epsilon$ be the smallest integer such that $T|N^{2^\epsilon}$\;

$(\, \_\_,\ b^\star\, ) \leftarrow \HL(f,\ c_d,\ b,\ \alpha, \ 0,\  N \leadsto N^{2^\epsilon}) $
\tcp*[r]{$f\equiv c_d b \bmod N,\quad  \alpha c_d + 0 \cdot b =1$}
\Return{$b^\star \bmod T$}\;
\caption{\label{algo:MQ} \quad $b \leftarrow \MQ(f,\ T,\ d,\ \alpha, \ N)$}
\end{algorithm}

\begin{ex}[Algorithm~\ref{algo:MQ} ``$\MQ$'']
  With input $f =  x(x+1) y^2 + (2 x + 1) y + x^2$ modulo $T = x^2 (x+1)^2$,
$d=1$   and $\alpha \equiv (2x+1)^{-1} \equiv -8 x^3 -12 x^2 - 2 x +1 \bmod T$,
and $N=x(x+1)$.
All input' specifications are satisfied: the coefficient $x(x+1)$ of degree 2 is nilpotent modulo $T$,
the coefficient $2 x+1$ of degree 1 is invertible. As expected, we have $d=1$ and $N=\gcd(T,\ x(x+1))=x(x+1)$;
As well as $\alpha\equiv (2 x+1)^{-1} \bmod T$.

We have $b \equiv f \alpha \equiv y+x \bmod N$. Now we lift
$f\equiv (2x+1)b \bmod N,\ N^2,\ N^4, \ldots$ until $T|N^{2^\ell}$ .
The B\'ezout identity $\alpha (2 x+1) + 0\cdot b =1$ that is also lifted is given for free.
In this example one step of lifting is sufficient since $N^2=T$, yielding $b^\star  =  y-x^3-2 x^2$. Then $\l b^\star,\ T\r = \l f,\ T\r$.
\hfill \qed \end{ex}

As already explained, Algorithm~\ref{algo:MF} ``$\MF$''  encapsulates the two algorithms ``$\WF$'' and ``$\MQ$''
described above in one algorithm.
On input $f,\ T$,  it outputs $U$ the largest polynomial that divides $f$, and $b$ {\em monic} which is ``equivalent''
to $f/U$ in the sense that $\l b,\ T/U\r = \l f/U,\ T/U\r$.

\begin{algorithm}[!h]
\DontPrintSemicolon
\KwIn{1. polynomial $f\in k[x,y]$
  
  2. $T=p^e\in k[x]$ the power of an irreducible polynomial $p$.
}
\KwOut{
  1. $b \in k[x,y]$  is monic (in $y$)
  
  2. Monic polynomial $U\in k[x]$ that divides $T$ and $f$, such that:
\hfill  \underline{$ \l b ,\ T/U\r = \l f/U,\ T/U \r$}.
}

\BlankLine

$(\, U,\ d,\ \alpha \,) \leftarrow \WF(f,\ T)$\nllabel{MF:WF1}\;

\eIf(\tcp*[f]{Case where $f$ is nilpotent}){$d=-1$}{ 
     $f' \leftarrow f/U$~; \ \    $ T' \leftarrow T/U$  \tcp*[r]{Now $f'$ is not nilpotent\ldots}
   
   $ (\, N, \ d,\ \alpha\, )  \leftarrow \WF(f',\ T')$\nllabel{MF:WF2}\tcp*[r]{\ldots so that $d\ge 0$}

  $b \leftarrow \MQ(f',\ T', \ d,\ \alpha, \ N)$\nllabel{MF:MQ1}\tcp*[r]{ $\l b,\ T'\r =\l f',\ T'\r$}
  
   \Return{$ (\, b,\ U\, )$}\nllabel{MF:ret1}\;

}(\tcp*[f]{Here $f$ was not nilpotent}){
  
$b \leftarrow \MQ(f,\ T, \ d,\ \alpha, \ U)$\nllabel{MF:MQ2}\tcp*[r]{$\l b,\ T\r =\l f,\ T\r$}

\Return{$(\, b,\ 1\, )$}\nllabel{MF:ret2}\;

}

\caption{\label{algo:MF} \quad $(\, b,\  U\, ) \leftarrow  \MF(f,\ T)$}
\end{algorithm}

\begin{Lem}[Correctness of Algorithm~\ref{algo:MF}]\label{lem:MF}
The output $(\, b,\ U\, )$ satisfies the equality of ideals $\l b,\ T/U\r = \l f/U,\ T/U\r$ with $b\in k[x,y]$
monic (in $y$) and $U\in k[x]$.
\end{Lem}
\begin{proof}
The call to $\WF(f,\ T)$ at Line~\ref{MF:WF1} outputs
$(\, U,\ d, \alpha\, )$. We must distinguish two cases:

{\em Case of exit at Line~\ref{MF:ret2}.}  Here $d\ge 0$. According to the output' specifications of ``$\WF$'',
$U=\gcd(T, \ c_{d+1},\ldots,\ c_\delta)$ divides $T$ and the coefficients $c_{d+1},\ldots, c_\delta$ of $f$. Additionally,
$c_d$ is invertible modulo $T$ of inverse $\alpha$.
Therefore the five entries $(\, f,\ T, \ d,\ \alpha,\ U\, )$ of $\MQ$ at Line~\ref{MF:MQ2}
satisfy the input' specifications of $\MQ$. 
Its output' specifications provide: $\l b,\ T\r = \l f,\ T\r$.
In conclusion, the output $(\, b,\ 1\, )$ of $\MF(f,\ T)$ satisfy the required specifications
with $U=1$.

{\em Case of exit at Line~\ref{MF:ret1}.}
Here $d=-1$, which means that, according to the output' specifications of ``$\WF$'',
$U=\gcd(T,\ c_0,\ \ldots,\ c_\delta)$ divides all coefficients of $f$, as well as $T$.
It follows also that $f'= f/U$ has an invertible coefficient modulo $T'=T/U$.
Therefore the  second call to  ``$\WF$'' at Line~\ref{MF:WF2} implies that the output
$(\, N,\ d,\ \alpha\, )$  satisfy: $c_d/U$ is invertible modulo $T'=T/U$ with $\alpha c_d /U \equiv 1 \bmod T'$.
Additionally, $N=\gcd(T',\ c_{d+1}/U,\ldots, c_\delta/U)$
divides all the coefficients $c_{d+1}/U,\ldots, c_\delta/U$ of $f'$.
Consequently, the five entries $(\, f',\ T',\ d,\ \alpha,\ N\, )$ satisfy the input' specifications
of $\MQ$ at Line~\ref{MF:MQ1}. The output $b$ hence verifies $\l b,\ T'\r = \l f',\ T'\r$,
which is what we wanted to prove.
\end{proof}
\subsection{Subresultant p.r.s}
As explained in the introductory examples~\ref{ex:trois}, the last non nilpotent
and the first nilpotent (which may be zero according to our convention) polynomials in the modified subresultant p.r.s.  need to be retrieved.
The p.r.s is computed modulo $p^e$ which requires special care
since very few results exist when coefficients are in  a non-reduced ring.
In particular the classic   formula~\eqref{eq:prs}
may fail because  of inverting a leading coefficient which may not be invertible.
If such a subresultant is met, either all its coefficients are nilpotent and we have found the first nilpotent,
either it has one invertible coefficient and it can be made monic (previous section);
Then a subresultant computation is restarted
(Line~\ref{LNL:ret1} of Algorithm~\ref{algo:LNL}).
Before giving details of  Algorithm~\ref{algo:LNL} in the next section,
the remaining of this section of preliminaries recalls the most fundamental results
of the theory of subresultants.

\paragraph{Review}
The subresultant p.r.s is a central topic  in computer algebra
and as such has been studied extensively; We will only recall the key results. It 
enjoys many convenient  properties both algorithmically
and theoretically.

On one hand a subresultant of two polynomials is the determinant 
of a certain matrix derived from the Sylvester matrix.
It can be defined over any ring. But computing them in this way is costly.
On the other hand, 
there is the subresultant p.r.s. computed through the  formula~\eqref{eq:prs}.
The main and classic result is that both objects are related through the block structure (a.k.a gap structure) theorem.
Computing subresultants through a p.r.s is   cheaper.
The latter assumes traditionally that the input  polynomials  $a$ and $b$
are in a unique factorization domain.

It is possible to address polynomials having coefficients in rings of type $k[x]/\l p^e\r$
thanks to the {\em specialization} property. This is addressed in Section~\ref{sec:prs2},
here only standard definitions and results are recalled.
\begin{Def}\label{def:sub}
Write $n_j$ the degrees in  $y$ of any Euclidean p.r.s of $a$ and $b$
in $(k[x])[y]$, with $n_1:=\deg_y(a)$, $n_2:=\deg_y(b)$.
\begin{itemize}
\item For $0\le j\le n_2-1$, the $j$-th subresultant of $a$ and $b$,
written $\Sr_j(a,b)$, is 
defined as the  polynomial of degree at most $j$ whose coefficients
are certain minors of the  Sylvester matrix (see {\em e.g.} \cite[Prop.~7.7.1]{mishra1993}).

\item For $j=n_2$, $\ \ \Sr_{n_2}(a,b):=\lc(b)^{n_1-n_2-1} b $.

\item For $n_2<  j < n_1-1$,  $\ \ \Sr_j(a,b):=0$.

\item  Let $\nu_1:=\deg_y(a \bmod T)$ and $\nu_2:=\deg_y(b \bmod T)$ and assume $\nu_1 \ge \nu_2$.
For $0\le j<\nu_1-1$ we define the $j$-th subresultant of $(a \bmod T)$ and $(b \mod T)$, 
written $\S_j((a \bmod T),\ (b \bmod T))$, whose coefficients are certain minors of the  Sylvester matrix of $(a \bmod T)$ and $(b \bmod T)$.
\end{itemize}
\end{Def}
Both subresultant families $(\Sr_j (a,b))_{0\le j\le n_2-1}$ and $(\S_j((a \bmod T),\ (b \bmod T)))_{0\le j\le \nu_2-1}$
are related by the following functorial property:
\begin{Lem}[Corollary~7.8.2 of~\cite{mishra1993}]\label{lem:hom}
Write $R=k[x]/\l T\r$, and $\phi : (k[x])[y]\to R[y]$ the reduction map.
As above let $n_1:=\deg_y(a)$ and $n_2:=\deg_y(b)$ with $n_1\ge n_2$,
and $\nu_1:=\deg_y(a \bmod T)$ and $\nu_2:=\deg_y(b \bmod T)$. Assume $\nu_1 \ge \nu_2$. Then
for $0\le j < n_2$,
\begin{enumerate}
\item If $\nu_2=n_2$  and $\nu_1\le n_1$, then $\phi(\Sr_j(a,b)) = \phi(\lc(b))^{n_1-\nu_1} \S_j(\phi(a), \phi(b))$.
\item If $\nu_2 \le n_2$ and $\nu_1 = n_1$, then $\phi(\Sr_j(a,b)) = \phi(\lc(a))^{n_2-\nu_2} \S_j(\phi(a), \phi(b))$.
\item If $\nu_2 < n_2$ and $\nu_1 < n_1$, then $\phi(\Sr_j(a,b))=0$.
\end{enumerate}
\end{Lem}

Writing $F_1:=a$ and $F_2:=b$, the subresultant p.r.s  $[F_1,\ F_2,\ \ldots,\ F_r,\ 0]$  
over $(k[x])[y]$ is defined through the formula below
  (sometimes called ``First kind subresultant p.r.s''~\cite[Definition~7.6.4]{mishra1993}):
\begin{eqnarray}
\nonumber F_3 & :=&  (-1)^{n_1-n_2+1} \prem( F_1,\ F_2)  \quad \text{and letting}\quad  c_3:=-1,\\
\label{eq:prs} 
 c_i &: =&  \left(\frac{ \lc(F_{i-2})}{c_{i-1}}\right)^{n_{i-3} - n_{i-2}} c_{i-1}
 \quad \text{for}\quad i\ge 4\\
 \nonumber F_i & :=&    \frac{\prem(F_{i-2} ,\ F_{i-1} )}{-\lc(F_{i-2}) (-c_i)^{n_{i-2} - n_{i-1}}},
\quad \text{for}\quad i\ge 4
\end{eqnarray}

\begin{Theo}[Subresultant's chain theorem --- Thm.~7.9.4 of~\cite{mishra1993}]\label{th:chain}
The subresultant p.r.s $[F_i]_{i\ge 3}$ and the subresultant chain $(\Sr_j(a,b))_{j=0,\ldots,n_2-1}$ are related
as follows:
$$
\Sr_{n_{j-1} -1} (a,\ b) = F_{j}, \quad \text{for}\quad j=3,\ldots, r.
$$
\end{Theo}
$\Sr_{n_{j-1}-1}(a,\ b)$ is the top subresultant in the block it belongs to.

\section{Gr\"obner basis from modified p.r.s\label{sec:lF}}

The algorithm~\ref{algo:SToGB} ``$\SToGB$'', presented in paragraph~\ref{sec:SToGB},
produces a minimal lexGb of the ideal $\l a,\ b,\ p^e\r$.
It is  very simple if we take 
the modified subresultant
algorithm~\ref{algo:LNL}  ``$\LNN$''
for granted: one call to ``$\LNN$'' (Line~\ref{SToGB:iterate}),
a recursive call (Line~\ref{SToGB:G1})
and an update of the output of this
recursive call (Line~\ref{SToGB:G}).

On the other hand the algorithm~\ref{algo:LNL} ``$\LNN$'' is more technical, as often
the case when dealing with subresultants. The rough idea summarizes
as: ``make subresultants monic'' if their leading coefficient is not invertible modulo $T$,
then pursue the computations of subresultants through a recursive call, until a nilpotent
subresultant is found.
This simple description hides however the special care that corner cases and degree conditions
require. It thus might be useful
to have a look at Algorithm~\ref{algo:SToGB} ``$\SToGB$'' in prior Algorithm~\ref{algo:LNL} ``$\LNN$'',
as a motivation.

\subsection{Subresultant p.r.s. modulo $T$\label{sec:prs2}}
This section introduces the algorithm~\ref{algo:LNL} ``\LNN''.
Running the subresultant p.r.s over $k[x,y]$ yields
significant and unnecessary growth in the degree in $x$ of the coefficients. Only the image modulo $T$,
that is of degree in $x$ bounded by that of $T$, is needed.
Unfortunately, the formula~\eqref{eq:prs} formally works only over unique factorization domains.
The classic workaround consists in taking the homomorphic image by $\phi$:
\begin{Prop}\label{prop:prs}
  Let polynomials $F_1$ and $F_2$ in $(k[x])[y]$ of respective degrees (in $y$)
$n_1,\ n_2$, and let $\phi$ be the homomorphism $(k[x])[y]\to R[y]$ as defined above.
For a given integer $ 2\le i\le r+1$, if the $(\lc(F_{j-1}))_j$ are invertible modulo $T$ for $j=2,\ldots,\ i $, then 
the formula~\eqref{eq:prs} specializes well by $\phi$:

Write $\overline{c_i} = \phi(c_i)$.
For $j=3$ then, $\overline{c_3}:= -1$ and $\phi(F_3):=(-1)^{n_1-n_2+1} \prem(\phi(F_1),\ \phi(F_2))$. And,
for $i+1 \ge j \ge 4$:
  $$
\overline{c_j}  =  \left(\frac{ \lc(\phi(F_{j-2}))}{\overline{c_{j-1}}}\right)^{n_{j-3} - n_{j-2}} \overline{c_{j-1}}
\quad\quad  \nonumber \phi(F_j)  :=     \frac{\prem(\phi(F_{j-2}) ,\ \phi(F_{j-1}) )}{-\lc(\phi(F_{j-2})) (-\overline{c_j})^{n_{j-2} - n_{j-1}}}.
$$
In particular $\S_{n_{j-1}-1}(\phi(F_1),\ \phi(F_2)) = \phi(F_j)$ for $j=3,\ldots, i$, and can be computed modulo $T$
using formula~\eqref{eq:prs} (index $i$ included although $\lc(F_i)$ is not assumed invertible modulo $T$).
  \end{Prop}
\begin{proof}
  We need to show that $\phi$ commutes with all operations and primitives involved in the formula~\eqref{eq:prs}.
  First of all, $\lc(\phi(F_j))=\phi(\lc(F_j))$ for $j=2,\ldots,i-1$ by assumption. 
  Hence the Euclidean degree sequence $(n_j)_j$ is also that of the one initiated with
  $(\phi(F_1), \phi(F_2))$ too, up to $j < i$.
  Therefore, for $j=2,\ldots, i-1$, the pseudo-division equality $\lc(F_{j-1})^{n_j - n_{j-1}+1} F_j = q F_{j-1} + \prem(F_j,\ F_{j-1})$
  specializes by $\phi$: $\lc(\phi(F_{j-1}))^{n_j - n_{j-1}+1} \phi(F_j) = \phi(q) \phi(F_{j-1}) + \prem(\phi(F_j), \phi(F_{j-1}))$,
whence: $\phi(\prem(F_{j-1},\ F_j))= \prem(\phi(F_j), \phi(F_{j-1}))$, for $j=2,\ldots, i-1$.
Moreover  the pseudo-quotient and pseudo-remainder are uniquely determined.
  Besides, the formula  involves only algebraic operations which commute by definition with the  homomorphism $\phi$. Therefore, $\phi$
  commutes as expected. In particular, although $\lc(F_i)$ is not assumed to be invertible modulo $T$, we still have 
  $$\phi(F_i) = \phi\left(\frac{\prem(F_{i-1},\ F_{i-2})}{-\lc(F_{i-2})(-c_i)^{n_{i-2}-n_{i-1}}} \right)
  = \frac{\prem(\phi(F_{i-1}), \phi(F_{i-2}))}{-\lc(\phi(F_{i-2}))(-\bar{c_i})^{n_{i-2} - n_{i-1}}}.$$

  Finally, the assumptions made on $F_1$ and $F_2$ fall into Case 1. of Lemma~\ref{lem:hom}: $\phi(\Sr_j(F_1,\ F_2)) = \S_j(\phi(F_1),\ \phi(F_2))$.
  Besides, Theorem~\ref{th:chain} gives $F_j = \Sr_{n_{j-1}-1}(F_1,\ F_2)$ hence 
$$
\phi(F_j) = \S_{n_{j-1}-1}(\phi(F_1),\ \phi(F_2))\quad\text{for} \quad j=2,\ldots, i.
$$
  \end{proof}

\begin{Cor}\label{cor:prs}
With the  notations as in Proposition~\ref{prop:prs}, let $f_1:=\phi(F_1)$ and $f_2:=\phi(F_2)$.
Consider the p.r.s $F_1,\ F_2, \ldots, F_r,\ F_{r+1}$  computed in $(k[x])[y]$ with $F_{r+1}=0$. Let  $f_1,\ f_2,\ f_3,\ \ldots,\ f_i,\ f_{i+1}$ 
be the one computed  modulo $T$ as explained in Proposition~\ref{prop:prs}, under the assumption
that the $(\lc(F_j))_{1 \le j \le i} $'s are all invertible modulo $T$. Assume that  $\lc(F_{i+1})$ is not invertible modulo $T$ or $F_{i+1}\equiv 0 \bmod T$. 

We have $\l F_1,\ F_2,\ T\r= \l f_{j-1},\ f_j, \ T\r$ for any $2\le j \le i+1$.
\end{Cor}
\begin{proof}
Over the unique factorization domain $k[x]$, it is classical that
the subresultant p.r.s verifies $\l F_1,\ F_2\r=\l F_{j-1},\ F_j\r$ for any $2\le j\le r+1$.
By Proposition~\ref{prop:prs}, $\phi(F_j)=f_j$ for $j=1,\ldots, i+1$, hence:
$$
\l F_1,\ F_2,\ T\r=\l \phi(F_1), \phi(F_2),\ T\r = \l \phi(F_{j-1}),\ \phi(F_j),\ T\r = \l f_{j-1},\ f_j,\ T\r,
$$
for $j=2,\ldots, i+1$ \end{proof}

\begin{algorithm}[H]

\DontPrintSemicolon

\KwIn{1. polynomials $f_1,\ f_2\in k[x,y]$, $\deg_y(f_1)\ge \deg_y(f_2)$. The leading coefficient of $f_1$ is invertible modulo $T$,
and $f_2\not=0$.

  2. $T=p^e\in k[x]$, the power of an irreducible polynomial $p$.
}

\KwOut{1-2. monic (in $y$) polynomials $u, v\in k[x,y]$
  verifying the degree condition of Proposition~\ref{prop:LNL}.

  3-4. monic polynomials $U,\ T_1\in k[x]$ such that: \hfill $ T_1 U = T,\qquad \l f_1, \ f_2 , \ T\r =  \l u,\ v \, U ,\ T\r$.
}

\BlankLine

\Repeat(\nllabel{LNL:prs})
{
  $\lc(f_i)$ is {\em not} invertible modulo $T$ or $f_{\ell+1}=0$
}
{
  Compute the subresultant p.r.s $f_1,\ f_2,\ldots,\ f_\ell,\ f_{\ell+1}$  of $f_1$ and $f_2$ modulo $T$
  (with $f_{\ell+1}=0$, following Proposition~\ref{prop:prs})
}

\If(\tcp*[f]{$\lc(f_j)$ is invertible modulo $T$ for $j <i$})
   {$\lc(f_i)$ is not invertible modulo $T$\nllabel{LNL:if}}{
      $(\, b ,\ U\, ) \leftarrow \MF(f_i,\ T)$ \nllabel{LNL:MF1}\tcp*[r]{$b$ is monic and $\l U b,\ T\r=\l f_i, \ T\r$}
       $a \leftarrow (\lc(f_{i-1})^{-1} \bmod T) f_{i-1} \bmod T$\nllabel{LNL:a}        \tcp*[r]{$a$ is monic and $\l a,\ T\r =\l f_{i-1},\ T\r$}
      \eIf(\tcp*[f]{$f_i$ is nilpotent}){$\deg_x(U)>0$ \nllabel{LNL:if2}}{
        \Return{$(\, a,\ b,\ U,\ T/U \, )$} \nllabel{LNL:ret1}\;
      }(\tcp*[f]{$f_i$ is not nilpotent})
          {
      \Return{ $\LNN(a ,\ b ,\ T)$}\nllabel{LNL:ret2}\tcp*[r]{Recursive call}
          }
   }
   
\tcp*[h]{Here, all the $\lc(f_1),\ldots,\ \lc(f_\ell)$ are invertible modulo $T$, and $f_{\ell+1}=0$}\;
\Return{$(\ (\lc(f_\ell) \bmod T)^{-1} f_\ell ,\ 0,\ 1,\ T\ )$\nllabel{LNL:ret3}}
\caption{\label{algo:LNL} \quad $(\, u,\ v,\ U,\ T_1\, ) \leftarrow \LNN(f_1,\ f_2,\ T)$}
\end{algorithm}

\begin{Prop}[Correctness of Algorithm~\ref{algo:LNL}]\label{prop:LNL}
  The output $u,\ v, \ U,\ T_1$ of $\LNN(f_1,\ f_2,\ T)$  verifies:
  $$
\l u,\ v U,\ U T_1\r =  \l f_1,\ f_2,\ T\r \quad\text{and}\quad U T_1 = T.
$$
If $v=0$ then $U=1$ and $T_1=T$.
Moreover, we also have the following {\em degree conditions:}
\begin{equation}\label{eq:deg-cond}
\text{if} \quad v\not=0 \quad \text{then}\quad \deg_y(v)<\deg_y(u)\quad\text{and}
\quad \deg_y(u)<\deg_y(f_1),
\end{equation}
except in the following corner cases:
\begin{enumerate}[(i)]
\item\label{enum:LNN-i} $\deg_y(u) = \deg_y(v)$ possibly holds if $u$ and $v$ are respectively the monic form of $f_1$ and $f_2$
  with $f_2$  nilpotent (implying $0<\deg_x(U)<\deg_x(T)$) and $\deg_y(f_2)=\deg_y(f_1)$. 
 
\item\label{enum:LNN-ii} $\deg_y(u)=\deg_y(f_1)$ possibly holds under the condition \ref{enum:LNN-i}  above.

\item\label{enum:LNN-iii}  The only other situation where $\deg_y(u)=\deg_y(f_1)$ 
is when $\lc(f_2)$ is not nilpotent (equivalently is invertible) modulo $T$,
and $\lc(f_2)f_1 \equiv \lc(f_1)f_2 \bmod \l p\r$.

In this case, $u$ is the monic form of $f_2$ and $v$ that of $\pm f_3$.
\end{enumerate}
\end{Prop}

\begin{ex}[corner cases in the degree condition]\label{ex:corner-cases}
  An example of corner cases \ref{enum:LNN-i}-\ref{enum:LNN-ii} is 
$$
f_1 = y + p,\qquad f_2= p y + p \quad\text{modulo}\quad T=p^2.
$$
Remark that  $f_2$ is nilpotent and  $(\, y+1,\ p \, )\leftarrow  \MF(f_2,\ T)$.
This is thus the first nilpotent subresultant and we have $v=y+1$ and $U=p$.
Besides $f_1$ is not nilpotent and already in monic form, thus $u=f_1$.
Note that $U=p$ verifies $0 < \deg_x(U) < \deg_x(T)$.
We have $\deg_y(v)=\deg_y(y+1)=1=\deg_y(y+p)=\deg_y(u)$
(degree equality~\ref{enum:LNN-i}) and $\deg_y(u)=1=\deg_y(f_1)$ (degree equality~\ref{enum:LNN-ii}).

An example of corner case~\ref{enum:LNN-iii} is when 
$$
f_1=y+p,\qquad f_2=y\quad \text{modulo}\quad T=p^2.
$$
We see directly that $\lc(f_2) f_1 \equiv \lc(f_1) f_2 \equiv y \bmod \l p\r$.
Then $\pm f_3= \prem(f_1,\ f_2)= p$ is the first nilpotent subresultant.
We thus have  $(\, v,\ U\, )\leftarrow \MF(f_3,\ T)$ with   $v=1$ and $U=p$.
Moreover $u=f_2$, being the last non nilpotent,
and already in monic form.
Finally, note 
the degree equality~\ref{enum:LNN-iii}: $\deg_y(f_1)=1=\deg_y(u)$. \hfill \qed
  \end{ex}
\begin{proof}
We start by the proof of correctness before turning to the proof of the degree conditions~\eqref{eq:deg-cond}.

{\em Proof of correctness.} We investigate the three returns in the algorithm separately.

{\em Case 1: Exit at Line~\ref{LNL:ret3}.} The algorithm does not enter Lines~\ref{LNL:if}-\ref{LNL:ret2},
which means that the if-test Line~\ref{LNL:if} was never passed: the leading coefficients $\lc(f_1),\ldots,\lc(f_\ell)$
in the p.r.s  are all invertible modulo $T$,
and $f_{\ell+1}=0$.
The output at Line~\ref{LNL:ret3} is 
$$
u= (\lc(f_\ell) \bmod T)^{-1}  f_\ell \bmod T,
\quad v=0,\quad   U=1,
\quad T_1=T.
$$ 
Therefore $U T_1 =T$. Moreover, 
$\l u,\ v U,\ T\r = \l u,\ T\r$.
Note that $u$ is monic
verifying $\l u,\ T\r = \l f_\ell ,\ T\r$.
Besides, by Corollary~\ref{cor:prs},  $\l f_1,\ f_2,\ T\r=\l f_\ell,\ f_{\ell+1},\ T\r$, and since $f_{\ell+1}=0$ we get
$\l f_1,\ f_2,\ T\r =\l f_\ell,\ T\r$.
We obtain: $\l f_1, \ f_2, \ T\r = \l u,\ T\r$ as expected.

In the remaining two cases, $v$ is not zero, which proves the assertion that if $v=0$ then $U=1$, and $T_1=T$. 
\smallskip

{\em Case 2: Exit at Line~\ref{LNL:ret1}.}
The if-test Line~\ref{LNL:if} tells that 
$\lc(f_i)$ is not invertible modulo $T$, while the $\lc(f_j)$s are for $j<i$.
At Line~\ref{LNL:MF1} ``$(\, b ,\ U\, ) \leftarrow \MF(f_i,\ T)$''
we have from the definition of the ``$\MF$'' algorithm
that $\l U b,\ T \r =\l f_i,\ T\r$,
with $U|T$ and $b$ monic.
Moreover $\deg_x(U)>0$ (the if-test at Line~\ref{LNL:if2}) implies that $f_i$ is nilpotent.
On the other hand, $f_{i-1}$ is invertible modulo $T$: the inversion at Line~\ref{LNL:a}
is correct and $a$ is monic verifying $\l a,\ T\r = \l f_{i-1},\ T\r$. By Corollary~\ref{cor:prs} we have
$\l f_1,\ f_2,\ T\r = \l f_{i-1},\ f_i,\ T \r=\l a,\ f_i,\ T\r = \l  a,\ U b, \ T\r$.
This is what we wanted to prove in this case 2.

\smallskip

{\em Case 3: Exit at Line~\ref{LNL:ret2}.} There, $\lc(f_i)$ is not invertible modulo $T$, and since $\deg_x(U)=0$, $f_i$ is not nilpotent.
As above, we have $\l b,\ T\r=\l f_i,\ T\r$.
Additionally, $\deg_y(b)<\deg_y(f_i)$:
$f_i$ is not nilpotent but $\lc(f_i)$ is, so the Weierstrass
factorization of $f_i$ produces a monic polynomial $b$
of degree smaller than that of $f_i$.
And since $\lc(f_{i-1})$ is invertible modulo $T$,  $\l a,\ T\r=\l f_{i-1},\ T\r$.
The recursive call is thus made with monic polynomials $a$ and $b$ of degree smaller than that of $f_1$ and $f_2$; Indeed, observe that:
\begin{equation}\label{eq:degLNL}
\deg_y(b) < \deg_y(f_i) \le \deg_y(f_2) 
,\qquad
\deg_y(a) = \deg_y(f_{i-1}), \qquad (\text{hence}
\ \deg_y(b)<\deg_y(a))
\end{equation}
Therefore, the recursive call ultimately boils down to Case 1: its output
$(\, u,\ v,\ U,\ T_1\, )$ verifies $\l a,\ b,\ T\r =\l u,\ v U,\ T_1 U\r$ and $T_1 U=T$. Since $\l a,\ b, \ T\r = \l f_{i-1},\ f_i,\ T\r$,
and that by Corollary~\ref{cor:prs} we have $\l f_1,\ f_2,\ T\r=\l f_{i-1},\ f_i, \ T\r$. We conclude that
$\l u,\ v U,\ T_1 U\r =\l f_1,\ f_2,\ T\r$.
\medskip

{\em Proof of the degree conditions~\eqref{eq:deg-cond}.}
If at least two pseudo-divisions occur in the algorithm (including in recursive calls),
they induce at least two strict degree decreases
 in the modified p.r.s computed.
Both $u$ and $v$ being the monic form of the last two polynomials in that modified p.r.s,
it follows that:
$$
\text{if } v\not=0, \quad \deg_y(v) < \deg_y(u)<\deg_y(f_2)\le\deg_y(f_1)
$$
which proves the degree conditions~\eqref{eq:deg-cond}.
No corner cases can happen in this situation.

If exactly one pseudo-division occurs then:
\begin{enumerate}[-1- ]
\item\label{enum:-1-} either $\lc(f_2)$ is invertible modulo $T$ and $f_3=\pm \prem(f_1,\ f_2)$
\item\label{enum:-2-} or not and then, letting $b$ be the monic form of $f_2$ (Line~\ref{LNL:MF1}),
  and $a$ that of $f_1$ (Line~\ref{LNL:a}), a recursive call ``$(\, u,\ v,\ U,\ T\, ) \leftarrow
  \LNN(a,\ b,\ T)$'' is performed
 (Line~\ref{LNL:ret2}). 
But then, $a$ and  $b$ being monic at least a pseudo-division should take place inside this recursive call, a contradiction.
This situation~\ref{enum:-2-} cannot happen in case of one pseudo-division only.
\end{enumerate}
Note that $f_2$ cannot be nilpotent for one pseudo-division to occur.
The algorithm then stops, meaning that $f_3$ is nilpotent (which may be zero according to our convention).
Then $v$ is its monic form, and $u$ is the monic form of $f_2$. We have:
\begin{equation}\label{eq:LNN-f3}
\text{if } v\not=0, \quad \deg_y(v) \le  \deg_y(f_3) <  \deg_y(f_2) = \deg_y(u) \le\deg_y(f_1)
\end{equation}
which proves the degree conditions~\eqref{eq:deg-cond}.
For $\deg_y(u) = \deg_y(f_1)$ (possible corner cases~\ref{enum:LNN-i}-\ref{enum:LNN-ii} ) to hold,
necessarily $\deg_y(u)=\deg_y(f_2)=\deg_y(f_1)$. The pseudo-division of $f_1$ by $f_2$ writes as:
$$
\lc(f_2)^{\deg_y(f_2)-\deg_y(f_1)+1} f_1 = \pquo(f_1,\ f_2) f_2 \pm f_3.
$$
The degree equality $\deg_y(f_1)=\deg_y(f_2)$ implies that $\pquo(f_1,\ f_2)=\lc(f_1)$,
and since $f_3 \equiv 0 \bmod \l p\r$:
$$
\lc(f_2) f_1 \equiv \lc(f_1) f_2 \bmod \l p\r.
$$
This is the corner case~\ref{enum:LNN-iii}. Corner cases~\ref{enum:LNN-i}-\ref{enum:LNN-ii}
do not occur.

If no pseudo-division occurs at all, then $f_2$ is already the ``first nilpotent''.  The
polynomials $u$ and  $v$ are then respectively the monic form of $f_1$ and $f_2$.
Since $\lc(f_1)$ is invertible modulo $T$ by assumption, $u\equiv \lc(f_1)^{-1}  f_1 \bmod T$,
in particular $\deg_y(u)=\deg_y(f_1)$: this is corner case~\ref{enum:LNN-ii}.
We can then observe that:
$$
\deg_y(v)\le \deg_y(f_2) \le \deg_y(f_1)=\deg_y(u).
$$
This proves the degree conditions~\eqref{eq:deg-cond}.
For $\deg_y(v)=\deg_y(u)$ to hold, the condition $\deg_y(f_1)=\deg_y(f_2)$
is necessary, proving the situation of corner case~\ref{enum:LNN-i}.
\end{proof}

\subsection{Deducing the Gr\"obner basis\label{sec:SToGB}}
The main algorithm~\ref{algo:SToGB} essentially iterates
the
algorithm~\ref{algo:LNL} ``$\LNN$'' through recursive calls.
We refer to Example~\ref{ex:trois} in Introduction.


\begin{algorithm}[H]
  \SetAlgoVlined 
  \DontPrintSemicolon 
\KwIn{1.-2. polynomials $a ,\ b\in k[x,y]$, $\deg_y(a)\ge \deg_y(b)$. The leading coefficient of $a$ is invertible modulo $T=p^e$,
$b$ is not nilpotent modulo $T$ (which also means $b\not=0$ by our convention).

  3. The power of an irreducible polynomial $T=p^e\in k[x]$}

\KwOut{A minimal lexGb $\calG$ of the ideal $\l a,\ b, \ T\r$.}

\BlankLine

\If{$a$ or $b$ is a non-zero constant}{\Return{$[1]$}\nllabel{SToGB:ret1}}
 
    $(\, u,\ v,\ U,\ T_1\, )  \leftarrow \LNN(a,\ b,\ T)$
    \nllabel{SToGB:iterate} \tcp*[r]{$\l u,\ v U,\ U T_1\r = \l a,\ b,\ T\r$}
    \If{$u =1$}{\Return{$[1]$}\nllabel{SToGB:ret2}}
\If{$v=0$}{\Return{$[u,\ T]$}\nllabel{SToGB:ret0}}
     $\calG_1\leftarrow \SToGB(u,\ v,\ T_1)$ \nllabel{SToGB:G1}\tcp*[r]{Recursive call}
       $\calG \leftarrow [u] \cat [U\cdot g\ :\ g \in \calG_1] $\nllabel{SToGB:G}
       \tcp*[r]{Update of the output of the recursive call}
     \Return{$\calG$\nllabel{SToGB:ret4}}
        
 \caption{\label{algo:SToGB} \quad $\calG \leftarrow \SToGB(a,\ b,\ T)$ }
\end{algorithm}

\begin{Theo}\label{th:SToGB}
The output $\calG$ is a minimal lexGb of $\l a,\ b,\ T\r$.
\end{Theo}

\begin{proof}
Among the four returns in the algorithm, the first three ones are base cases
treated in Cases 1, 2, 3 below. The last return (Case 4) involves a recursive call and requires more care.

{\em Case 1: Exit at Line~\ref{SToGB:ret1}.} Here $a$ or $b$ is a non-zero constant, thus $\l a,\ b,\ T\r = \l 1\r$ and the output should be $[1]$.

{\em Case 2: Exit at Line~\ref{SToGB:ret2}.} By the input' Specification 1., $a$ is assumed to  have an invertible leading
coefficient modulo $T$, and by input' Specification 2. $b$ is not nilpotent (which implies that $b\not=0$
according to our convention).
This legitimates the call to $\LNN(a,\ b,\ T)$ at Line~\ref{SToGB:iterate}, according
that $a$ and $b$ should verify these assumptions.
By definition of ``\LNN'', the output $(\, u,\ v,\ U,\ T_1\, )$ verifies $\l u, \ v U,\ U T_1\r=\l a,\ b,\ T\r$. So that if $u=1$ then $\l a,\ b, \ T\r=\l 1\r$
and the output should be $[1]$. This is precisely what returns the algorithm at Line~\ref{SToGB:ret2}.

{\em Case 3: Exit at Line~\ref{SToGB:ret0}.}
Here $v=0$, which implies $U=1$, $T_1=T$ and
$\l a,\ b, \ T\r = \l u,\ v T_1,\ U T_1\r=\l u, \ T \r$
by Proposition~\ref{prop:LNL}.
The return  precisely outputs  $[u,\ T]$,
which is a lexGb (actually a bivariate triangular set).
 \smallskip

{\em Case 4: Exit at Line~\ref{SToGB:ret4}.}  
The recursive call at Line~\ref{SToGB:G1} makes sense, its input $u$, $v$ and $T_1$ satisfying the input' specifications:
$u$ and $v$ are monic and $\deg_y(u)\ge \deg_y(v)$
(this degree inequality follows from the degree condition~\eqref{eq:deg-cond}
of the output of ``$\LNN$'').
Here  $v\not=0$, as the case $v=0$ is treated in Case 3. By the degree condition~\eqref{eq:deg-cond} of Proposition~\ref{prop:LNL}
we have $\deg_y(u) < \deg_y(a)$ unless
one of the corner cases~\ref{enum:LNN-ii} or~\ref{enum:LNN-iii}  occur.
In  Case~\ref{enum:LNN-ii}, which assumptions are that of Case~\ref{enum:LNN-i},
$b$ is nilpotent, excluded by the input' Specifications.
And in Case~\ref{enum:LNN-iii}, $v=\pm \prem(a,\ u)$ which is nilpotent (maybe $0$ by our convention,
but the case $v=0$ is already treated in Case 3.).
Thus, by Eq.~\eqref{eq:LNN-f3}:
$$
\deg_y(v)\le \deg_y(\prem(a,\ u))<\deg_y(u)\le \deg_y(b)\le \deg_y(a)
$$
Therefore, the input $(u,\ v,\ T_1)$ of the recursive call Line~\ref{SToGB:G1}
displays in any case a degree decrease compared to the input $(a,\ b,\ T)$
of the main call: always hold the following inequalities
$$
\deg_y(u)\le \deg_y(a), \quad \deg_y(v) \le \deg_y(b), \quad \deg_y(T_1)\le \deg_y(T),
$$
and at least one of the two strict inequalities holds
$$
\quad \deg_y(u) < \deg_y(a),\qquad \deg_y(v) < \deg_y(b).
$$
We can assume by induction that this recursive call is correct: $\calG_1$ is a
minimal lexGb of $\l u,\ v,\ T_1\r$.
Moreover, {\em inside} the recursive call made at Line~\ref{SToGB:G1},
the algorithm either went through
Line~\ref{SToGB:ret1} or through Line~\ref{SToGB:iterate}.
If it exists at Line~\ref{SToGB:ret1}, then since $u \ne 1$, this means $v=1$.
The output is $[1]$ and the final output is $\calG_1$ which is a minimal lexGb.
Else, at Line~\ref{SToGB:iterate}, let us  write
the call to ``$\LNN$'' as:
$$
(\, u_0,\ v_0,\ U_0,\ T_0\, ) \leftarrow \LNN(u,\ v,\ T_1),
$$
and let us prove that $\deg_y(u_0)<\deg_y(u)$.
The degree conditions~\eqref{eq:deg-cond} of the proposition~\ref{prop:LNL}
asserts that it is indeed the case, except maybe for
the corner cases~\ref{enum:LNN-ii} or~\ref{enum:LNN-iii}
where possibly $\deg_y(u_0) = \deg_y(u)$ holds.
Both corner cases~\ref{enum:LNN-ii} or~\ref{enum:LNN-iii} imply that $\deg_y(u)=\deg_y(v)$, but this equality does not hold. Indeed $\deg_y(u) = \deg_y(v)$ holds only under Case~\ref{enum:LNN-i} (as output of the call $\LNN(a,\ b,\ T)$ at Line~\ref{SToGB:iterate} this time) where $b$ is assumed nilpotent. But the specifications of Algorithm~\ref{algo:SToGB} ``$\SToGB$'' prescribes $b$ to be nilpotent.
Hence $\deg_y(v)<\deg_y(u)$ and finally $\deg_y(u_0)<\deg_y(u)$.

Therefore, by construction of the lexGb, the polynomial $u$ has a degree (in $y$) larger than all polynomials constructed in $\calG_1$.
We can thus apply Theorem~\ref{th:Laz}-(B) (to $u$ and $\calG_1$) which allows to conclude that $\calG:=[u] \cat [U  g \ :\ g\in \calG_1]$ (as defined at Line~\ref{SToGB:G})
is a {\em minimal} lexGb if and only if $u\in \calG_1$. This is obviously the case, since $\l \calG_1\r=\l u,\ v,\ T_1\r$.

It remains to show that the minimal lexGb $\calG$ generates the ideal $\l a,\ b,\ T\r$ to achieve the proof of Case~4.
Since $\l u,\ v U, \ T\r=\l a,\ b,\ T \r$, it suffices to show that
$\l u,\ v U,\ T\r=\l \calG\r$. From $\l u,\ v, \frac{T}{U}\r=\l \calG_1\r$, we obtain $\l u U ,\ v U,\ T\r = \l U \calG_1\r$.
Hence $\l \calG\r = \l u \r + \l U \calG_1\r = \l u\r + \l u U,\ v U, \ T \r=\l u,\ v U, \ T\r$. 
\end{proof}

\section{Generalizing dynamic evaluation\label{sec:D5}}

\subsection{Splitting  ``invertible/nilpotent''}
\paragraph{Part of a polynomial whose support of irreducible factors are given by another polynomial}
Assume that two monic polynomials $a$ and $b$ are squarefree. 
Then $a/\gcd(a,b)$ and $\gcd(a,b)$ are pairwise coprime.
Moreover $\gcd(a,b) = \prod_{v_p(b)>0} p^{v_p(a)}$, according that $v_p(a)=1$ or $0$.
We need a similar routine when $a$ and $b$ are not squarefree.
As seen in Example~\ref{ex:newD5}, it suffices to iterate a gcd computation until the two factors become coprime.

\begin{Def}\label{def:isol}
Let $\calP$ be the set of irreducible polynomials of $k[x]$. Write $a=\prod_{p\in \calP} p^{v_p(a)}$ the factorization into irreducibles of $a\in k[x]$,
and $b=\prod_{p\in\calP} p^{v_p(b)}$ that of $b\in k[x]$.
The $b$-component of $a$, denoted $a^{(b)}$ is the polynomial  $a^{(b)} := \prod_{p \in \calP ,\ v_p(b)>0} p^{v_p(a)}$. 
\end{Def}

\begin{ex}[Isolating irreducible factors algorithm~\ref{algo:isol} ``$\isol$'']
Let $a=x^3 (x+1)^4 (x+2)$ and $b = x^4 (x+1)$. Then $a^{(b)} = x^3 (x+1)^4$.
\end{ex}
There are several ways to take the $b$-component, we give a standard and easy one.

\begin{algorithm}
\DontPrintSemicolon
\KwIn{1-2. Polynomials $a,b\in k[x]$.}
\KwOut{1-2. The $b$-component $a^{(b)}$ (Definition~\ref{def:isol}), and  $\frac{a}{a^{(b)}}$. They are coprime. 
}

\BlankLine

$a_0 \leftarrow a$~~;~~~ $b_0\leftarrow b$~~;~~~$c_0\leftarrow 1$~~; ~~~$i\leftarrow 0$\;
\Repeat{$b_i=1$}{
$b_{i+1}\leftarrow \gcd(a_i,\ b_i)$~~~;~~~~
 $a_{i+1}\leftarrow \frac{a_i}{b_{i+1}}$~~~;~~~~ $c_{i+1} \leftarrow c_i b_{i+1}$~~~;~~~$i \leftarrow i+1$\nllabel{isol:gcd}\;
}
\Return{$(\, c_i,\ a_i\, )$\tcp*[f]{Write $s+1$ the index $i$ at the exit of the repeat loop}}
\caption{\label{algo:isol} $(\, a^{(b)},\ \frac{a}{a^{(b)}}\, )\leftarrow \isol(a,\ b)$}
\end{algorithm}

\begin{Lem}[Correctness of Algorithm~\ref{algo:isol} ``$\isol$'']
The output $c_{s+1},\ a_{s+1}$ are coprime polynomials that verify $c_{s+1}=c_s,\ a_{s+1}=a_s$, and  $c_s a_s=a$, with 
$c_s=a^{(b)} = \prod_{p\in \calP,\ v_p(a)>0,\ v_p(b)>0} p^{v_p(a)}$. 
\end{Lem}
\begin{proof}
Let $s+1$ be the last index $i$ at the exit of the repeat  loop,
so that $c_{s+1},\ a_{s+1}$ are the output.
The exit condition of the repeat/until loop is $b_i=1$, hence $b_{s+1}=1$ with the notation we have adopted.
Thus $c_{s+1} = c_s b_{s+1}=c_s$ and $a_{s+1}=a_s/b_{s+1}=a_s$.
We  must show  two things. First, that $v_p(c_s)=v_p(a)$ and $v_p(c_s)>0$ 
is equivalent to $v_p(b)>0$.
Second,  that $a_s =\frac{a}{c_s} $ is coprime with $c_s$.

If we rewrite the update line~\ref{isol:gcd} in terms of $p$-adic valuations, we get:
\begin{equation}\label{eq:val}
\begin{array}{rcl}
v_p(b_{i+1}) & = & \min (v_p(a_i), v_p(b_i))\\
v_p(a_{i+1}) & = & v_p(a_i) - v_p(b_{i+1})\\
v_p(c_{i+1}) & =& v_p(c_i) +  v_p(b_{i+1})
\end{array}
\end{equation}
With the notations of Definition~\ref{def:isol}, at initialization we have $a_0=\prod_{p\in \calP} p^{v_p(a)}$ and
$b_0=\prod_{p\in \calP} p^{v_p(b)}$.
Note that for any irreducible polynomial $p$ such that $v_p(b)>0$, the sequence of integers $(v_p(b_i))_{i=0,\ 1,\ldots}$
strictly decreases to zero, regarding that $b_{s+1}=1$ and thus $v_p(b_{s+1})=0$.

Assume first that there exists an  index $i\ge 0$ such that $v_p(a_i) < v_p(b_i)\ (\star)$, and then take the smallest index  $j$ that verifies this inequality.
Then $v_p(b_{j+1}) = \min(v_p(b_j),\ v_p(a_j)) = v_p(a_j)$ and thus $v_p(a_{j+1}) = v_p(a_j) - v_p(b_{j+1})=0$. 
It follows that $v_p(b_{j+2}) =\min(v_p(b_{j+1}),\ v_p(a_{j+1}))=0$, then $v_p(a_{j+2}) = v_p(a_{j+1})=0$, and
$v_p(c_{j+2}) = v_p(c_{j+1}) $. By induction, we observe that $v_p(b_{j+2}) =\cdots = v_p(b_{s+1}) = 0$,
that $v_p(a_{j+1}) = \cdots = v_p(a_{s+1}) = 0$, and that $v_p(c_{j+1}) = \cdots = v_p(c_{s+1})$.

Besides, since $j$ is the smallest index for which $v_p(a_j)<v_p(b_j)$,
we have that $v_p(b_{\ell}) \le  v_p(a_{\ell})$ for $\ell<j$. Consequently,
$v_p(b_{\ell+1}) = \min(v_p(b_\ell),\ v_p(a_\ell))=v_p(b_\ell)$ and by induction
we observe that  $v_p(b_j) = v_p(b_{j-1}) = \cdots = v_p(b_2)=v_p(b_1)=v_p(b_0)=v_p(b)$.
We obtain:
$$
v_p(c_s) =  j v_p(b)  +  v_p(a_j), \quad \text{and}\quad v_p(a_j) = v_p(a_{j-1}) - v_p(b_j) = v_p(a_{j-1}) - v_p(b).
$$
By induction, $v_p(a_j) = v_p(a_{j-1}) - v_p(b) = v_p(a_{j-2}) - 2 v_p(b) = \cdots = v_p(a_0) - j v_p(b)$.
We get: $v_p(c_s) = j v_p(b)  + v_p(a_0) - j v_p(b) = v_p(a_0)$.
In conclusion, $v_p(c_s) = v_p(a)$ when $v_p(b) >0$ and when there is a $j$ that satisfies $(\star)$.

If there is no such $j$, then the first equality in Eq.~\eqref{eq:val} 
implies that  $v_p(b_{s+1}) = v_p(b_s) = \cdots = v_p(b_0)=v_p(b)$.
Moreover the output $b_{s+1}$ is equal to 1 hence $v_p(b_{s+1})=0$ thereby
$v_p(b)=0$. Additionally, it follows still from Eq.~\eqref{eq:val}
that $v_p(c_0) = v_p(c_1)=\cdots=v_p(c_{s+1})=0$. 
For such an irreducible polynomial $p$,  the definition
of the $b$-component $a^{(b)}=c_{s+1}$ says that $v_p(c_{s+1})$ must be $0$.
This concludes the proof of the output's specification 1. in  any case.

Finally, let us prove that the output $c_s,\ a_s$ are coprime.
From $c_s a_s = a$ and $[v_p(b)>0 \Rightarrow  v_p(c_s)=v_p(a)]$, we deduce that $v_p(a_s)=0$ when $v_p(b)>0$.
When $v_p(b)=0$ then $v_p(c_s) =0$, and thus $v_p(a_s)=v_p(a)$.
It follows that, either $(v_p(a_s)=0$ and $v_p(c_s)=v_p(a))$ or
$(v_p(a_s)=v_p(a)$ and $v_p(c_s)=0)$. This means that $c_s$ and $a_s$ are coprime.
\end{proof}

\begin{Rmk}\label{rmk:isol}
 $b$  is nilpotent modulo $a^{(b)}$, and $b$ is invertible modulo $\frac{a}{a^{(b)}}$.
Indeed all irreducible factors of $a^{(b)}$ are irreducible factors of $b$, hence $\sqfp(a^{(b)}) | \sqfp(b)$.
And $\frac{a}{a^{(b)}}$ does not have any common irreducible factors with $b$.
\end{Rmk}

\paragraph{The splitting ``invertible/nilpotent''}
The algorithm~\ref{algo:invNil} ``$\invNil$'' is the cornerstone for generalizing the dynamic evaluation
from a modulus $T$ that is a squarefree
 polynomial to a general modulus.
The splitting that it induces is ``invertible/nilpotent'' instead of ``invertible/zero''  in the standard dynamic evaluation.
It is similar to ``$\isol$'' algorithm except that the inverse of $b$ modulo $\frac{a}{a^{(b)}}$ is also
returned.

\begin{algorithm}[H]
\DontPrintSemicolon
\KwIn{1. Non-zero polynomial $a\in k[x]$\\
2. Non constant monic polynomial $T \in k[x]$}
\KwOut{1. $[f_1,\ T_1]$: monic polynomial $T_1$ and   $f_1$  such that  $f_1 f \equiv 1 \bmod T_1$ if $T_1\not=1$, or $f_1=0$ if $T_1=1$,\\
  2.$[f_2,\ T_2]$: $T_2$ monic polynomial, $f_2$ nilpotent mod $T_2$, and $f_2 \equiv a\bmod T_2$.

  $\star$ Condition: $T=T_1 T_2$, $T_1$ and $T_2$ coprime.
}
\BlankLine

$(\, T_2,\ T_1\, ) \leftarrow \isol(T,\ f)$\;

$f_2 \leftarrow f \bmod T_2$~~;~~~ $f_1\equiv f^{-1} \bmod T_1$\;

\Return{$[[f_1,\ T_1],[f_2, T_2]]$}
\caption{\label{algo:invNil} \quad $[[f_1,\ T_1],[f_2,\ T_2]]\leftarrow \invNil(a,\ T)$}
\end{algorithm}

\begin{Prop}
  The algorithm~\ref{algo:invNil} ``$\invNil$'' splits the polynomial $T$ into two coprime polynomials $T_1$ and $T_2$,
  and outputs two  polynomials $f_1$ and $f_2\in k[x]$ such that:
  $f_1$ is invertible modulo $T_1$ and $f_2$ is nilpotent (maybe zero according to our convention) modulo $T_2$. Moreover,
  $f_1 a \equiv 1 \bmod T_1$, and $a \equiv f_2 \bmod T_2$.
\end{Prop}
\begin{Rmk}
The polynomials $T_1, \ T_2, \ f_1, \ f_2$ are uniquely determined by $a$ and $T$.
\end{Rmk}

\begin{proof}
The specifications of Algorithm~\ref{algo:isol} ``$\isol$'' implies that $T_1 T_2 =T$ and that $T_1$ is coprime with $T_2$.
Remark~\ref{rmk:isol} implies that $a$ is nilpotent modulo $T_2$ and invertible modulo $T_1$.
\end{proof}

\begin{ex}[Algorithm~\ref{algo:invNil} ``$\invNil$''] 
Let $p,\ q\in k[x]$ be two monic and distinct irreducible polynomials of $k[x]$,
$T = p^2 q^2$. Let $a := p a'$ with $a'\in k[x]$ coprime with $p$ and with $q$.
Then $a$ is invertible modulo $q^2$, and $a$ is   nilpotent modulo $p^2$.
Then $[[( p a' \bmod q^2)^{-1},\ q^2],\ [p a' \bmod p^2 ,\ p^2]] \leftarrow \invNil(a,\ T)$.

With $T = p^2 q^2$ as above and $a = p q$, then $[[0,\ 1 ],\ [ a,\ T]] \leftarrow \invNil(a,\ T)$.

Still with $T = p^2 q^2$, and $a$ invertible modulo $T$ then $[[(a \bmod T)^{-1} ,\ T],\ [0,\ 1]] \leftarrow \invNil(a,\ T)$.
\label{ex:invNil}\end{ex}

\subsection{Monic form according to dynamic evaluation}
The next fundamental operation  makes polynomials monic.
In the case of polynomials over a field, it suffices to invert the leading coefficient.
Modulo a squarefree polynomial $T$, it requires classic dynamic evaluation to handle potential zero-divisors:
start with the leading coefficient, and if there is a zero branch,
carry on with the next coefficient etc.
until there is no zero branch.

The situation is more subtle when $T$ is not squarefree.
This stems from the necessity to perform a Weierstrass factorization to
make a polynomial monic.

\paragraph{Overview of Algorithm~\ref{algo:WFD5} ``$\WFD5$''}
It investigates all the coefficients of $f$ by decreasing degree
through recursive calls,
by {\em splitting} into an invertible branch, and a nilpotent one.
It returns the inverse of the coefficient in the invertible branches,
and continues with the next coefficient in the  nilpotent one.
That the branches do not overlap is due to the specification $\star$ in Algorithm~\ref{algo:invNil},
namely the underlying polynomials $T_1$ and $T_2$ obtained in the splitting
are  coprime.
The output is a compilation of the results obtained at the endpoints of all the branches.
See Example~\ref{ex:WFD5}.

The purpose of the algorithm~\ref{algo:WFD5} ``$\WFD5$'' is to provide the correct input for ``$\MQ$'' algorithm,
like its local counterpart, Algorithm~\ref{algo:WF} ``$\WF$''.
It is a recursive algorithm, the main call being $\WF(f,\ T,\ \deg_y(f),\ T)$.
Recursive calls take as third input an integer smaller than $\deg_y(f)$,
and a polynomial $N$ that divides $T$ as the fourth input.
The lemma below addresses the validity of the algorithm throughout recursive calls,
and Corollary~\ref{cor:SToGB} finalizes the proof of correctness of Algorithm~\ref{algo:SToGB}.

\begin{Lem}\label{lem:induc}
Assume that the algorithm is correct for a fixed polynomial $f = c_0(x)+c_1(x) y+ \cdots+ c_\delta(x) y^\delta$,
and for any  polynomial $T$,  and integers $-1\le d \le d_0$ for a fixed integer $-1\le d_0< \delta$,
and a polynomial $N$, all satisfying the input' specifications.
Then the algorithm is correct with input $f$, any $T$  and $N$ satisfying the input' specifications,
for the integer $d = d_0+1$,
\end{Lem}
\begin{proof}
Let $c_d$ be the coefficient of degree $d$ in $f$ (Line~\ref{WFD5:lc}).
There are four cases to investigate: 

(a) $c_d =0$, return at Line~\ref{WFD5:0}.

(b) $c_d\not=0$, $c_d$ is invertible modulo $T$, that is $T_2=1$. The return occurs at Line~\ref{WFD5:T1}

(c) $c_d\not=0$, $c_d$ is  nilpotent, that is $T_2 = T$ and $T_1 = 1$. The algorithm ends at Line~\ref{WFD5:recT2}.

(d) $c_d\not=0$, $c_d$ has an invertible part and a  nilpotent part,
that is $T_1\not=1$, $T_2\not=1$. Return occurs at Line~\ref{WFD5:T1T2}.

{\em Case (a)}. The algorithm goes directly to Line~\ref{WFD5:0} where  a recursive call is done with input $d-1$ (instead of $d$ in the main call),
hence is correct by assumption: the output $[[f_i,\ T_i,\ d_i,\ N_i]_i]$ verifies the 
specifications. Since $f$ has coefficient of degree $d$ $\ c_d=0$ in this case,
the output is the same with input $d$. Moreover these specifications are unchanged with $d$ or $d-1$.

{\em Case (b)}. The coefficient $c_d$ is invertible modulo $T$. There is no splitting at Line~\ref{WFD5:invNil}.
The output $[[f,\ T,\ d,\ a_1,\ N]]$ at Line~\ref{WFD5:T1}  verifies the specification: since $T_1=T$, $a_1$ is the inverse of the coefficient $c_d$
modulo $T$, $N$
is the gcd with $T$ of the coefficients of degree larger than $d$ by assumption,
and $\sqfp(T)=\sqfp(N)$.

{\em Case (c)}. The coefficient $c_d$ is  nilpotent hence $T_2=T$.
The recursive call made at Line~\ref{WFD5:recT2} outputs
$[[f_i,\ T_i,\ d_i,\ N_i]_i]$ which verifies the specifications with
input $f,\ T,\ d-1,\ \gcd(N,\ c_d) $. 
One easily verifies that this output also complies with the specifications corresponding  to the input $f,\ T, \ d, \ N$.

{\em Case (d)}. There is a splitting into an invertible  branch $[a_1,\ T_1]$
and a  nilpotent one $[a_2,\  T_2]$ (Line~\ref{WFD5:invNil}). In the invertible branch,
the algorithm returns an output similar to Case (b). In the non-invertible
branch, a recursive call  is performed and reduces to Case (c).
The final step  merges these output, which all-in-all
verify the required specifications.
\end{proof}

\begin{Cor}\label{cor:SToGB}
  Algorithm~\ref{algo:WFD5} ``$\WFD5$'' is correct.
  \end{Cor}
\begin{proof}
  The main call is made with $f,\ T,\ \deg_y(f),\ T$.
  The proof proceeds  by induction on $d_0=-1,\ldots, \deg_y(f)$.
 The base case corresponds to  $d_0=-1$ and more precisely to:
 $$
 \WFD5(f,\ T,\ -1, \ \gcd(T,\ {\rm content}(f))).
 $$
 The output at Line~\ref{WFD5:outNil} is then $[f,\ T,\ -1,\ 0,\ \gcd(T,\ {\rm content}(f))]$
 and satisfies the required specifications.

 The induction hypothesis assumes that the algorithm is correct for a fixed integer $-1<d_0 <\deg_y(f)$:
 with input $\WFD5(f,\ T,\ d,\ N)$ where $N=\gcd(T,\ c_{d+1},\ldots, c_\delta)$, $\sqfp(N)=\sqfp(T)$, and $d$ is any integer $d=-1,\ldots, d_0 < \deg_y(f)$, the algorithm is correct.
 Then Lemma~\ref{lem:induc} shows that the algorithm is correct for input $\WFD5(f,\ T,\ d,\ N)$ where $d\le d_0 +1$ achieving the
 proof by induction.\end{proof}

\begin{algorithm}[H]
\DontPrintSemicolon
\KwIn{1. polynomial $f\in k[x,y]$, \ \   $f=c_0(x) + c_1(x) y + \cdots + c_\delta (x) y^\delta$  \\
2. a non constant monic polynomial $T \in k[x]$,\\
3. an integer $d\ge -1$,\\
4. if $d< \delta$ then  $N=\gcd(T,\ c_{d+1},\ldots \ ,\ c_\delta)$.
It must also verify $\sqfp(N)=\sqfp(T)$, that is $c_{d+1},\ldots, c_\delta$  are nilpotent modulo $T$.
And $N=T$ if $d \ge \delta$.}

\KwOut{ 1. $f_i \equiv f \bmod T_i$\\
2. $\prod_i T_i =T$ and the $(T_i)$s are pairwise coprime.\\
3. $d_i\le \delta$ is the largest integer such that $c_{d_i} = \coeff(f_i,\ d_i)$ is invertible modulo $T_i$.
If such an integer does not exist then $d_i=-1$.\\
4. $a_i \equiv \coeff(f_i,d_i)^{-1} \bmod T_i$ if $d_i\ge 0$   and $a_i=0$
if $d_i=-1$.\\
5. $N_i = \gcd(T_i,\ \ c_{d_i+1},\ldots \ ,\ c_\delta)$ is the gcd with $T_i$ of all the coefficients of $f_i$ at a degree 
larger than $d_i$. In particular, 
 $\deg_y(f_i \bmod N_i) = d_i$.
Moreover, $c_{d_i+1},\ldots, c_\delta$ are nilpotent modulo $T_i$, that is $\sqfp(N_i)=\sqfp(T_i)$.
}
\BlankLine

\If(\tcp*[f]{Base case of the recursive calls}){$d = -1$}{
\Return{$[[f,\ T, \ -1,\ 0, \ N]]$\nllabel{WFD5:outNil}\;}
}
$c_d \leftarrow \coeff(f,  d)$\nllabel{WFD5:lc}\;

\eIf{$c_d \not=0$}{
  $[[a_1,\ T_1],\ [a_2,\ T_2]] \leftarrow \invNil(c_d,\ T)$\nllabel{WFD5:invNil}\tcp*[r]{$a_1 c_d \equiv 1 \bmod T_1,\
a_2\equiv c_d \bmod T_2$}
  \eIf{$T_2\not=1$}{
       \eIf(\tcp*[f]{Case (d)}){$T_1\not=1$}
           {  $N_2 \leftarrow \gcd(a_2,\ N,\ T_2)$ \nllabel{WFD5:N2}\tcp*[r]{Now $N_2=\gcd(T_2,\ c_d,\ c_{d+1},\ \ldots, \ c_\delta)$}
             $f_1 \leftarrow f\bmod T_1~~;~~~ f_2 \leftarrow f\bmod T_2$\;
             $N_1 \leftarrow \gcd(N,\ T_1)$\nllabel{WFD5:N1}\tcp*[r]{Now $N_1=\gcd(T_1,\ c_d,\ c_{d+1},\ \ldots, \ c_\delta)$}
             \Return{$[[f_1,\ T_1,\ d, \ a_1, \ N_1]] \cat \WFD5(f_2,\ T_2,\ d-1,\ N_2)$}\nllabel{WFD5:T1T2}}
           (\tcp*[f]{$c_d$ is nilpotent (Case (c))})
           {    $N \leftarrow \gcd(c_d,\ N)$ \tcp*[r]{Now $N=\gcd(T,\ c_d,\ c_{d+1},\ \ldots, \ c_\delta)$}
             \Return{$ \WFD5(f,\ T,\ d-1,\ N)$\nllabel{WFD5:recT2}\tcp*[r]{$T_2=T$}}
           }
  }(\tcp*[f]{$c_d$ is invertible (Case (b))})
      {\Return{$[[f,\ T,\ d, \ a_1,\ N]]$\nllabel{WFD5:T1}\tcp*[r]{$T_1=T$}}}
}(\tcp*[f]{Case (a)})
    {\Return{$\WFD5(f,\ T,\ d-1,\ N)$}\nllabel{WFD5:0}}
\caption{\label{algo:WFD5} \quad $[[f_i,\ T_i,\ d_i,\ a_i,\ N_i]_i] \leftarrow \WFD5(f,\ T,\ d,\ N)$}
\end{algorithm}

\begin{ex}[Algorithm~\ref{algo:WFD5} ``$\WFD5$'']\label{ex:WFD5} Consider the polynomial
  $f=x(x+1) y^2 + 2 x y - x-1$ and $T = x^2(x+1)^2$. The main call is
  $$
  \WFD5(f,\ T,\ 2,\ T)\quad \text{the third input is $2$ because $f$ is of degree $2$}
  $$
  First coefficient, that  of $y^2$: $c_2(x)=x (x + 1)$. We have 
$$[\, [ 0,\ 1],\ [ x(x+1),\ x^2(x+1)^2]\, ] \leftarrow \invNil(c_2,\ T)$$
  since $c_2$ is nilpotent modulo $T$. Thus $T_1=1$, $T_2=T$ and $N=\gcd(T,\ c_2) = x(x+1)$.

  Recursive call at Line~\ref{WFD5:recT2} is : $\WFD5(f,\ x^2 (x+1)^2,\ 1,\ x (x+1))$.

  Second coefficient, that of $y$ is: $c_1(x) = 2 x$. 
$$[ [- \frac 12 x -1   ,\ (x+1)^2],\ [ 2 x ,\ x^2]] \leftarrow \invNil(2 x,\ x^2 (x+1)^2),$$
    thus $T_1=(x+1)^2$ and $T_2=x^2$. Thus at Line~\ref{WFD5:N2}-\ref{WFD5:N1} $N_1 = \gcd(x(x+1),\ (x+1)^2) = x+1$ and $N_2 = \gcd(x(x+1), x^2,\ 2 x) = x$.
    Next consider Line~\ref{WFD5:T1T2}. We obtain 
$$\out=[[ (- 1 - x ) y^2  + 2 x y - x -1 \ ,\ \  (x+1)^2\ ,\ \ 1\ ,\ \ -1 - \frac 1 2 x\ ,\ \ x+1]]$$
    that is supplemented with the recursive call $\WFD5(x y^2 + 2 x y - x -1 ,\ x^2,\ 0,\ x)$.

    Third coefficient is $c_0 \equiv -x -1 \bmod x^2$.
    $[ [ -1 +x,\ x^2],\ [ 0,\ 1]] \leftarrow \invNil( - x -1,\ x^2) $,
    thus $T_1=x^2$ and $T_2=1$.  The return occurs at Line~\ref{WFD5:T1} 
    with $[x y^2 + 2 x y - x -1 \ ,\ \ x^2\ ,\ \ 0\ ,\ \ -1+x\ ,\ \ x]$.
    Finally, $\out=[[f_1,\ T_1,\ d_1,\ a_1,\ N_1],\ [f_2,\ T_2,\ d_2,\ a_2,\ N_2]]$ where:
    $$
    \begin{array}{ccc}
      \begin{array}{rcl}
        f_1 & =& (- 1 - x ) y^2  + 2 x y - x -1\\
        T_1 & =& (x+1)^2\\
        d_1 & =& 1\\
        a_1  & = & \ -1 - \frac 1 2 x\\
        N_1  & = & x+1
        \end{array}
      & \qquad &
      \begin{array}{rcl}
         f_2 & =& x y^2 + 2 x y - x -1 \\
         T_2 & =&  x^2\\
         d_2  &  = & 0\\
         a_2 & = & -1+x\\
         N_2  & =& x
      \end{array}
    \end{array}
    $$
    We have indeed $T=T_1 T_2$, $d_1=1$, and $d_2=0$.  Also, for $i=1,2$, $\deg_y(f_i \bmod N_i)=d_i$,
    and $\sqfp(N_i)=\sqfp(T_i)$.
 Moreover, $a_i \coeff(f_i,\ d_i) \equiv 1 \bmod T_i$.  So that all specifications are satisfied.\hfill \qed
  \end{ex}

\paragraph{Overview of Algorithm~\ref{algo:MFD5} ``$\MFD5$''}
Again it translates the algorithm~\ref{algo:MF} ``$\MF$'' to dynamic evaluation.
Among the two subroutines involved, only $\WFD5$ creates splittings,
indeed $\MQ$ does not perform divisions. Note that the description of ``$\MQ$'' in Algorithm~\ref{algo:MQ}
does not assume $T$ to be the power of an irreducible polynomial. There is no need to adapt it to dynamic evaluation.

After a first call to $\WFD5$ at Line~\ref{MFD5:WF1},
$\bfW$ is a collection of data like the output of the local version $\WF$.
For each component $W=[f_i,\ T_i,\ d_i,\ \alpha_i,\ N_i]$ of $\bfW$
two cases must be distinguished: if $d_i=-1$ then $f_i$ is nilpotent and $N_i | f_i$,
or $d_i\ge 0$. In the former, a second call to $\WFD5$ (Line~\ref{MFD5:WF2})
is performed after dividing the input by $N_i$. 
This second call creates subbranches requiring to ``project'' some polynomials on
these new factors (role of the call to $\isol$ at Line~\ref{MFD5:isol}).

\begin{algorithm}[h]
\DontPrintSemicolon
\KwIn{1. polynomial $a\in k[x,y]$\\
  2. monic non constant polynomial $T\in k[x]$
}
\KwOut{
  1. $b_i \in k[x,y]$  is monic (in $y$) \\
  2. monic polynomials $T_{i}\in k[x]$, pairwise coprime dividing $T$.\\
  3. monic polynomials $U_i \in k[x]$, pairwise coprime dividing  $T$. Furthermore:
  \begin{multicols}{2}
    \begin{enumerate}[(a)]
 \item\label{MFD5:a} $ \prod\nolimits_{i} U_i T_{i} = T $ and,
 \item\label{MFD5:b} if $U_i\not=1$, $\sqfp(U_i)=\sqfp(T_{i})$ 
 \item\label{MFD5:c} $U_i$ divides $a$ modulo $T_{i}$
 \item\label{MFD5:d} $\prod\nolimits_i \l U_i   b_{i},\ U_i T_{i}\r \simeq \l a,\ T\r$.
    \end{enumerate}
    \end{multicols}
}
\BlankLine
 $\out\leftarrow []$\;

${\bf W} \leftarrow \WFD5(a,\ T,\ \deg_y(a),\ T)$\nllabel{MFD5:WF1}\;

 \For(\tcp*[f]{write $W=[f_i,\ T_i,\ d_i,\ \alpha_i,\ N_i]$}){$W \in {\bf W}$}{

   \eIf(\tcp*[f]{$f_i$ is nilpotent and $N_i|T$, $N_i|f_i$}){$d_i=-1$\nllabel{MFD5:di=-1}}{
     
     $\bfZ \leftarrow \WFD5(f_i/N_i,\ T_i/N_i, \ \deg_y(f_i),\ T_i/N_i)$\nllabel{MFD5:WF2}

     $U'_{i0} \leftarrow N_i$\nllabel{MFD5:U0}\;
     
     \For(\tcp*[f]{write $Z=[f_{ij},\ T_{ij},\ d_{ij},\ \alpha_{ij},\ N_{ij}]$}){$Z\in \bfZ$}{
       
           $b_{ij} \leftarrow \MQ(f_{ij},\ T_{ij}, \ d_{ij},\ \alpha_{ij}, \ N_{ij})$\nllabel{MFD5:MQ1}\tcp*[r]{$f_{ij}$ is not nilpotent, $d_{ij}\ge0$}
       
           $(\, U_{ij},\ U'_{ij} \, )\leftarrow \isol(U'_{i j-1},\ T_{ij})$\nllabel{MFD5:isol}\tcp*[r]{$U_{ij} U'_{ij}=U'_{ij-1}$, $ U_{ij} = N_i^{(T_{ij})}$}

       $\out\leftarrow \out  \cat [\ [b_{ij},\ T_{ij},\ U_{ij}]\ ]$ \nllabel{MFD5:out1}\;
     }
   }(\tcp*[f]{$f_i$ is not nilpotent}\nllabel{MFD5:else})
   {
     
     $b_i \leftarrow \MQ(f_i,\ T_i, \ d_i,\ \alpha_i, \ N_i)$\nllabel{MFD5:MQ2}\;

   $\out\leftarrow \out  \cat [\ [b_i, \ T_i,\ 1]\ ] $ \nllabel{MFD5:out2} \;
   }
 }

 \Return{$ \out$}\;

\caption{\label{algo:MFD5} \quad $ [ [b_i,\ T_i,\ U_i]_i ] \leftarrow \MFD5(a,\ T)$}
\end{algorithm}

\begin{Prop}[Correctness of Algorithm~\ref{algo:MFD5} ``$\MFD5$''] \label{prop:MFD5}
 Specifications~\ref{MFD5:a},~\ref{MFD5:b},~\ref{MFD5:c} and~\ref{MFD5:d}
  are satisfied by the output $\out$ of the algorithm.
\end{Prop}

\begin{proof} Consider $W\in \bfW$  at Line~\ref{MFD5:WF1}, written $W = [f_i,\ T_i,\ d_i,\ \alpha_i,\ N_i]$.
The specifications 1-5. of
Algorithm~\ref{algo:WFD5} ``\WFD5'' give:

1'. $\l a,\ T_i \r = \l f_i,\ T_i\r$,

2'. $\prod_i T_i =T$ and the $T_i$s are pairwise coprime

3'. $d_i=-1$ (Line~\ref{MFD5:di=-1}) or $d_i \ge 0$ (Line~\ref{MFD5:else})

4'. (not important)

5'. $N_i=\gcd(T_i,\ {\rm coeffs.~of~} f_i {\rm~of~degree~} >  d_i)$,
and $\sqfp(N_i)=\sqfp(T_i)$.
\medskip

{\em Case $d_i \ge 0$}. This concerns Line~\ref{MFD5:else} and onward.
This means that $f_i$ is not nilpotent modulo $N_i$ and  $\deg_y(f \bmod N_i) = d_i$.
  The output $[f_i,\ T_i,\ d_i,\ \alpha_i,\ N_i] = W\in \bfW$ of the
  first call to $\WFD5$ at Line~\ref{MFD5:WF1} is then
  ready to be used by Algorithm $\MQ$ at Line~\ref{MFD5:MQ2} (its specifications being satisfied).
  We obtain $\l b_i,\ T_i \r = \l f_i,\ T_i\r$ with $b_i$ monic.
  And consequently, by the Chinese remainder theorem and the specifications 1'-2'. above:
$$
\prod_{i,\ d_i \ge 0} \l b_i,\ T_i \r 
= \prod_{i,\ d_i \ge 0} \l f_i,\ T_i \r
= \prod_{i,\ d_i \ge 0} \l a,\ T_i \r
\simeq \l a,\, \prod\nolimits_{i,\ d_i \ge 0}  T_i\r.
$$
Besides, note that the third component of the list $[b_i,\ T_i,\ 1]$ added to $\out$ at Line~\ref{MFD5:out2}
is $1$, which means that  $U_i=1$. Specification~\ref{MFD5:c} is clearly true
while Specification~\ref{MFD5:b} is void. Moreover:
\begin{equation}\label{eq:CRT1}
\prod_{i,\ d_i \ge 0}  \l U_i b_i,\  U_i T_i \r \simeq \l a,\, \prod\nolimits_{i,\ d_i \ge 0} T_i\r.
\end{equation}

{\em Case $d_i=-1$}. This concerns Lines~\ref{MFD5:di=-1}-\ref{MFD5:out1}.
We then have  $N_i | f_i$ and $N_i=\gcd(T_i,\ {\rm content}(f_i))$.
By specification 5'. above,  $\sqfp(N_i)=\sqfp(T_i)$.
  Thus $\gcd(T_i, {\rm content}(\frac{f_i}{N_i}))=1$. Consequently,
\begin{equation}\label{eq:(ddagger)}
\text{for any}\quad  T_{ij} \quad \text{dividing} \quad T_i
\quad \text{we have}\quad  \gcd\left( T_{ij}, {\rm content}\left(\frac{f_i}{N_i}\right)\right) = 1.
\end{equation}
Consider now Line~\ref{MFD5:WF2}
$\bfZ \leftarrow \WFD5(f_i/N_i,\ T_i/N_i,\ \deg_y(f_i),\ T_i/N_i)$
as well as a 
 component $Z\in \bfZ$, written $Z=[f_{ij},\ T_{ij},\  d_{ij},\ \alpha_{ij},\  N_{ij}]$.
It verifies the specifications 1.-5. of Algorithm~\ref{algo:MQ}.
\smallskip

1''. $\l f_i / N_i,\ T_{ij} \r = \l f_{ij},\ T_{ij}\r$,

2''. $\prod_j T_{ij} =T_i/N_i$ and the $(T_{ij})_j$s are pairwise coprime

3''. $d_{ij}\ge0 $ 

4''. (not important)

5''. $N_{ij}=\gcd(T_{ij},\ {\rm coeffs.~of~} f_{ij} {\rm~of~degree~} > d_{ij})$, 
and $\sqfp(N_{ij})=\sqfp(T_{ij})$.

\smallskip

\noindent Indeed,   $d_{ij} \ge 0$ holds. Otherwise 
$N_{ij} = \gcd(T_{ij}, {\rm content}(\frac{f_i}{N_i}))\not=1$ would  divide $f_{ij}$, excluded as seen above 
in Eq~\eqref{eq:(ddagger)}.
  Therefore, the component $Z$ is ready to be used by Algorithm~\ref{algo:MQ} ``$\MQ$'' at Line~\ref{MFD5:MQ1},
  its input' specifications being satisfied by $f_{ij},\ T_{ij},\ d_{ij},\ \alpha_{ij},\ N_{ij}$. Its output $b_{ij}$ 
  hence verifies $\l b_{ij},\ T_{ij}\r =\l f_{ij}, \ T_{ij}\r$ with $b_{ij}$ monic.
  Specifications 1''-2''. imply, with the Chinese remainder theorem:
\begin{equation}\label{eq:CRT2}
\prod_j \l b_{ij},\ T_{ij}\r=  \prod_j \l f_{ij}, \ T_{ij}\r = \prod_j \l f_i/N_i, \ T_{ij}\r \simeq  \l f_i/N_i,\ T_i/N_i\r.
\end{equation}
Eq.~\eqref{eq:CRT2} then implies  with the Chinese remainder theorem according to Specifications 1'-2'.:
\begin{equation}\label{eq:CRT3}
  \prod_{i, \ d_i =-1} (N_i \prod_j \l  b_{ij},\  T_{ij} \r) \simeq 
  \prod_{i, \ d_i =-1} (N_i \l f_i/N_i,\ T_i/N_i\r )
  \simeq \prod_{i, \ d_i =-1} \l f_i,\ T_i \r
\simeq \prod_{i, \ d_i =-1} \l a,\ T_i \r 
\end{equation}
Line~\ref{MFD5:isol} introduces the $T_{ij}$-component (Definition~\ref{def:isol}) of $U'_{ij-1}$:
$$
\quad (\, U_{ij},\ U_{ij}'\, )  \leftarrow \isol(U'_{ij-1},\ T_{ij}),\quad \text{with} \quad U_{ij}= U_{ij-1}'^{(T_{ij})}\ \text{and} \ U'_{i0}=N_i \ \text{(Line~\ref{MFD5:U0})}.
$$
The specifications of $\isol$ tell
that $U_{ij} U'_{ij} = U_{ij-1}'$ and $U_{ij}$ is coprime with $ U_{ij}'$.
We see then that 
\begin{equation}\label{eq:factNi}
N_i= U_{i0}' = U_{i1} U_{i1}' = U_{i1} U_{i2} U_{i2}' = \cdots = U_{i1} \cdots U_{ij} U_{ij}', \ \text{ for all } j\ge 1.
\end{equation}
Moreover the polynomials in the product are pairwise coprime. By definition, $U_{ij} = U_{ij-1}'^{(T_{ij})} = \prod_{p\in \calP, \ v_p(T_{ij})>0} p^{v_p(U_{ij-1}')}$.
Since $U_{ij-1}'$ is a  factor in the factorization of $N_i$  of Eq.~\eqref{eq:factNi},  we deduce that
 $[ v_p(T_{ij}) > 0 \Rightarrow v_p(U_{ij-1}') =v_p(N_i)]$, and thus that $U_{ij} = N_i^{(T_{ij})}$. 

On the other hand, the specification 2''. above provides $N_i \prod_j T_{ij} = T_i$ with the $(T_{ij})_j$s coprime,
and the specification 5'. gives $\sqfp(T_i) = \sqfp(N_i)$.
It follows that
\begin{equation}\label{eq:Ni}
\prod_{j} U_{ij} = \prod_j \prod_{p\in \calP,\ v_p(T_{ij})>0} p^{v_p(N_i)} =  \prod_{p\in \calP,\ v_p(T_i)>0} p^{v_p(N_i)}
=\prod_{p\in \calP,\ v_p(N_i)>0} p^{v_p(N_i)} = N_i,
\end{equation}
as well as  $\sqfp(U_{ij})=\sqfp(T_{ij})$ if $U_{ij}\not=1$. This proves
Specifications~\ref{MFD5:b} in the case $d_i=-1$.

We have seen that $N_i|f_i$, hence $f_i\in\l N_i\r$.
Specification~1'. thus implies that $a\in \l N_i,\ T_i\r$. By Eq.~\eqref{eq:Ni}
$U_{ij}$ divides $N_i$ and since $T_{ij}$ divides $T_i$ we obtain that $a\in \l U_{ij},\ T_{ij}\r$, or equivalently that $U_{ij}$ divides $a$ modulo $T_{ij}$.
This proves Specification~\ref{MFD5:c} in the case $d_i=-1$.

In Eq.~\eqref{eq:CRT3}, substituting the $N_i$s by $\prod_j U_{ij}$s from Eq.~\eqref{eq:Ni},
and combining with
Eq.~\eqref{eq:CRT1} yield:
$$
\prod_{i, \ d_i =-1} \prod_j \l  U_{ij} b_{ij},\  U_{ij} T_{ij} \r  \cdot \prod_{i,\ d_i\ge 0} \l U_i b_{i},\ U_i T_{i}\r 
\simeq \prod_{i, \ d_i =-1 } \l a,\ T_i \r  \cdot \prod_{i,\ d_i\ge 0} \l a  ,\ T_i \r
\simeq \l a, \ T\r.
$$
This proves Specification~\ref{MFD5:d} in all cases.

Moreover Specification 2''. above  implies that  $ \prod_{i, \ d_i=-1} \prod_j U_{ij}  T_{ij} = \prod_{i, \ d_i=-1} T_i$ with the $(T_{ij})_j$s  pairwise coprime.
With Specification 2'. above $\prod_i T_i = T$, we get 
$$
\prod_{i,\ d_i=-1}   \prod_j  U_{ij} T_{ij} \cdot \prod_{i,\ d_i \ge 0}  U_i T_i =  T.
$$
This proves the specification~\ref{MFD5:a} (in all cases).
\end{proof}

\section{Computation of lexGb through dynamic evaluation\label{sec:lexGbD5}} 
\subsection{Subresultant p.r.s}
With the ability to make monic nilpotent polynomials
occurring in the subresultant p.r.s. of $a$ and $b$ modulo $T$,
we are ready to generalize  the algorithm~\ref{algo:LNL} ``$\LNN$''.
The skeleton is the same as that of Algorithm~\ref{algo:LNL}:
find the ``last non nilpotent'' and the ``first  nilpotent'' polynomials in the modified subresultant p.r.s.
The difference lies in the management of the splittings arising when  making polynomials monic
at Line~\ref{LNL:MF1}, and from the recursive call at Lines~\ref{LNL:ret2}.
The output is a family of  objects each similar to the output of the local version Algorithm~\ref{algo:LNL}.

\begin{algorithm}[h]
\DontPrintSemicolon
\KwIn{1. polynomial $f_1\in k[x,y]$ which has an invertible leading coefficient modulo $T$.

2. polynomial $f_2 \in k[x,y]$,  satisfies Assumption~\eqref{tag:H}
(in particular  is not zero).

3.  $T$ is a monic polynomial $T\in k[x]$
}
\KwOut{1-2. monic (in $y$) polynomials $a_i, \ b_i\in k[x,y]$
  verifying the degree condition of Proposition~\ref{prop:LNLD5}.

3.-4. monic polynomials $U_i,T_i\in k[x]$ such that the family $(T_i U_i)_i$ is pairwise coprime,
and if $T_i \not= 1$ then $\sqfp(T_i) = \sqfp(U_i)$.

  \begin{equation}\label{eq:LNLD5}
  \prod\nolimits_i T_i U_i = T,\qquad \l f_1, \ f_2 , \ T\r\simeq \prod\nolimits_i \l a_i,\ T_i b_i,\ U_i T_i\r
  \end{equation}
}

\BlankLine

$\out\leftarrow []$\;

\Repeat(\nllabel{LNLD5:prs})
{
  $\lc(f_i)$ is {\em not} invertible modulo $T$ or $f_{\ell+1}=0$
}
{Compute the subresultant p.r.s $f_1,\ f_2,\ldots,\ f_\ell,\ f_{\ell+1}$  of $f_1$ and $f_2$ modulo $T$
(with $f_{\ell+1}=0$, following Proposition~\ref{prop:prs}).
}

\If(\tcp*[f]{$\lc(f_j)$ is invertible modulo $T$ for $j <i$})
   {$\lc(f_i)$ is not invertible modulo $T$\nllabel{LNLD5:if}}{
     $a \leftarrow (\lc(f_{i-1})^{-1} \bmod T) f_{i-1} \bmod T$\nllabel{LNLD5:a}     \tcp*[r]{$a$ is monic and $\l a,\ T\r =\l f_{i-1},\ T\r$}
     $ \bMb \leftarrow \MFD5(f_i,\ T)$ \nllabel{LNLD5:MF1}
\tcp*[r]{Write $\bMb = [[b_j,\ U_j,\ T_j]_j]$}
     \For
{$[b_j,\ U_j,\ T_j] \in \bMb$}{
       \eIf(\tcp*[f]{$f_i$ is nilpotent mod $U_j T_j$}){$T_j\not=1$\nllabel{LNLD5:if2}}{
        $\out = \out \cat  [[a,\ b_j,\ T_j,\ U_j]] $ \nllabel{LNLD5:ret1}\;
       }(\tcp*[f]{$f_{i}$ is not nilpotent modulo $U_j T_j$}\nllabel{LNLD5:else})
           {$\out = \out \cat \LNND5(a,\ b_j,\ U_j)$\nllabel{LNLD5:rec}}
     }
     \Return{$\out$}\nllabel{LNLD5:ret2}
   }

   \tcp*[f]{Here, all $\lc(f_1),\ldots, \lc(f_\ell)$ are invertible modulo $T$, and $f_{\ell+1}=0$}\;
   \Return{$[[(\lc(f_\ell) \bmod T)^{-1} f_\ell ,\ 0,\ 1,\ T]]$\nllabel{LNLD5:ret3}}

\caption{\label{algo:LNLD5} \quad $[[a_i,\ b_i,\ T_i,\ U_i]_i] \leftarrow \LNND5(f_1,\ f_2,\ T)$}
\end{algorithm}

\begin{ex}[Algorithm~\ref{algo:LNLD5} ``$\LNND5$'']\label{ex:LNLD5}
Consider $T=x^2 (x+1)^2$, and $a= y^3 + (x-1) y^2 + x(x+1) y +x$ and $b=y^3 - y^2 + y-x$.
The subresultant of degree 2 is $S_2(a,\ b) =  x y^2  + ( x^2 + x-1 )y + 2 x$.
Since $\lc(S_2) = x$ is not invertible modulo $T$, the algorithm calls $\MFD5( S_2, \ T)$ at Line~\ref{LNLD5:MF1}.
It outputs two branches:
$$
[ [ y^2 + (2 x+3) y + 2,\ (x+1)^2, \ 1 ],\ [ y- 2 x ,\ x^2,\ 1 ] ]
\leftarrow
\MFD5( x y^2  + ( x^2 + x-1 ) y + 2 x,\ T)
$$
Recursive calls are then performed on each of these two branches:
\smallskip

1st branch: $[[1,\ 0,\ 1, \ (x+1)^2]] \leftarrow \LNND5(b,\ y^2 + (2 x+3) y + 2,\ (x+1)^2) $.

2nd branch: $[ [ y- 2 x,\ 1,\ x,\ x]] \leftarrow \LNND5(b,\ y - 2 x ,\ (x+1)^2)$.
\smallskip

\noindent Finally: $[[1,\ 0,\ 1, \ (x+1)^2],
  \  [ y- 2 x,\ 1,\ x,\ x]] \leftarrow \LNND5(a,\ b,\ T)$,
meaning that $\l a,\ b,\ T\r \simeq \l 1,\ 0,\ (x+1)^2\r \cdot \l y- 2 x,\ x,\ x^2\r$ (actually, the latter product of ideals is equal to $\l y,\ x\r$).
\qed  \end{ex}

\begin{Prop}[Correctness of Algorithm~\ref{algo:LNLD5} ``$\LNND5$'']\label{prop:LNLD5}
The family $[[a_i,\ b_i,\ T_i,\ U_i]_i]$ output by $\LNND5(f_1,\ f_2,\ T)$  satisfies the specifications 1-4 and Eq.~\eqref{eq:LNLD5}.

We also have the {\em degree conditions:} For all $j$,
\begin{equation}\label{eq:deg-cond-D5}
\deg_y(a_j)\le \deg_y(f_1), \quad \text{and if}\quad b_j\not=0,
\quad 
\deg_y(b_j) < \deg_y(f_2),
\quad \deg_y(b_j) < \deg_y(a_j)
\end{equation}
Moreover $[\deg_y(a_j) = \deg_y(f_1)] \Rightarrow [\deg_y(f_1)=\deg_y(f_2)]$. 
\end{Prop}
\begin{proof}
We separate the proof of correctness, from the proof 
of the degree conditions, treated after.

{\em Proof of correctness:} We investigate the two
returns at Line~\ref{LNLD5:ret3} and at Line~\ref{LNLD5:ret2}.

  {\em Case of return at Line~\ref{LNLD5:ret3}}. The subresultant p.r.s. was
  computed modulo $T$ without failure until to get a zero $f_{\ell+1}=0$.
  By Corollary~\ref{cor:prs}, $\l f_1,\ f_2,\ T \r = \l f_\ell,\ f_{\ell+1},\ T\r$,
  whence the output $[[ u, \ v, \  1,\ T ]]$  with $u\equiv \lc(f_\ell)^{-1} f_\ell \bmod T$,
  $v=f_{\ell+1}=0$ verifies $\l u,\ v, \ T\r =\l f_1,\ f_2,\ T\r$. This equality is a special easy case of Eq.~\eqref{eq:LNLD5} when the output has one component only.

  {\em Case of return at Line~\ref{LNLD5:ret2} }. The subresultant p.r.s $f_1,\ f_2,\ \ldots,$ was correctly computed until $f_i$,  $\lc(f_i)$ being not invertible modulo $T$.
  Since $f_{i-1}$ passed the if-test at Line~\ref{LNLD5:if}, it has an invertible leading coefficient and $a\equiv \lc(f_{i-1})^{-1} f_{i-1} \bmod T$ makes sense at Line~\ref{LNLD5:a}
  and $\l f_{i-1},\ T\r = \l a, \ T\r$.
  
Algorithm~\ref{algo:MFD5} ``$\MFD5$''  is then called at Line~\ref{LNLD5:MF1}
  to ``make''  $f_{i}$ monic. The output $\bMb$ is a list of lists of three polynomials, 
say  $[[b_j,\ T_j,\ U_j]_j]$
verifying
$\l T_j b_j,\ U_j T_j\r = \l f_i,\ U_j T_j\r $
and $\prod_j U_j T_j = T$, with $b_j$ monic; Also the polynomials $(U_j T_j)_j$s are pairwise coprime and $\sqfp(U_j)=\sqfp(T_j)$ if $T_j\not=1$. 
Next two cases occur: either $f_i$ is nilpotent modulo $U_j T_j$ (Line~\ref{LNLD5:if2}) or not (Line~\ref{LNLD5:else}).

In the first case, $f_{i-1}$ is the last non-nilpotent and $f_i$ is the first nilpotent polynomial met in the modified p.r.s.
This is the expected result,
thus the component $[a,\ b_j, \ T_j,\ U_j]$ is added to the output $\out$.
With Corollary~\ref{cor:prs}, we obtain:
 \begin{equation}\label{eq:dagger}
\l f_1,\ f_2, \ U_j T_j\r = \l f_{i-1},\ f_i,\ U_j T_j\r 
= \l a,\ T_j b_j,\ T_j U_j\r
\end{equation}

 In the second case, $T_j=1$, $f_i$ is not nilpotent modulo $U_j T_j$.
 Since  $\lc(f_i)$ is not invertible modulo $T$ (Line~\ref{LNLD5:if}),
 all monic forms $(b_j)_j$ verify $\deg_y(b_j)<\deg_y(f_i)$. Besides,
 $\deg_y(a) = \deg_y(f_{i-1})\le \deg_y(f_1)$ and $\deg_y(U_j) \le \deg_y(T)$.
 Therefore, the input $(a,\ b_j,\ U_j)$ of the recursive call at Line~\ref{LNLD5:rec}
 displays a strict degree decrease compared to the main call with input
 $(f_1,\ f_2,\ T)$.
We can then assume by induction that the output is correct. Write
  $$
  [[u_{j n},\ v_{j n},\ V_{j n},\ W_{j n}]_n] \leftarrow  \LNND5(a,\ b_j,\ U_j).
  $$
  We have
  $\prod_n \l u_{j n},\ V_{j n} v_{j n},\ W_{j n} V_{j n}\r \simeq \l a,\ b_j ,\  U_j\r$.
  Besides,  from the equalities
  $$
  \l a,\ T\r =\l f_{i-1},\   T\r,\qquad \l T_j b_j,\  T_j U_j\r = \l f_i,\ U_j T_j\r,
  \qquad T_j=1, \quad \text{and} \quad U_j|T
  $$
  in this case, we get: $ \l a,\ b_j ,\  U_j\r  =\l f_{i-1},\ f_i,\ U_j\r$.
  Corollary~\ref{cor:prs} provides the equality of ideals  $\l a,\ b_j ,\  U_j\r  = \l f_1,\ f_2, \ U_j\r$, and we have:
  \begin{equation}\label{eq:ddagger}
    \prod_n \l u_{j n}, V_{j n} v_{j n}, \ W_{j n} V_{j n}\r \simeq \l f_1,\ f_2, \ U_j\r.
  \end{equation}
  Recall that $\prod_j T_j U_j =T$ and that the $(U_j T_j)_j$ are pairwise coprime.
  The Chinese remainder theorem applied to Eq.~\eqref{eq:dagger} when $T_j\not=1$ and to Eq.~\eqref{eq:ddagger} when $T_j=1$ gives:
\begin{equation}\label{eq:LNLD5iso}
\prod_{j,\ T_j\not=1} \l a,\ T_j b_j,\ T_j U_j\r \ \cdot\ \prod_{j,\ T_j=1} \prod_n \l u_{j n},\ V_{j n} v_{j n},\ W_{j n } V_{j n }\r \simeq
\l f_1, \ f_2,\ T\r
\end{equation}
The final output of the algorithm is then 
$$
\out=\ \underset{j,\ T_j\not=1}{\cat}\  [[ a,\ b_j,\ T_j,\ U_j]]\  \underset{j,\ T_j=1}{\cat}\ \  \underset{n}{\cat}\ [[u_{j n},\ v_{j n},\ V_{j n} ,\ W_{j n }]].
$$
And with these notations, the isomorphism~\eqref{eq:LNLD5iso}
above is precisely Eq.~\eqref{eq:LNLD5}. This achieves the proof of correctness of the algorithm in all cases.
\medskip

{\em Proof of the degree condition.} Consider one component $[a_j,\ b_j,\ U_j]$.
Assume that the total number of pseudo-divisions (including those occurring in recursive calls) is at least two.
This implies at least two strict degree decreases, and since
$a_j$ and $b_j$ are the monic form of the two last polynomials in the modified subresultant p.r.s, this implies:
$$
\deg_y(a_j) < \deg_y(f_1), \quad \text{and if}\quad b_j\not=0\quad \deg_y(b_j)<\deg_y(f_2),\ \ \text{and}\ \   \deg_y(b_j)<\deg_y(a_j),
$$
proving~\eqref{eq:deg-cond-D5} in case of more than one pseudo-division.
 
If only one pseudo-division occurs in the algorithm,
then two possibilities may happen:
\begin{enumerate}[-1- ]
\item\label{enum:-1-bis} either $\lc(f_2)$ is invertible modulo $T$ and $f_3=\pm \prem(f_1,\ f_2)$
\item\label{enum:-2-bis} or not and then, let $[[b_j,\ U_j,\ T_j]_j]$
  be the monic form of $f_2$ (Line~\ref{LNLD5:MF1}),
  and $a$ that of $f_1$ (Line~\ref{LNLD5:a}).
  According to Assumption~\eqref{tag:H}, $f_2$ is not nilpotent modulo $U_j T_j$ for any $j$: the algorithm then never goes through the lines~\ref{LNLD5:if2}-\ref{LNLD5:ret1}. At Line~\ref{LNLD5:rec}, a recursive
  call occurs: ``$\LNND5(a,\ b_j, U_j)$'', with $a$ and $b_j$ monic.
  Inside this recursive call, another pseudo-division takes place, a contradiction.
\end{enumerate}
Only Case~\ref{enum:-1-bis} can happen. Let $a$ be the monic form of $f_2$ (Line~\ref{LNLD5:a}).
Write the monic form of $f_3$ at Line~\ref{LNLD5:MF1} as $[[b_j,\ U_j,\ T_j]_j]$.
$$
\text{if} \quad b_j\not=0,\qquad \deg_y(b_j)\le \deg_y(f_3)<\deg_y(f_2)=\deg_y(a)\le \deg_y(f_1).
$$
This proves Eq~\eqref{eq:deg-cond-D5}.
For $\deg_y(a) = \deg_y(f_1)$ to hold, necessarily
$\deg_y(f_1) = \deg_y(f_2)$. This achieves the proof
of the degree condition in the case of one pseudo-division.

If no pseudo-division at all occur, then $\lc(f_2)$ is not invertible modulo $T$.
Moreover the recursive call ``$\LNND5(a,\ b_j,\ U_j)$'' at Line~\ref{LNLD5:rec} does not involve pseudo-division neither.
But since $a$ and $b_j$ are monic, this cannot be the case. Thus the algorithm
never reaches Line~\ref{LNLD5:rec}. It cannot reach the lines~\ref{LNLD5:if2}-\ref{LNLD5:ret1} neither: this would mean that $f_2$ is nilpotent
modulo $U_i T_i$, excluded by Assumption~\eqref{tag:H}. The situation
of no pseudo-division does not occur.
\end{proof}

\begin{Rmk}\label{rmk:LNLD5=1}
A component $[a_i,\ b_i,\ T_i,\ U_i]$ of the output verifies $T_i=1$ if and only if
$b_i=0$. Indeed, this occurs only at an exit Line~\ref{LNLD5:ret3}. The other possibility at Line~\ref{LNLD5:ret1} is made
under the condition $T_i\not=1$ (Line~\ref{LNLD5:if2}). At Line~\ref{LNLD5:ret1},  we have
$b_j\not=0$, as an output of $\MFD5$ with input a non-zero polynomial.
\end{Rmk}
\subsection{Family of minimal lexicographic Gr\"obner bases from the subresultant p.r.s}
We focus now on translating Algorithm~\ref{algo:SToGB} ``$\SToGB$'' to dynamic evaluation. A direct consequence is that the output
is no more one lexGb, but a family of thereof.

\paragraph{Overview of Algorithm~\ref{algo:SToGBD5} ``$\SToGBD5$''}
Naively, it suffices to track divisions in the local version Algorithm~\ref{algo:SToGB} ``$\SToGB$'' and to create splittings accordingly.
A slight difficulty occurs in that the ``$\LNND5(a,\ b,\ T)$'' algorithm assumes that $a$ has an invertible leading coefficient  modulo $T$.
This assumption is too restrictive in many cases: there are some local components of $T$ over which this assumption is verified and others over which it is not.
To circumvent this, it suffices to ``make monic'' $a$ with a call to ``$\MFD5$'' (Line~\ref{SToGBD5:a}) and proceed with this output.
To clarify the exposition, an auxiliary algorithm~\ref{algo:SToGBA} ``$\SToGBA$'' first treats the case of a polynomial $a$ that has an invertible
leading coefficient modulo $T$, which becomes a subroutine in the main algorithm~\ref{algo:SToGBD5}.

However, it is assumed that both
$a$ and $b$ satisfy Assumption~\eqref{tag:H}.
It says that for any primary factor $p^e$ of $T$, $a$ and $b$ are not nilpotent modulo $p^e$.
This assumption can likely be removed; we let if for future work.

\begin{algorithm}[!h]
\SetAlgoVlined 
\DontPrintSemicolon 
\KwIn{ 1-2. polynomials $a ,\ b\in k[x,y]$, $\deg_y(a)\ge \deg_y(b)$. The leading coefficient of $a$ is invertible modulo $T$.

 3. $T\in k[x]$ is a monic polynomial.}

\KwOut{Family of minimal lexGbs $\bfG=(\calG_i)_i$ written (following Theorem~\ref{th:Laz}) as in Eq.~\eqref{SToGBA:specif}
  and verifying the specifications (i)-(ii)-(iii) of Proposition~\ref{prop:SToGBA}
  \vspace{-2ex}

  \begin{equation}\label{SToGBA:specif}
    \calG_i = [h_{i1},\ h_{i2} f_{i2},\ \ldots,\ h_{i \mu-1} f_{i \mu -1},\ f_{i \mu}],\quad h_{ij} \in k[x],\ f_{ij} \in  k[x,y] {\rm~monic~in~} y,
  \end{equation}
}

\BlankLine

\If{$a$ or $b$ is an invertible constant modulo $T$\nllabel{SToGBA:1}}{\Return{$[[1]]$}\nllabel{SToGBA:endab}}

$\bfS  \leftarrow \LNND5(a,\ b,\ T)$ \nllabel{SToGBA:LNL}
\tcp*[r]{Write $\bfS=[[g_i,\ h_i,\ U_i,\ T_i]_i]$}
 \For
{$[g_i,\ h_i,\ U_i,\ T_i]\in \bfS$}{
   \If(\tcp*[f]{Non-trivial input: pursue computations}){$g_i\not=1$\nllabel{SToGBA:if}}{
     \eIf(\tcp*[f]{Here $g_i$ is a gcd mod $T_i$}){$h_i = 0$}{
       $\out \leftarrow \out \cat [[g_i,\ U_i T_i]]$\nllabel{SToGBA:out}\;
     }(\nllabel{SToGBA:else})
         {
           $\bfG_i\leftarrow \SToGBA(g_i,\ h_i,\ U_i)$\nllabel{SToGBA:rec}\tcp*[r]{Recursive call}
           $V_{i0}' \leftarrow T_i$\nllabel{SToGBA:0}\;
           \For(\tcp*[f]{$G_{ij}$ is a minimal lexGb}){$G_{ij}\in \bfG_i$}{
             \tcp*[r]{Write $G_{ij}=[ h_{ij1},\ h_{ij2} f_{ij2},\ \ldots,\ h_{ij \lambda-1} f_{ij \lambda-1} ,\ f_{ij\lambda}]$ as in Theorem~\ref{th:Laz}}\nllabel{SToGBA:isol}
             $(\, V_{ij},\ V_{ij}'\, ) \leftarrow \isol(V_{ij-1}',\ h_{ij1})$\;
              $ G_{ij}^{new}\ \leftarrow \   V_{ij} \cdot G_{ij} ~ \cat ~ [g_i]$ \nllabel{SToGBA:switch}\;
           }
           $\out \leftarrow \out \cat [G_{ij}^{new}]$\;
         }
   }
 }
 \Return{$\out$}
\caption{\label{algo:SToGBA} \quad $\bfG \leftarrow \SToGBA(a,\ b,\ T)$ }
\end{algorithm}

\begin{Prop}[Correctness of Algorithm~\ref{algo:SToGBA} ``$\SToGBA$'']\label{prop:SToGBA}
Let $\bfG=(\calG_i)_i$ be the output of $\SToGBA(a,\ b,\ T)$.
Writing the systems  $(\calG_i)_i$ as in Eq.~\eqref{SToGBA:specif},
$$
\calG_i = [\ h_{i1},\ h_{i2} f_{i2},\ \ldots,\ h_{i \mu-1} f_{i \mu -1},\ f_{i \mu}\ ],
$$
they are minimal lexGbs verifying the specifications:
 \begin{enumerate}[{\em (i)}]
  \item $h_{i1}$ and $h_{j1}$ are coprime.
    Additionally,  $\prod_i h_{i1} | T$ and $\prod_i \sqfp(h_{i1}) = \sqfp(T)$.
  \item $\sqfp(h_{i1})=\sqfp(h_{i2})=\cdots =\sqfp(h_{i \mu-1})$.
  \item   $\prod_i \l \calG_i\r = \l a,\ b,\ T\r$ and $\l \calG_i\r + \l \calG_j\r =\l 1\r$ for $i\not=j$.
\end{enumerate}
\end{Prop}
\begin{proof}
If the input $a$ or $b$ is an invertible constant modulo $T$
(Line~\ref{SToGBA:1}) then the ideal $\l a,\ b,\ T\r=\l 1\r$. The output Line~\ref{SToGBA:endab} is then logically $[[1]]$.

The call to $\LNND5(a,\ b,\ T)$ at Line~\ref{SToGBA:LNL} outputs a family $[[g_i,\ h_i,\ U_i, T_i]_i]$
such that 
\begin{equation}\label{eq:gi=a}
\l g_i,\ h_i T_i,\ T_i U_i\r= \l a,\ b,\ U_iT_i\r
\end{equation}
 with the polynomials $(U_i T_i)_i$ pairwise
coprime and $\sqfp(U_i)=\sqfp(T_i)$ if $U_i\not=1$.

If $g_i=1$ then  the ideal $\l g_i,\ h_i U_i,\ T_i U_i\r =\l 1\r$ and the corresponding lexGb is $[1]$,
which does not need to be added to the output. That is why the if-test at Line~\ref{SToGBA:if} discards this case.

If $g_i\not=1$ and $h_i=0$ then $\l a,\ b,\ U_i T_i \r = \l g_i,\ T_i U_i\r$,
and  $[g_i, \ T_i U_i]$  is a minimal lexGb, which is added to the output (Line~\ref{SToGBA:out}).

Otherwise, a recursive call at Line~\ref{SToGBA:rec} is performed.
Its validity is proved in the lemma~\ref{lem:heart} hereafter:
we can assume that the output $\bfG_i$
verifies the output's specifications.
\begin{Lem}\label{lem:heart}
The recursive call $\bfG_i\leftarrow\SToGBA(g_i,\ h_i,\ U_i)$ at Line~\ref{SToGBA:rec} is valid.
\end{Lem}
\begin{proof}
  By the degree condition Eq.~\eqref{eq:deg-cond-D5} of Proposition~\ref{prop:LNLD5}
  related to Algorithm  $\LNND5$, we have $\deg_y(g_i)>\deg_y(h_i)$ unless $h_i=0$.
Regarding that $h_i\not=0$ ({\tt else} at  Line~\ref{SToGBA:else}) we have $\deg_y(g_i)>\deg_y(h_i)$. Moreover $g_i$ is monic so the input $(g_i,\ h_i,\ U_i)$ satisfies the input's specifications of Algorithm~\ref{algo:SToGBA} ``$\SToGBA$''.
Still from that proposition~\ref{prop:LNLD5}, $\deg_y(g_i) = \deg_y(a)$ possibly holds only if $\deg_y(a)=\deg_y(b)$.
But then $\deg_y(h_i)<\deg_y(g_i)<\deg_y(b)$. Thus, in any situations the input $(g_i,\ h_i,\ U_i)$ presents a degree decrease compared to that
of the main call $(a,\ b,\ T)$. By induction, we can assume that
the output is as expected.
\end{proof}
Writing $\bfG_i=(G_{ij})_j$, where $(G_{ij})_j$ is a family of minimal
lexGbs that we denote 
$G_{ij}=[ h_{ij1},\ h_{ij2} f_{ij2},$ $ \ldots,$ $\ h_{ij \lambda-1} f_{ij \lambda-1} ,\ f_{ij\lambda}]$ as in Theorem~\ref{th:Laz}.
They verify:

\noindent (i') $h_{ij1}$ and $h_{i\ell 1}$ are coprime and $\prod_j h_{ij1} | U_i$, as well as $\prod_j \sqfp(h_{ij1}) = \sqfp(U_i)$.\\
(ii') $\sqfp(h_{ij1}) =\cdots =\sqfp(h_{ij\lambda-1} )$.\\
(iii') $\prod_j \l G_{ij}\r = \l g_i,\ h_i,\ U_i\r$ and $\l G_{ij}\r + \l G_{i\ell}\r =\l 1\r$ if $j\not=\ell$.\\

\noindent  Then a call to $\isol$ has the effect of `` projecting'' $T_i$ along the $(h_{ij1})_j$s.
We have $V_{ij}= V_{ij-1}'^{(h_{ij1})}$ where $ (\, V_{ij},\ V_{ij}' \, ) \leftarrow \isol(V_{ij-1}',\ h_{ij1})$, with $V_{i0}'=T_i$ (Line~\ref{SToGBA:0}).
The specifications of Algorithm~\ref{algo:isol} ``$\isol$'' provides $V_{ij} V_{ij}' = V_{ij-1}'$, the product being that of coprime polynomials.
By induction we get:
\begin{equation}\label{eq:spade}
T_i = V_{i0}' = V_{i1} V_{i1}' = V_{i1} V_{i2} V_{i2}' = \cdots = V_{i1} V_{i2} \cdots V_{ij} V_{ij}', \quad \text{for all}\ \  j\ge1.
\end{equation}
Moreover the factors in the product are pairwise coprime.
It follows that 
$$
V_{ij} = V_{ij-1}'^{(h_{ij1})} = \prod_{p\in \calP, v_p(h_{ij1})>0} p^{v_p(V_{ij-1}')}=
\prod_{p\in \calP, v_p(h_{ij1})>0} p^{v_p(T_i)}=T_i^{(h_{ij1})}.
$$
Since the $(h_{ij1})_j$s are pairwise coprime,  
and that $\prod_j \sqfp(h_{ij1}) = \sqfp(U_i)$, each irreducible factor of $U_i$ appears  in those of one (and only one) of $h_{ij1}$.
Moreover, $U_i\not=1$  by Remark~\ref{rmk:LNLD5=1}, hence $\sqfp(U_i)=\sqfp(T_i)$ by the output' specifications of Algorithm~\ref{algo:LNLD5}
``$\LNND5$''. Therefore 
\begin{equation}\label{eq:Vij}
\prod_j V_{ij} = \prod_j T_i^{(h_{ij1})} =
\prod_j \prod_{p\in \calP,\ v_p(h_{ij1})>0} \negthickspace \negthickspace\negthickspace 
p^{v_p(T_i)} = \prod_{p \in \calP,\ v_p(U_i)>0} \negthickspace \negthickspace\negthickspace  p^{v_p(T_i)}
=\prod_{p \in \calP,\ v_p(T_i)>0} \negthickspace \negthickspace\negthickspace  p^{v_p(T_i)} = T_i.
\end{equation}
At Line~\ref{SToGBA:switch}, $G_{ij}^{new} := V_{ij} G_{ij} ~\cat~[g_i]$, which rewrites using the notation above:
$$
G_{ij}^{new}:=[  h_{ij1} V_{ij},\ h_{ij2} V_{ij} f_{ij2},\ \ldots,\ h_{ij \lambda-1} V_{ij} f_{ij \lambda-1} ,\ V_{ij} f_{ij\lambda},\ g_i].
$$
Since $g_i\in \l G_{ij}\r$, Theorem~\ref{th:Laz}-(B) implies that $G_{ij}^{new}$ is a lexGb.
The following lemma~\ref{lem:diamond} shows that it is minimal.

\begin{Lem}\label{lem:diamond}
  The lexGb  $G_{ij}^{new}$ is a minimal one.
\end{Lem}
\begin{proof}
According to Theorem~\ref{th:Laz}-(B), it suffices to prove
that $\deg_y(g_i)$ is larger than the degree (in $y$) of all the polynomials in $G_{ij}$.
Consider the first recursive call at Line~\ref{SToGBA:rec}. Recall that $g_i\not=1$ and $h_i\not=0$.
If $h_i$ is a constant (Line~\ref{SToGBA:endab}) then  the recursive call outputs $ [[1]]$
and the minimality of $G_{ij}$ is obvious.
Else it goes to another call to $\LNND5$ at Line~\ref{SToGBA:LNL}.
$$
[[g_{i\ell},\ h_{i\ell}, \ U_{i\ell},\ T_{i\ell}]_\ell] \leftarrow \LNND5(g_i,\ h_i,\ U_i).
$$
According to the degree condition of Proposition~\ref{prop:LNLD5}, an equality $\deg_y(g_{i\ell}) = \deg_y(g_i)$
can only occur if $\deg_y(h_i)=\deg_y(g_i)$, which does not hold as seen above.
\end {proof}

This proves one output's specification.
According to the specification~(ii') above, we obtain $\sqfp(h_{ij1} V_{ij}) =\cdots =\sqfp(h_{ij\lambda-1} V_{ij})$.
Moreover, all irreducible factors of $V_{ij}=T_i^{(h_{ij1})}$ are irreducible factors of $h_{ij1}$ by definition.
We obtain that $\sqfp(h_{ij\ell} V_{ij})=\sqfp(V_{ij})$ for $1\le \ell\le \lambda-1$. This proves Specification~(ii).

Additionally, $\prod_i \prod_j h_{ij1} V_{ij} | \prod_i U_i T_i | T$ proving the second statement of Specification~(i).
The first statement is that the $(h_{ij1} V_{ij})_j$s are pairwise coprime. This follows from Eq.~\eqref{eq:spade} and the specification~(i').

Besides, $\prod_j \l G_{ij}^{new}\r= \prod_j (\l V_{ij} G_{ij}\r + \l g_i\r)$. Lemma~\ref{lem:id} in Appendix,
proves that $\prod_j \l G_{ij}^{new}\r = (\prod_j \l V_{ij} G_{ij}\r ) + \l g_i\r$.
This ideal is equal to $(\prod_j V_{ij})(\prod_j \l G_{ij}\r) + \l g_i\r$. By the specification~(iii') and Eq.~\eqref{eq:Vij} above,
we get: $\prod_j \l G_{ij}^{new}\r=T_i \l g_i,\ h_i,\ U_i\r + \l g_i\r = \l T_i g_i,\ T_i h_i, T_i U_i,\ g_i\r = \l g_i,\ T_i h_i,\ T_i U_i \r$.
Finally, with Eq.~\eqref{eq:gi=a} and the Chinese remainder theorem we obtain:
$$
\l a,\ b,\ T\r\simeq \prod_i \l a,\ b,\ U_i T_i\r = \prod_i \l g_i,\ T_i h_i,\ U_i T_i\r = \prod_{i,\ g_i\not=1,\ h_i=0} \l g_i,\ U_i T_i\r \cdot \prod_{i,\ g_i\not=1,\ h_i\not=0} \prod_j \l G_{ij}^{new}\r
$$
which is output's specification~(iii).
That they are pairwise coprime follows from the fact that the ideals $(\l g_i,\ T_i h_i,\ T_i U_i\r)_i$
are pairwise coprime.
\end{proof}

Next the assumption that $a$ has an invertible leading coefficient modulo $T$ is lifted.
It suffices to run the algorithm~\ref{algo:MFD5} ``$\MFD5(a,\ T)$'' (Line~\ref{SToGBD5:a}),
and naively to call the previous algorithm~\ref{algo:SToGBA} ``$\SToGBA$'' on each component of that monic form.
But this induces a minor issue, in that the monic forms $(a_j)_j$ of $a$ may not all
satisfy $\deg_y(b) \le \deg_y(a_j)$, a requirement for  calling ``$\SToGBA$''.
The if-test Line~\ref{SToGBD5:if} distinguishes the $j$s
for which $\deg_y(a_j) < \deg_y(b)$ from those $j$s for which $\deg_y(b)\le \deg_y(a_j)$.
In the former case (Lines~\ref{SToGBD5:else}-\ref{SToGBD5:out2}), calling
directly $\SToGBA(b_j,\ a_j,\ T_j)$ does not necessarily work
because $b_j\equiv b \bmod T_j$  may not have an invertible leading coefficient modulo $T_j$
--- another requirement of ``$\SToGBA$''.
Thus we make $b_j$ monic  first (Line~\ref{SToGBD5:Mbj}). Then
a last if-test (Line~\ref{SToGBD5:if2})  distinguishes the cases $\deg_y(a_{ij}) < \deg_y(b_{ij})$
from the cases $\deg_y(a_{ij}) \ge \deg_y(b_{ij})$, before this time calling $\SToGBA$
with all its specifications guaranteed.

\begin{algorithm}[!h]
\SetAlgoVlined 
 \DontPrintSemicolon 
 \KwIn{ 1-2. polynomials $a ,\ b\in k[x,y]$, $\deg_y(a)\ge \deg_y(b)$ satisfying Assumption~\eqref{tag:H}.
   
 3. $T\in k[x]$ is a monic non-constant polynomial.}

\KwOut{Family of minimal lexGbs $\bfG=(\calG_i)_i$ written following Theorem~\ref{th:Laz} as in Eq.~\eqref{SToGBD5:specif},
  and verifying the specifications (i)-(ii)-(iii) of Theorem~\ref{th:SToGBD5}.
  
 \vspace{-2ex}
  \begin{equation}\label{SToGBD5:specif}
    \calG_i = [h_{i1},\ h_{i2} f_{i2},\ \ldots,\ h_{i \mu-1} f_{i \mu -1},\ f_{i \mu}],\quad h_{ij} \in k[x],\ f_{ij} \in  k[x,y] {\rm~monic~in~} y,
  \end{equation}
}

\BlankLine

\If{$a$ or $b$ is a non-zero  constant}{\Return{$[[1]]$}\nllabel{SToGBD5:endab}}

$\out\leftarrow []$\;

$\bMa \leftarrow \MFD5(a,\ T)$ \nllabel{SToGBD5:a} \tcp*[r]{Write $\bMa=[[a_j,\ T_j,\ 1]_j]$}
\For
    {$[a_j,\ T_j,\ 1]\in \bMa$\nllabel{SToGBD5:for1}}{
      $b_j\leftarrow b \bmod T_j$\;
      \eIf(\nllabel{SToGBD5:if}){$\deg_y(a_j)\ge \deg_y(b_j)$}
          {
            $\out \leftarrow \out \cat \SToGBA(a_j,\ b_j,\ T_j)$\nllabel{SToGBD5:out}\;
          }(\nllabel{SToGBD5:else})
          {
            $\bMb_j  \leftarrow\MFD5(b_j,\ T_j)$ \nllabel{SToGBD5:Mbj} \tcp*[r]{Write $\bMb_j=[[b_{ij},\ T_{ij},\ 1]_i]$}
            \For
                {$[b_{ij},\ T_{ij},\ 1]\in \bMb_j$}{
                  $a_{ij} \leftarrow a_{j} \bmod T_{ij}$\;
                  \eIf(\nllabel{SToGBD5:if2}){$\deg_y(a_{ij}) \ge \deg_y(b_{ij})$}
                      {
                        $\out \leftarrow \out \cat \SToGBA(a_{ij},\ b_{ij},\ T_{ij})$\nllabel{SToGBD5:out1}\;
                      }
                      {
                        $\out \leftarrow \out \cat \SToGBA(b_{ij},\ a_{ij},\ T_{ij})$\nllabel{SToGBD5:out2}\;
                      }
                      
                }
                
          }
    }
 \If{$\out =[]$}{\Return{$[[1]]$}}
 \Return{$\out$}
\caption{\label{algo:SToGBD5} \quad $\bfG \leftarrow\SToGBD5(a,\ b,\ T)$ }
\end{algorithm}

\begin{Theo}[Correctness of Algorithm~\ref{algo:SToGBD5} ``$\SToGBD5$'']\label{th:SToGBD5}
The  output $\out$ of $\SToGBD5(a,\ b,\ T)$  verifies the specifications below.
  \begin{enumerate}[{\em (i)}]
  \item $h_{i1}$ and $h_{j1}$ are coprime.
    Additionally,  $\prod_i h_{i1} | T$ and $\prod_i \sqfp(h_{i1}) = \sqfp(T)$.
    \item $\sqfp(h_{i1})=\sqfp(h_{i2})=\cdots =\sqfp(h_{i \mu-1})$.
    \item   $\prod_i \l \calG_i\r = \l a,\ b,\ T\r$ and $\l \calG_i\r + \l \calG_j\r =\l 1\r$ for $i\not=j$.
\end{enumerate}
\end{Theo}
\begin{proof}
If $a$  or $b$  is a non-zero constant, then $\l a,\ b,\ T\r =\l 1\r$ 
so that the output $[[1]]$ at Line~\ref{SToGBD5:endab} is correct.

By Assumption~\eqref{tag:H},
for each primary factor $p^e$ of $T$, $a \bmod p^e$ is not nilpotent.
Therefore, when ``making'' $a$ monic with $\MFD5(a,\ T)$, the collection of output
is always of the form $[a_j,\ T_j,\ 1]$. The third component is indeed a polynomial $U_i\not=1$
if and only if $U_i$ divides $a$ modulo $T_i$. Since $\sqfp(U_i)=\sqfp(T_i)$,
this would imply that  $a$ is nilpotent modulo $T_i$, in contradiction with Assumption~\eqref{tag:H}.
 
The output $\bMa =[[a_j,\ T_j, 1]_j]$ of $\MFD5(a,\ T)$  at Line~\ref{SToGBD5:a},
verify $\l a_j,\ T_j\r = \l a,\ T_j\r$ with $a_j$ monic and $\prod_j T_j = T$,
the $(T_j)_j$s being pairwise coprime.
The algorithm treats  each component $T_j$ of $T$ separately
in the for loop Line~\ref{SToGBD5:for1}-\ref{SToGBD5:out2}.

Naively, a call to ``$\SToGBA$'' shall be performed
with $(a_j,\ b_j,\ T_j)$, but this input does not necessarily satisfy the specifications
of that algorithm if $\deg_y(a_j) < \deg_y(b_j)$.
In that case (Lines~\ref{SToGBD5:else}-\ref{SToGBD5:out2}), we cannot simply switch $a_j$ with $b_j$
by calling directly $\SToGBA(b_j,\ a_j,\ T_j)$ neither 
since $b_j$ may not have a leading coefficient  invertible modulo $T_j$ (a requirement of that algorithm).
Therefore we ``make'' $b_j$  monic with a call to $\MFD5(b_j,\ T_j)$ at Line~\ref{SToGBD5:Mbj}.
Since $b$ satisfies Assumption~\eqref{tag:H}, $b\bmod p^e$ is not nilpotent for
all primary factor $p^e$ of $T$. Therefore,  $\bMb_j=[[b_{ij},\ T_{ij},\ 1]_i]$ (third component
is always 1).
$$
b_{ij}{\rm~monic}\qquad
\l b_{ij},\ T_{ij}\r = \l b_j,\ T_{ij}\r,\qquad \prod\nolimits_i T_{ij} = T_j,
\quad (T_{ij})_i{\rm s~pairwise~coprime}.
$$
Now both families $(a_{ij})$ and $(b_{ij})$ are made of polynomials having an invertible coefficient
modulo $T_{ij}$,
hence calls to $\SToGBA$ at Line~\ref{SToGBD5:out1} or Line~\ref{SToGBD5:out2}
are correct.
$\out$ concatenates all these minimal lexGbs. This leads to the following equalities, where the second one
is deduced from Lemma~\ref{lem:id} in Appendix.
$$
\prod_{\calG \in \out} \l \calG\r \simeq \prod_j \prod_i \l a_{j},\ b_{ij},\ T_{ij}\r 
\simeq
\prod_j (\l a_j\r + \prod_i \l b_{ij},\ T_{ij}\r)
\simeq \prod_j \l a_j,\ b_j,\ T_j\r = \l a,\ b,\ T\r. 
$$
This proves the first statement of the specification (iii), the second statement being  that these lexGbs are pairwise coprime.
Let $\calG_1$ and $\calG_2$  be two different lexGbs in $\out$.
If $\calG_1$ and $\calG_2$ are in the same output of a call to ``$\SToGBA$'' 
at Line~\ref{SToGBD5:out},~\ref{SToGBD5:out1} or~\ref{SToGBD5:out2},
then the output' specifications of ``$\SToGBA$'' guarantee that 
 $\calG_1$ and $\calG_2$ are coprime.

If $\calG_1$ and $\calG_2$ are obtained from different calls to ``$\SToGBA$'',
then the $(T_j)_j$s (Line~\ref{SToGBD5:out}) are pairwise coprime,
or the $(T_{ij})_{ij}$s (Lines~\ref{SToGBD5:out1} or~\ref{SToGBD5:out2}) are
pairwise coprime, so that $\calG_1$ and $\calG_2$ are coprime. This ends the proof of the specification (iii).

Finally, the specifications (i)-(ii) are also satisfied according that the lexGbs are computed by calls to ``$\SToGBA$''
whose output satisfy these two specifications. 
\end{proof}

\section{Implementation and experimental section}\label{sec:exp}
All algorithms are implemented in Magma v2.25-1 and can be found at
 \url{http://xdahan.sakura.ne.jp/gb.html}.
Experiments were realized on an Intel processor Core i7-6700K clocked at 4GHz.
The timings naturally compare with that of the {\tt GroebnerBasis} command.
We recall that a common strategy to compute a lexGb divides in two steps:
first compute a Gr\"obner basis for the degree reverse lexicographic order ({\tt grevlex} for short)
then applies a change of monomial order algorithm, typically FGLM since we are constrained to systems
of dimension zero.
Magma is equipped with one of the best implementation of Faug\`ere algorithm F4~\cite{F4},
which is called in the first step.
We report on  the timing F4$+$FGLM, with the timing of
FGLM appended inside parentheses when not negligible.

We consider only polynomials over a finite field. Rational coefficients induce large coefficients swell,
and without modular methods become difficult to compare. We could have turned the {\tt Modular} option off
in Magma when calling a Gr\"obner basis computation, but for now we would rather stick to
finite fields. The comparison with small size (16bits) and medium size (64bits)
finite fields bring enough evidences of the advantage of the proposed algorithms.

\subsection{Testing suite\label{sec:testing}}
We consider two families of examples, each tested over a 16-bits finite field and a 64-bits finite field.
The polynomials $a,\ b$ are defined as follows:

\begin{itemize}[left=0pt]
\item First family (Tables~\ref{tab:data1}-\ref{tab:2-15bits}-\ref{tab:2-63bits}): 16 polynomials  $a,b$  are constructed with respect to a modulus $T$ whose factors
  are defined as follows:

  \begin{itemize}[left=0pt]
  \item Consider $T = p_1^{e_1} p_2^{e_2}$ (two factors where $p_1 = x$ and $p_2=(x+1)$),
    
$T = p_1^{e_1} p_2^{e_2} p_3^{e_3} $ (three factors, where
    $p_1 = (x+10),\ \  p_2=(x+20), \ \ p_3=(x+30) $), and finally
    
   $T=p_1^{e_1} p_2^{e_2} p_3^{e_3} p_4^{e_4}$ (four factors,
where $p_1 = x,\ \  p_2=(x+5),\ \  p_3=(x+10),\ \ p_4=(x+15) $).

\item The exponents are respectively 

(examples 1-4) $(e_1, e_2) =(5 i,\ 5(i+1))_{i=1,2,3,4}$ when there are two factors,

(examples 5-11) $(e_1, e_2, e_3) = (3i-1,\ 3i,\ 3i+1)_{i=3,4,5,6,7,8,9}$, when there are three factors,

(examples 12-16) $(e_1, e_2, e_3, e_4) = (4 i,\ 4 i+1,\ 4i+2,\ 4i+3)_{i=1,2,3,4,5}$, when there are four factors.

\item For each $p_i^{e_i}$ polynomials $a_i, b_i$ 
are built as follows:

$a_i := (y+p_i)\cdot \prod_{\ell=1}^{e_i-1} (y+p_i + p_i^2 + \cdots + p_i^\ell + \ell+ 2 p_i^{\ell+1}) \bmod p_i^{e_i}$

$b_i := (y+ 2 p_i)\cdot \prod_{\ell=1}^{e_i-1} (y+p_i + p_i^2 + \cdots + p_i^\ell + \ell+  p_i^{\ell+1}) \bmod p_i^{e_i}$

\item Then the CRT is applied to construct polynomials $a,b$ modulo $T$
from their local images $a_i,b_i$.

\end{itemize}

\item Second family (Tables~\ref{tab:data2}-\ref{tab:mul-15bits}-\ref{tab:mul-63bits}): 6 polynomials $a,b$ constructed according to a modulus $T$ whose factors are defined as follows:

  \begin{itemize}[left=0pt]
    \item Consider $T= q_7^{e} q_6^{e-1}\cdots q_1^{e-6}$ where each polynomial
$q_i=r_{i1}\cdots r_{i(8-i)}$ and $r_{ij} =x + i(i+1)/2 + j-2$.
For example $r_{71}=x+27$ and $r_{1j}=x+1+j-2=x+j-1$. 

\item $e = 7,8,9,10,11,12$.

\item For each $r_{ij}^{e-i}$ 
polynomials  $a_{ij}, b_{ij}$ 
are built as follows:

$a_{ij} :=  (y+r_{ij})\cdot \prod_{\ell=1}^{e - i} (y+r_{ij} + r_{ij}^2 + \cdots + r_{ij}^\ell + \ell+ 2 r_{ij}^{\ell+1})
\bmod r_{ij}^{e-i} $

$b_{ij} :=  (y+2 r_{ij})\cdot \prod_{\ell=1}^{e-i} (y+r_{ij} + r_{ij}^2 + \cdots + r_{ij}^\ell + \ell+  r_{ij}^{\ell+1}) \bmod r_{ij}^{e-i}$

\item Then the CRT is applied to construct polynomials $a,b$ modulo $T$
from their local images $a_{ij},b_{ij}$.
\end{itemize}
\end{itemize}
The first family concerns a modulus $T$ with few factors of moderate
 exponent degrees, each exponent degrees being moderately distant.
The second family targets a modulus $T$ having 27 factors,
and a squarefree decomposition made of seven factors.
The exponent degrees are distant of one, and are relatively small.

The correctness of the  output of Algorithm~\ref{algo:SToGBD5} ``$\SToGBD5$'' was
checked as follows:
Take $\calG$ the lexGb output by the Gr\"obner engine of Magma, and $(\calG_i)_i$ the family
of lexGbs output by Algorithm~\ref{algo:SToGBD5}.
We checked that each polynomial in $\calG$ reduces to zero modulo each $\calG_i$. In this way
$\calG \subset \prod_i \l \calG_i\r$. And compared the dimension of the two vector spaces
$k[x,y]/\l \calG\r$ and $\prod_i k[x,y]/\l \calG_i\r$.

The two tables~\ref{tab:data1}-\ref{tab:data2} summarize data attached to the system $(a,\ b)$ (without a  modulus $T$, which will turn out
to be $\Res_y(a,b)$, see next Section~\ref{sec:time}).
This kind of input is of practical importance and the timings are not always favorable to
Algorithm~\ref{algo:SToGBD5}, which require a special analysis, undertaken in
Section~\ref{sec:compl}.
\medskip

\noindent \underline{Description of Tables~\ref{tab:data1}-\ref{tab:data2}}
\begin{enumerate}[Column 1:, left=0pt]\itemsep0pt
\item refers to the example's number.\ ($4+7+5=16$ examples for the 1st family, and $6$ examples for the 2nd family).
\item degree $\DEG$ of the ideal generated by $\l a,\ b\r$, equal to $\dim_k(k[x,y]/\l a,\ b\r)$.
\item degrees  $\deg_y(a)$ and $\deg_y(b)$ of the input polynomials $a$ and $b$ (always equal).
\item total degrees $\tdeg(a)$, $\tdeg(b)$ of the input polynomials $a$ and $b$ (always equal).
\item $\deg(\Res_y(a,b))\in k[x]$  degree of the resultant of $a$ and $b$.
\item sum of the degrees of the  factors having multiplicity one in the factorization of  $\Res_y(a,b)$.
\item sum of the degrees of the factors having multiplicity $>1$ in the factorization of  $\Res_y(a,b)$.
\item average multiplicity of an irreducible factor of $\Res_y(a,b)$ having multiplicity $>1$ (do not count the factor of multiplicity one).
\item number of lexGbs output by Algorithm~\ref{algo:SToGBD5}
``$\SToGBD5''$. Inside parentheses, the average number of polynomials in these lexGbs.
\end{enumerate}

\begin{table}
\begin{small}
  \begin{tabular}{|c||ccc|cccc|c|}\hline
    \begin{tabular}{c}ex.\\ nbr. \end{tabular} & \multicolumn{1}{c|}{\!\!\!\!$\DEG$\!\!\!\!} &
    \multicolumn{1}{c|}{ \!\!\!\!\!\!\!\! $\begin{array}{c}\deg_y(a)\\=\deg_y(b)\end{array}$ \!\!\!\!\!\!\!\!}&
    \multicolumn{1}{c||}{\!\!\!\!\!\!\!\! $\begin{array}{c}\tdeg(a)\\=\tdeg(b)\end{array}$ \!\!\!\!\!\!\!\!} &
\multicolumn{1}{c|}{\kern-1em    \begin{tabular}{c}degree\\ resultant\end{tabular} \kern-1em} & \multicolumn{1}{c|}{\kern-1em   \begin{tabular}{c}deg. multi-\\plicity 1 \end{tabular} \kern-1em}  & \multicolumn{1}{c|}{\kern-1em 
    \begin{tabular}{c}deg. multi-\\plicity  $>$1\end{tabular} \kern-1em} & \multicolumn{1}{c||}{\kern-1em  \begin{tabular}{c}avg. multi-\\plicity $>$1\end{tabular} \kern-1em} & \kern-1em\begin{tabular}{c} \#lexGbs\\ (avg. \#poly.)\end{tabular}\kern-1em \\\hline
1 & 241 & 10 & 24 &  266 & 171 &  95 &    47.5 & 2 (8.5)  \\
  2 & 646 & 15 & 39 &  696  & 471 &  225 &        112.5 & 2 (14.5)\\
  3 & 1251 & 20 & 54 & 1326  & 919 &  407 &          135.67 & 2  (18.5) \\
  4 & 2056 & 25 &69 &  2165 & 1521 &  635 &   317.5 & 2 (23.5)\\ \hline\hline
  5 & 281 & 8 & 28 & 300 &  196 &  104 & 34.67 & 3 (8) \\
  6 & 581 & 11 & 40 & 609  & 415 &  194 & 64.67 & 3 (11)  \\
  7 & 989 & 14 & 52 &  1026 & 715 & 311 & 103.67 & 3 (14) \\
  8 & 1505 &17 &  64 & 1551  & 1096 & 455 & 151.67 & 3 (17)\\
  9 & 2129 & 20 & 76 &  2184  & 1558 & 626 &  208.67 & 3 (20)\\
  10 & 2861 & 23 & 88 & 2925  & 2101 & 824 & 274.67 & 3 (23) \\
  11 & 3701 & 26 & 100 &  3774 & 2725 & 1049 & 349.67 & 3 (26) \\ \hline\hline
  12 & 275 & 7 & 28 &  273 & 171 & 102 & 25.5 & 4 (6.5) \\
  13 & 725 & 11 & 48 & 777  & 523 & 254 & 63.5 & 4 (10.5) \\
  14 & 1461 & 15 & 68 &  1537  & 1067 & 470 & 117.5 & 4 (14.5)\\
  15 & 2453 & 19 & 88 & 2553  & 1803 & 750 & 187.5 & 4 (18.5)\\
  16 & 3701 & 23 & 108 & 3825  & 2731 & 1094 & 273.5 & 4 (22.5) \\\hline
  \end{tabular}
\end{small}
\caption{Data attached to the polynomials $a,\ b$ in the 1st family of examples}\label{tab:data1}
\end{table}

\begin{table}
  \begin{small}
    \begin{tabular}{|c||ccc|cccc|c|}\hline
    \begin{tabular}{c}ex.\\ nbr. \end{tabular} & \multicolumn{1}{c|}{\!\!\!\!$\DEG$\!\!\!\!} &
    \multicolumn{1}{c|}{ \!\!\!\!\!\!\!\! $\begin{array}{c}\deg_y(a)\\=\deg_y(b)\end{array}$ \!\!\!\!\!\!\!\!}&
    \multicolumn{1}{c||}{\!\!\!\!\!\!\!\! $\begin{array}{c}\tdeg(a)\\=\tdeg(b)\end{array}$ \!\!\!\!\!\!\!\!} &
\multicolumn{1}{c|}{\kern-1em    \begin{tabular}{c} degree\\ resultant\end{tabular} \kern-1em} & \multicolumn{1}{c|}{\kern-1em   \begin{tabular}{c}deg. multi-\\plicity 1 \end{tabular} \kern-1em}  & \multicolumn{1}{c|}{\kern-1em 
    \begin{tabular}{c}deg. multi-\\plicity  $>$1\end{tabular} \kern-1em} & \multicolumn{1}{c||}{\kern-1em  \begin{tabular}{c}avg. multi-\\plicity $>$1\end{tabular} \kern-1em} & \kern-1em\begin{tabular}{c} \#lexGbs\\ (avg. \#poly.)\end{tabular}\kern-1em \\\hline
    1 & 827 & 7 & 90 &  1079  & 617 & 462 &   6.5  & 7 (5) \\
    2 & 1301 & 8 & 119 &  1665 & 979 & 686 &  24.5 & 7 (6)\\
    3 & 1887 & 9 & 148 &  2363 & 1425 & 938 & 33.5 & 7 (7)\\
    4 & 2585 & 10 & 177 &  3173 & 1955 &  1218 &  43.5 & 7 (8) \\
    5 & 3395 & 11 & 206 &  4095  & 2569 & 1526 &  54.5 & 7 (9)\\
    6 & 4317 & 12 & 235 &  5129 & 3267 &  1862 &  66.5 & 7 (10)\\\hline
    \end{tabular}
    \end{small}
  \caption{Data attached to the polynomials $a,\ b$ in the 2nd family of examples}\label{tab:data2}
\end{table}

\subsection{Timings\label{sec:time}}
Tables~\ref{tab:2-15bits}-\ref{tab:2-63bits} display timings for the first family of 16 polynomials
described in Section~\ref{sec:testing}
(see also Table~\ref{tab:data1}). The double horizontal lines separate the four polynomials $T$ which have two primary factors,
from the seven polynomials $T$ that have three, and from the last five polynomials $T$ which have four primary factors.
Tables~\ref{tab:mul-15bits}-\ref{tab:mul-63bits} show the  timings for the second family of 6 polynomials
described in Section~\ref{sec:testing}
(see also Table~\ref{tab:data2}).
We remark that the polynomials $a$ and $b$ do not have solutions at infinity:
 $\gcd(\lc(a),\lc(b))=1$.
The four tables all have eight columns. They represent:
\medskip

\noindent \underline{Description of Tables~\ref{tab:2-15bits}-\ref{tab:2-63bits}-\ref{tab:mul-15bits}-\ref{tab:mul-63bits}}
\begin{enumerate}[Column 1:, left=0pt]\itemsep -2pt
\item  timing of the algorithm~\ref{algo:SToGBD5} ``$\SToGBD5$'' with input $a,\ b,\ T$.

\item  timing of {\tt GroebnerBasis}$([a,b,T])$ command in Magma: F4$+$FGLM with FGLM in parenthesis when not negligible.

\item  timing for the resultant computation $r(x):=\Res_y(a,b)$


\item  timing for Algorithm~\ref{algo:SToGBD5} ``$\SToGBD5$'' with input $a,\ b,\ r$.

\item  timing of {\tt GroebnerBasis}$([a,b])$ command in Magma: F4$+$FGLM with FGLM in parenthesis when not negligible.

\item  timing for the squarefree decomposition of $r=R_1^{e_1} \cdots R_s^{e_s}$. Here,
  $R_i$ is the product of those irreducible factors of the resultant $R$ that come with the same multiplicity $e_i$.

\item  total time taken for calling Algorithm~\ref{algo:SToGBD5} ``$\SToGBD5$'' on input $(a,\ b,\ R_i^{e_i})$ for $i=1,\ldots,s$.

\item  total time taken for calling {\tt GroebnerBasis}$([a,\ b,\ R_i^{e_i}])$  for $i=1,\ldots,s$.
\end{enumerate}

\begin{table}
\begin{small}
  \begin{tabular}{|c|c||c|c|c||c|c|c|}
\hline
    1 & 2 & 3 & 4 & 5 & 6 & 7 & 8 \\ \hline \hline
    \multicolumn{2}{|c||}{Input: $a,b,T$} & \multicolumn{2}{c|}{Input: $a,b,\Res_y(a,\ b)$}
    & Input: $a,b$ & \multicolumn{3}{|c|}{Several input: $a,b,R_i^{e_i}$}\\
    \cline{3-3}\cline{6-6}
 this & Gb(fglm) & res  & this & Gb(fglm) & sqf dec & this & Gb(fglm)\\ \hline \hline
 0.01 & 0 & 0 & 0.13 &  0.04 (0.029) & 0.01 & 0.02 & 0.13(0.02)\\ 
 0 & 0.03 & 0.02 & 2.32 &  0.33 (0.279) & 0 & 0.26 & 0.96(0.18)\\ 
 0.01 &  0.1 (0.009) & 0.04 & 12.45 &  1.8 (1.679) & 0.01 & 1.34 & 5.21(1.06)\\ 
 0.02 &  0.36 (0.05) & 0.11 & 48.62 &  6.28 (6) & 0.01 & 5.02 & 14.5(3.929\\ \hline\hline
 0 & 0.01 & 0.01 & 0.08 &  0.06 (0.03) & 0 & 0.02 & 0.2(0.02)\\ 
 0 & 0.03 & 0 & 0.69 &  0.26 (0.219) & 0.01 & 0.09 & 0.93(0.129)\\ 
 0.01 &  0.1 (0.02) & 0.03 & 2.69 &  0.96 (0.849) & 0.01 & 0.34 & 3.79(0.509)\\ 
 0.01 &  0.23 (0.039) & 0.06 & 8.78 &  2.87 (2.65) & 0.01 & 0.99 & 9.54(1.62)\\ 
 0.02 &  0.48 (0.079) & 0.1 & 23.58 &  6.96 (6.5) & 0.01 & 2.4 & 22.17(4.299)\\ 
 0.02 &  1.09 (0.17) & 0.19 & 55.05 &  15.44 (14.629) & 0.02 & 5.36 & 43.14(9.269)\\ 
 0.03 &  2.06(0.3) & 0.33 & 120.7 &  30.65(29.3) & 0.02 & 11.1 & 40.67(17.6)\\ \hline\hline
 0 & 0.01 & 0 & 0.04 &  0.04 (0.03) & 0 & 0.01 & 0.14(0.019)\\ 
 0.01 &  0.07 (0.01) & 0.02 & 0.94 &  0.48 (0.4) & 0 & 0.11 & 2.33(0.219)\\ 
 0.01 &  0.3 (0.039) & 0.05 & 5.76 &  2.84 (2.509) & 0.01 & 0.69 & 12.04(1.439)\\ 
 0.01 &  1.07 (0.139) & 0.14 & 23.15 &  10.69 (9.85) & 0.02 & 2.41 & 39.24(5.92)\\ 
 0.02 & 2.71(0.37) & 0.31 & 76.49 & 30.86(28.78) & 0.02 & 7.1 & 47.88(18.33)\\ \hline
  \end{tabular}
\end{small}
  \caption{First family over a finite field $GF(p)$, $p$ a 16bits prime}\label{tab:2-15bits}
\end{table}

\begin{table}
  \begin{small}
\begin{tabular}{|c|c||c|c|c||c|c|c|}
    \hline
    1 & 2 & 3 & 4 & 5 & 6 & 7 & 8 \\ \hline \hline
    \multicolumn{2}{|c||}{Input: $a,b,T$} & \multicolumn{2}{c|}{Input: $a,b,\Res_y(a,\ b)$}
    & Input: $a,b$ & \multicolumn{3}{|c|}{Several input: $a,b,R_i^{e_i}$}\\
    \cline{3-3}\cline{6-6}
    this & Gb(fglm) & res  & this & Gb(fglm) & sqf dec & this & Gb(fglm)\\ \hline \hline
 0.01 & 0.02 & 0.04 & 2.28 &  0.25 (0.22) & 0 & 0.28 & 0.44(0.13)\\ 
 0.02 &  0.22 (0.02) & 0.21 & 36.13 &  3.83 (3.649) & 0.02 & 3 & 5.21(2.23)\\ 
 0.04 &  1.23 (0.09) & 0.74 & 174.69 &  27.56 (26.65) & 0.05 & 17.36 & 32.49(18)\\ 
 0.09 &  6.17 (0.33) & 1.92 & 592.18 &  121.1 (118.009) & 0.08 & 73.91 & 143.42(78.74)\\ \hline\hline
 0.01 &  0.05 (0.009) & 0.05 & 1.28 &  0.39 (0.339) & 0 & 0.21 & 0.87(0.2)\\ 
 0.01 &  0.27 (0.019) & 0.17 & 10.16 &  3.87 (3.6) & 0.02 & 1.06 & 5.55(1.79)\\ 
 0.03 &  1.2 (0.109) & 0.47 & 36.48 &  20.48 (19.39) & 0.04 & 4.19 & 23.41(8.9)\\ 
 0.05 &  3.7 (0.339) & 1 & 107.99 &  60.39 (57.279) & 0.07 & 12.22 & 73.15(32.88)\\ 
 0.09 &  6.76 (0.74) & 2.06 & 277.45 &  169.11 (162) & 0.13 & 33.37 & 205.93(102)\\ 
 0.15 &  27.07 (2.039) & 3.45 & 717.37 &  386.3 (370.8) & 0.18 & 77.88 & 508.07(245)\\ 
 0.22 &  60.91 (4.5) & 5.72 & 1403.83 &  793.41 (763.609) & 0.27 & 166.48 & 939.93(513)\\ \hline\hline
 0.01 &  0.04 (0.01) & 0.04 & 0.65 &  0.29 (0.239) & 0.01 & 0.13 & 0.71(0.15)\\ 
 0.02 &  0.66 (0.059) & 0.26 & 14.41 &  7.54 (6.88) & 0.03 & 1.34 & 15.18(4.7)\\ 
 0.04 &  5.42 (0.399) & 0.98 & 83.5 &  62.34 (58.05) & 0.06 & 7.78 & 97.33(33.5)\\ 
 0.09 &  25.44 (1.489) & 2.59 & 296.32 &  273.9 (257.47) & 0.14 & 30.87 & 429.38(165)\\ 
 0.17 &  91.02 (6.909) & 7 & 943.93 &  907.33 (856.36) & 0.23 & 101.05 & 1383.59(624)\\ \hline
\end{tabular}
\end{small}
\caption{First family over a finite field $GF(p)$, $p$ a 64bits prime}\label{tab:2-63bits}
\end{table}

\begin{table}\begin{small}
    \begin{tabular}{|c|c||c|c|c||c|c|c|}
      \hline
    1 & 2 & 3 & 4 & 5 & 6 & 7 & 8 \\ \hline \hline
    \multicolumn{2}{|c||}{Input: $a,b,T$} & \multicolumn{2}{c|}{Input: $a,b,\Res_y(a,\ b)$}
    & Input: $a,b$ & \multicolumn{3}{|c|}{Several input: $a,b,R_i^{e_i}$}\\
    \cline{3-3}\cline{6-6}
 this & Gb(fglm) & res  & this & Gb(fglm) & sqf dec & this & Gb(fglm)\\ \hline \hline
 0 &  0.77(0.03) & 0.03 & 0.23 &  1.4(0.66) & 0 & 0.06 &  5.09(0.36)\\ 
 0 &  3.6(0.07) & 0.07 & 0.63 &  5.1(2.39) & 0.01 & 0.15 &  19.36 (1.289)\\ 
 0.01 &  9.84(0.119) & 0.13 & 1.53 &  13.08(5.969) & 0.01 & 0.34 &  50.39 (3.289)\\ 
 0.01 &  23.91(0.25) & 0.23 & 3.16 &  30.47(13.699) & 0.02 & 0.7 &  127.5 (7.5)\\ 
 0.01 &  60.21(0.509) & 0.37 & 6.32 &  64.96(28.7) & 0.02 & 1.28 &  268(15.9) \\ 
 0.02 &  115.72(0.97) & 0.57 & 11.37 &  134.31(54.509) & 0.03 & 2.28 &  527.95(27.8)\\ \hline
\end{tabular}
\end{small}
\caption{Second family over a finite field $GF(p)$, $p$ a 16bits prime}\label{tab:mul-15bits}
\end{table}

\begin{table}
  \begin{small}
    \begin{tabular}{|c|c||c|c|c||c|c|c|}
\hline
    1 & 2 & 3 & 4 & 5 & 6 & 7 & 8 \\ \hline \hline
    \multicolumn{2}{|c||}{Input: $a,b,T$} & \multicolumn{2}{c|}{Input: $a,b,\Res_y(a,\ b)$}
    & Input: $a,b$ & \multicolumn{3}{|c|}{Several input: $a,b,R_i^{e_i}$}\\
    \cline{3-3}\cline{6-6}
 this & Gb(fglm) & res  & this & Gb(fglm) & sqf dec & this & Gb(fglm)\\ \hline \hline
 0.01  &   11.79 (0.149)  &  0.54  &  3.04  &  18.05(10.05)  &  0.04  &  0.86  &   87.28 (5.5) \\ 
 0.02  &   90.69 (0.549)  &  1.29  &  8.27  &   91.21(43.519)  &  0.08  &  1.9  &   338.78 (20)\\ 
 0.04  &   297.74(1.23)  &  2.52  &  18.73  &   243.46(135)  &  0.13  &  3.94  &   808.97 (60) \\ 
 0.06  &   681.95(3.09)  &  4.44  &  37.28  &   560.26(326.289)  &  0.19  &  8.06  &   2186.52 (170)\\ 
 0.09  &   2402 (7.8)  &  7.76  &  68.63  &   1375.74(899.62)  &  0.27  &  15.39  &   5101.74 (300)\\ 
 0.13  &   5512.63(16.739)  &  12.67  &  134.43  &   2789(1909.229)  &  0.36  &  28.41  &   11789.98 (825)\\ \hline
\end{tabular}
\end{small}
\caption{Second family over a finite field $GF(p)$, $p$ a 64bits prime}\label{tab:mul-63bits}
\end{table}

\subsection{Comments
on Tables~\ref{tab:2-15bits}-\ref{tab:2-63bits}-\ref{tab:mul-15bits}-\ref{tab:mul-63bits}}
\paragraph{Columns~1-2}
The Gr\"obner engine is unable to take full advantage of the addition of $T$ in the system.
The comparison between the timings of the first two
columns is telling: up to several thousands times faster and always more than
one hundred times faster. 

\paragraph{Columns~6-7-8}
Taking the squarefree decomposition $R_1^{e_1}\cdots R_s^{e_s}$
of the resultant $r=\Res_y(a,b)$   is cheap and natural.
As one can expect, running Algorithm~\ref{algo:SToGBD5} on  each factor found 
brings a clear advantage in any case, even in comparison with Column~5 F4$+$FGLM.
This last experiment supports the conclusion that decomposing helps to alleviate
the growth of internal  computations.

\paragraph{Columns~3-4-5: preliminary comments}
This is related to the computation of a lexGb of a bivariate
system of two polynomials: the input is the system $[a,\ b]$.
Column 3 displays timings (always negligible) for computing the resultant $r$,
and then Column 4 displays the timing taken by Algorithm~\ref{algo:SToGBD5}
to find a product of lexGbs of $\l a,\ b,\ r\r$.
This timing compares with that of {\tt GroebnerBasis}$([a,b])$ of Column~5.
One question arises: is it faster to compute {\tt GroebnerBasis}$([a,b])$ or to compute
{\tt GroebnerBasis}$([a,b,r])$ ?
Experiments show that almost always the first is faster. That is why we have  shown the timings
of {\tt GroebnerBasis}$([a,b])$, rather than those
{\tt GroebnerBasis}$([a,b,r])$ which are slower.

We observe that the situation is more balanced in this case.
First the computation of $r$ is most often negligible. This suggests for future work to take advantage of its computation
through a subresultant p.r.s of $a,b$ (not modulo a polynomial $T$)
and work directly on it. 
Second, the timings of the first family of polynomials (Tables~\ref{tab:data1}-\ref{tab:2-15bits}-\ref{tab:2-63bits}
and rows~1-2-3 of Table~\ref{tab:DEG})
are slower in general, whereas those  of the second family 
(Tables~\ref{tab:data2}-\ref{tab:mul-15bits}-\ref{tab:mul-63bits}
and
row~4 of Table~\ref{tab:DEG})
are better.
To get an educated knowledge of this situation, which is of practical importance,
let us provide more details.

\subsection{Elements of complexity in case of  input $a$ and $b$}\label{sec:compl}
This section extends  the above comment on the timings of Columns~3-4-5 of Tables~\ref{tab:2-15bits}-\ref{tab:2-63bits}-\ref{tab:mul-15bits}-\ref{tab:mul-63bits} that concern an input $a$ and $b$ without a modulus $T$. For clarity, Table~\ref{tab:DEG}
displays these data into 8 plots.
Magma computes with F4~\cite{F4} a {\tt grevlex} Gr\"obner basis and then applies
FGLM~\cite{FGLM93} to transform it to a lexGb.
The complexity of the latter is well-known, it is $O(\DEG^3)$,
where $\DEG=\dim_k(R[y]/\l a,\ b \r)$ is the 
degree of the ideal generated by $a$ and $b$.
The complexity of the F4 algorithm is notoriously more
difficult to analyze.
Only the less efficient Lazard's 
algorithm relying on  Macaulay  matrices has a known
complexity upper bound, written in Eq.~\eqref{eq:Mac} below.
The complexity results given hereafter
can be found in~\cite{Laz83} (see also~\cite[Section~1.5]{SpaenPhD} for a nice overview and the references therein).
The standard assumption in the context of~\cite{Laz83} is:
\begin{equation}\label{tag:R}
  \text{    \begin{minipage}{0.85\textwidth}
      $a$ and $b$ form a  homogeneous regular sequence
\end{minipage}
  }
\tag{R}
\end{equation}
The examples of polynomials
$a$ and $b$  in the testing suite are
not homogeneous,
but form a regular sequence.
The complexity of Lazard's method
in the affine case is not well-understood.
Nonetheless, the bounds for the
homogenous case provide some
indications for affine polynomials.
We define the {\em degree of regularity} $\dreg$
of the ideal $\l a,\ b\r$, equal under~\eqref{tag:R}
to be the smallest integer such that
the Hilbert function of $k[x,y]/\l a,\ b\r$
becomes equal to the corresponding Hilbert polynomial.
Then the number of arithmetic operations over $k$
to compute a {\tt grevlex} Gr\"obner
basis of $\l a,\ b\r$ is lower than:
\begin{equation}\label{eq:Mac}
O\left(\binom{2+\dreg}{2}^\omega \right) = O(\dreg^{2\omega})
\end{equation}
We refer to~\cite{GaGe03} for details about the exponent of linear algebra $\omega$.
It is  possibly overestimated in many situations.
Considering that the Macaulay bound $\dreg=\tdeg(a) +\tdeg(b)-1$ holds
under~\eqref{tag:R} we obtain:
$$
O(\tdeg(a)^{2\omega} + \tdeg(a)^\omega \tdeg(b)^\omega + \tdeg(b)^{2\omega})\le O(\tdeg(a)^{2\omega})
$$
The latter inequality follows from $\deg_Y(b)\le\deg_y(a)$.
Adding the cost of FGLM in $\DEG^3$  yields:
\begin{equation}
\label{eq:compl-F4}
O(\tdeg(a)^{2\omega} + \tdeg(a)^\omega \tdeg(b)^\omega + \tdeg(b)^{2\omega} + \DEG^3)\le O(\tdeg(a)^{2\omega} + \DEG^3).
\end{equation}
Even in the affine case,
we can think that the time taken by F4 depends more on $\tdeg(a)$ and $\tdeg(b)$,
whereas the one taken by FGLM depends exclusively on $\DEG$.
The comparison of the timings of the examples
of the 1st family and of the 2nd one supports this claim:
\begin{itemize}
\item Column~5 of Tables~\ref{tab:2-15bits}-\ref{tab:2-63bits}, as well as the plots in rows~1-2-3 of Table~\ref{tab:DEG} in the 1st family show that FGLM occupies most of the computations. We observe that
  the ratio $\DEG/\tdeg(a)$ ranges from 10 to at least 36 (Table~\ref{tab:data1}), putting a higher  cost on FGLM.
\item Column~5 of Tables~\ref{tab:mul-15bits}-\ref{tab:mul-63bits}, as well as the plots in row~4 of Table~\ref{tab:DEG}
  in the 2nd family shows that FGLM takes more or less half of the total time.
   We observe that
  the ratio $\DEG/\tdeg(a)$ ranges from 9 to around 18 (Table~\ref{tab:data1}). The cost of FGLM is then
less important.
\end{itemize}

\begin{table}
\begin{tabular}{cc}
  \includegraphics[height=4.5cm, width=8.5cm]{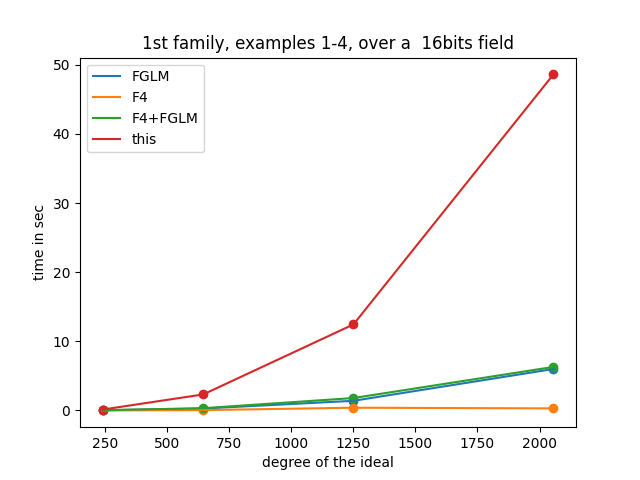}
  &
  \includegraphics[height=4.5cm, width=8.5cm]{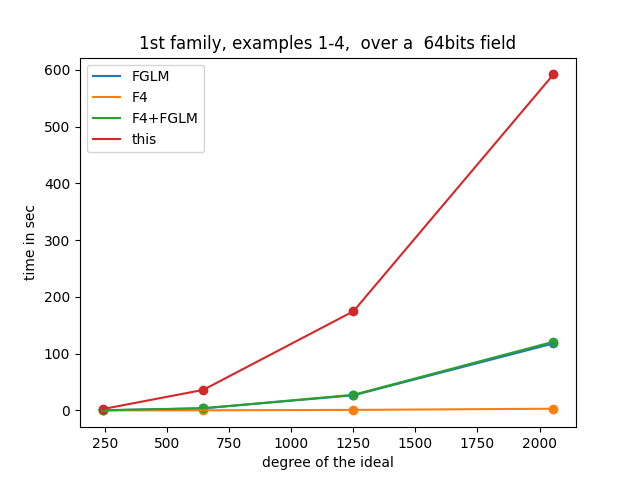}\\
    \includegraphics[height=5cm, width=8.5cm]{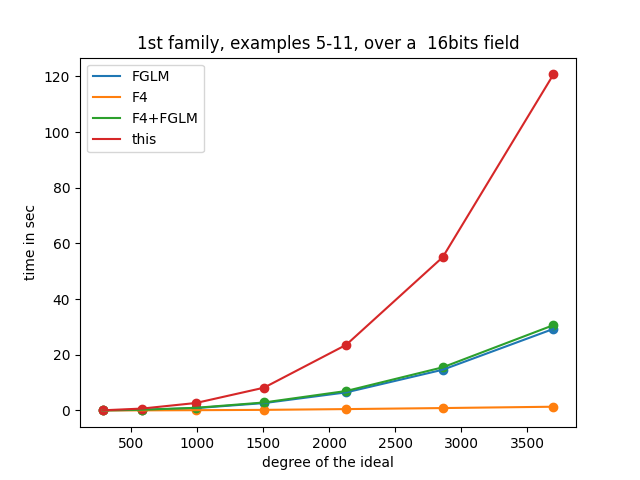}
  &
  \includegraphics[height=5cm, width=8.5cm]{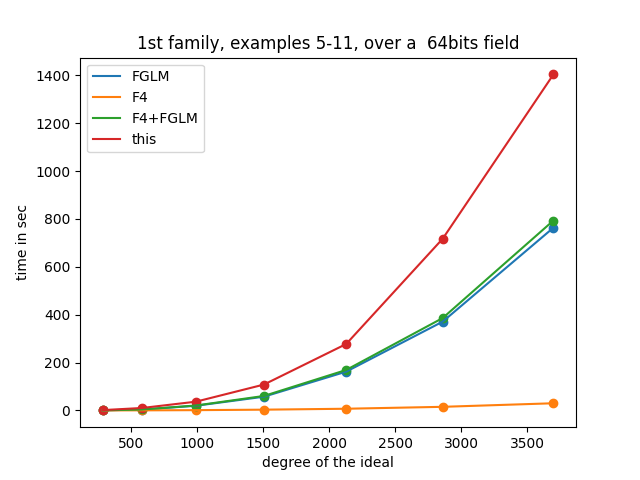}\\
    \includegraphics[height=4.5cm, width=8.5cm]{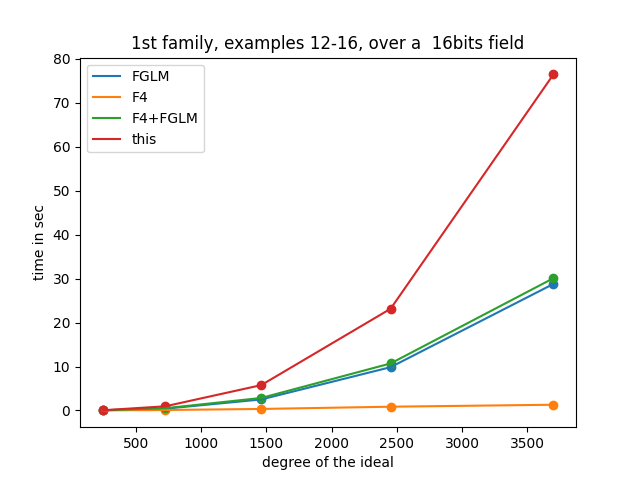}
  &
    \includegraphics[height=4.5cm, width=8.5cm]{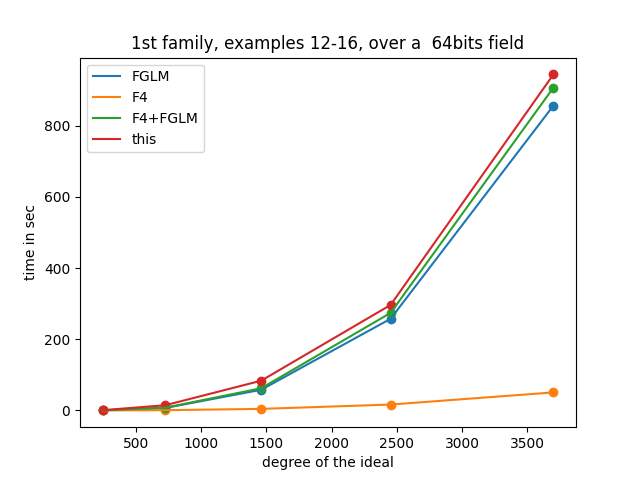}\\
        \includegraphics[height=5cm, width=8.5cm]{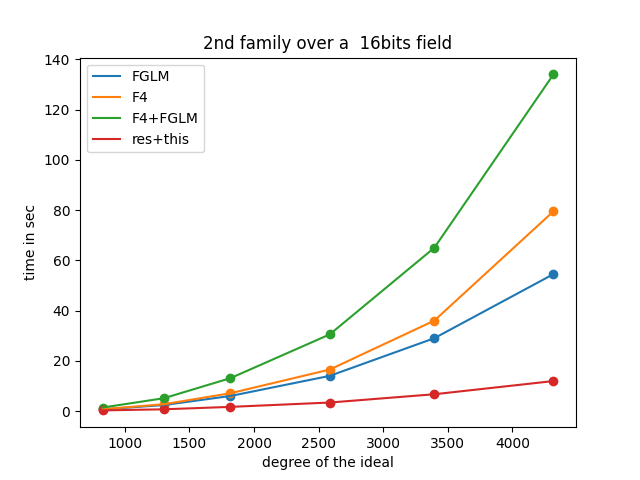}
  &
        \includegraphics[height=5cm, width=8.5cm]{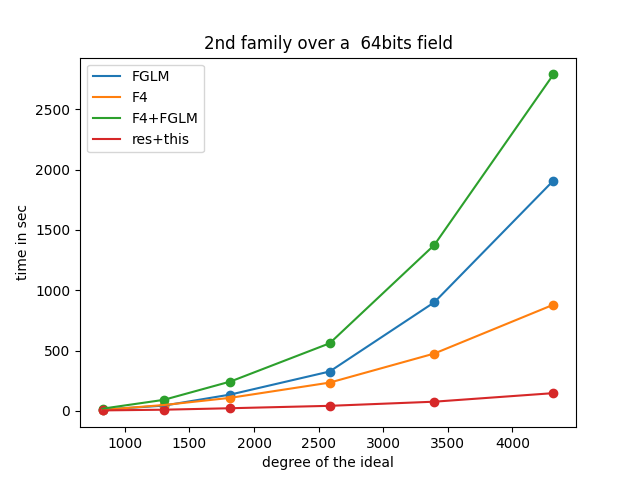}
\end{tabular}
\caption{Input $a,b$ (Columns~3-4-5 of Tables~\ref{tab:2-15bits}-\ref{tab:2-63bits}-\ref{tab:mul-15bits}-\ref{tab:mul-63bits}). Timing against the degree $\DEG$ of the ideal.
  {\sf Left:} 16bits field.\qquad  {\sf Right:} 64bits field. {\sf Row 1:} examples~1-4 of the 1st family. {\sf Row 2:} examples 5-11 of the 1st family.
  {\sf Row 3:} examples 12-16 of the 1st family. {\sf Row 4:} examples of the 2nd family}
\label{tab:DEG}
\end{table}

In comparison, Algorithm~\ref{algo:SToGBD5} ``$\SToGBD5$''
certainly has the subresultant as a central routine.
The number of arithmetic operations in $R=k[x]/\l T\r$ to compute naively the subresultant p.r.s of
$a$ and $b$ modulo $T$, when divisions are possible, is well-known, upper-bounded by $\deg_y(a)\deg_y(b)$.
The number of arithmetic operations in $k$ for peforming naively  any operation in $R$ is $O(\deg(T)^2)$.
Thus the cost of computing the subresultant p.r.s modulo $T$ is within:
\begin{equation}\label{eq:costD5}
O(\deg_y(a) \deg_y(b) \deg(T)^2).
\end{equation}
Without entering into details, other functions to make polynomials
monic and to split the computations do increase the cost,
but unlikely  of a much higher magnitude compared to Eq~\eqref{eq:costD5}.
When $T=\Res_y(a,b)$, the classical upper-bound~\cite[Thm.~6.22]{GaGe03}
$\deg_x(\Res_y(a,b))\le d_x (\deg_y(a)+ \deg_y(b))$ where $d_x=\max(\deg_x(a),\ \deg_y(b))$ yields:
\begin{equation}\label{eq:compl-D5}
O(d_x^2(\deg_y(a)^3\deg_y(b) + \deg_y(a)^2\deg_y(b)^2+\deg_y(a)\deg_y(b)^3))\le O(d_x^2 \deg_y(a)^4).
\end{equation}
The latter inequality is deduced from  $\deg_y(b)\le \deg_y(a)$. 
Thus the parameter $\deg_y(a)$ dominates the total cost.
\begin{itemize}
\item Looking at Columns~3-4
of Table~\ref{tab:data1} (1st family), we observe that the ratio $\tdeg(a)/\deg_y(a)$
ranges from around 2 to 4,
which is quite small. The estimate~\eqref{eq:compl-F4}
involving $\tdeg(a)$ and the estimate~\eqref{eq:compl-D5}
involving $\deg_y(a)$ suggest
that F4 will spend less time
than Algorithm~\ref{algo:SToGBD5}.
\item  Whereas this ratio in Columns~3-4 of Table~\ref{tab:data2} (2nd family) ranges from around 12 to 19,
which is quite bigger.
This indicates that Algorithm~\ref{algo:SToGBD5} will
spend less time than F4.
\end{itemize}
 This is certainly one main
 explanation behind the difference in the  timings of ``$\SToGBD5$'' and F4$+$FGLM
 observed in Columns~4-5 of Tables~\ref{tab:2-15bits}-\ref{tab:2-63bits}-\ref{tab:mul-15bits}-\ref{tab:mul-63bits},
 and in Table~\ref{tab:DEG}.

\section{Concluding Remarks}\label{sec:concl}
The new algorithms brought in this work open the way to various generalizations
and some improvements.

\paragraph{Subresultant p.r.s of $a$ and $b$}
The line of works~\cite{gonzalez1996improved,diochnos2009asymptotic,BoLaMoPoRoSa16,LaPoRo17,li2011modpn} that  produces a triangular decomposition
of the set of solutions (not ideal preserving) of $a$ and $b$, 
works directly on the subresultant p.r.s of $a$ and $b$.
The present article requires a modulus $T$ and works on the p.r.s modulo $T$.
Although this kind of situations is necessary for more general incremental Euclidean algorithm-based
decomposition methods, starting with the p.r.s of $a$ and $b$ directly
is desirable. Besides the discussion of Section~\ref{sec:compl}
suggests enhanced performances.

\paragraph{Decomposition along the $y$-variable}
In the case of a radical ideal, decomposing it along the $y$-variable is well-known.
This is a special instance of dynamic evaluation in two variables. It amounts to compute
a ``gcd'' in $y$ of two bivariate polynomials. The inversions occurring in the coefficients ring $k[x]$
required by the Euclidean
algorithm are managed by dynamic evaluation in one variable, namely the $x$-variable.
In this way, dynamic evaluation in two variables boils down to that of one variable.  
To do the same for non-radical lexGbs, two complications occur: first, Hensel lifting comes into play
to compensate the division required to make nilpotent polynomials non-nilpotent, second,
the output may be a lexGb not necessarily a triangular set.
Computing a certain decomposition (as introduced in~\cite{Da17})
of the ideal $\l a,\ b,\ p^e\r$ and decomposing a minimal lexGb of $\l a,\ b, \ p^e\r$, necessarily
along the $y$-variable, does not differ much. 
Note that Lazard's structural theorem applies in this situation too.
This constitutes a necessary step towards the generalization of this work
to more than two variables.

\paragraph{More than two variables}
Besides a dynamic evaluation in two variables,
a form of structure theorem for lexGbs of three variables or more 
becomes necessary.
For now,  it is known to hold in the radical case~\cite{FeRaRo06,MaMo03},
not sufficient for our purpose since we  target  non-radical ideals.
Evidences that a broad class of non-radical lexGbs also holds a factorization pattern
are coming out~\cite{wang2016connection}. A key step in that direction lies in the development
of CRT-based algorithms that reconstruct a lexGb from its  primary components~\cite{Da17macis-short}.

Beyond this first step,
the idea would be a recursive algorithm on the number of variables
in order to  ultimately reduce  to the bivariate case treated here.

\paragraph{Complexity}
Instead of standard quadratic-time subresultant p.r.s,
can half-gcd based algorithm be adapted ?
This would motivate to analyze the  complexity carefully.
The work of~\cite{MoSc16} paves the way to this direction.
However, the additional algorithms appearing here compared to~\cite{MoSc16},
which is already quite sophisticated, complicate it furthermore,
making the complexity analysis an interesting challenge.

\bibliographystyle{plain}
\bibliography{/home/xav/Volume/GoogleDrive/main}

\begin{thebibliography}{10}

\bibitem{AL94}
William~Wells Adams and Philippe Loustaunau.
\newblock {\em An introduction to Gr{\"o}bner bases}.
\newblock Number~3. American Mathematical Soc., 1994.

\bibitem{Alvandi2015-short}
Parisa Alvandi, Marc~Moreno Maza, {\'E}ric Schost, and Paul Vrbik.
\newblock A standard basis free algorithm for computing the tangent cones of a
  space curve.
\newblock In {\em CASC 2015 Proceedings}, pages 45--60. Springer International
  Publishing, Cham, 2015.

\bibitem{Auz88}
Winfried Auzinger and Hans~J Stetter.
\newblock An elimination algorithm for the computation of all zeros of a system
  of multivariate polynomial equations.
\newblock In {\em Numerical Mathematics Singapore 1988}, pages 11--30.
  Springer, 1988.

\bibitem{Bates2013}
Daniel~J Bates, Andrew~J Sommese, Jonathan~D Hauenstein, and Charles~W Wampler.
\newblock {\em Numerically solving polynomial systems with Bertini}.
\newblock SIAM, 2013.

\bibitem{Magma}
Wieb Bosma, John Cannon, and Catherine Playoust.
\newblock The {M}agma algebra system. {I}. {T}he user language.
\newblock {\em Journal of Symbolic Computation}, 24(3-4):235--265, 1997.
\newblock Computational algebra and number theory (London, 1993).

\bibitem{BoLaMoPoRoSa16}
Yacine Bouzidi, Sylvain Lazard, Guillaume Moroz, Marc Pouget, Fabrice
  Rouillier, and Michael Sagraloff.
\newblock Solving bivariate systems using rational univariate representations.
\newblock {\em Journal of Complexity}, 37:34--75, 2016.

\bibitem{caruso2017numerical}
Xavier Caruso.
\newblock Numerical stability of {E}uclidean algorithm over ultrametric fields.
\newblock {\em Journal de Th{\'e}orie des Nombres de Bordeaux}, 29(2):503--534,
  2017.

\bibitem{ChenMMM11}
Chengbo Chen and Marc Moreno~Maza.
\newblock Algorithms for computing triangular decompositions of polynomial
  systems.
\newblock In {\em Proceedings of ISSAC'11}, pages 83--90, New York, NY, USA,
  2011. ACM.

\bibitem{Cheng2014}
Jin-San Cheng and Xiao-Shan Gao.
\newblock Multiplicity-preserving triangular set decomposition of two
  polynomials.
\newblock {\em Journal of Systems Science and Complexity}, 27(6):1320--1344,
  2014.

\bibitem{DaMMMScXi06}
X.~{Dahan}, M.~{{Moreno Maza}}, {\'E}.~{Schost}, and Y.~{Xie}.
\newblock On the complexity of the {{D5}} principle.
\newblock In {\em Proc. of {\em Transgressive Computing 2006}}, Granada, Spain,
  2006.

\bibitem{DaSc04}
X.~Dahan and {\'E}.~Schost.
\newblock Sharp estimates for triangular sets.
\newblock In {\em ISSAC 2004}, pages 103--110. ACM Press, 2004.

\bibitem{Da17}
Xavier Dahan.
\newblock Gcd modulo a primary triangular set of dimension zero.
\newblock In {\em Proceedings of ISSAC'17}, pages 109--116, New York, NY, USA,
  2017. ACM.

\bibitem{Da17macis-short}
Xavier Dahan.
\newblock On the bit-size of non-radical triangular sets.
\newblock In {\em MACIS 2017, Vienna, Austria, November 15-17, 2017,
  Proceedings}, pages 264--269, Cham, 2017. Springer International Publishing.

\bibitem{decker1999primary}
Wolfram Decker, Gert-Martin Greuel, and Gerhard Pfister.
\newblock Primary decomposition: algorithms and comparisons.
\newblock In {\em Algorithmic algebra and number theory}, pages 187--220.
  Springer, 1999.

\bibitem{diochnos2009asymptotic}
Dimitrios~I Diochnos, Ioannis~Z Emiris, and Elias~P Tsigaridas.
\newblock On the asymptotic and practical complexity of solving bivariate
  systems over the reals.
\newblock {\em Journal of Symbolic Computation}, 44(7):818--835, 2009.

\bibitem{D5}
J.~Della Dora, C.~Dicrescenzo, and D.~Duval.
\newblock About a new method for computing in algebraic number fields.
\newblock In {\em EUROCAL '85: Research Contributions from the European
  Conference on Computer Algebra-Volume 2}, pages 289--290, London, UK, 1985.
  Springer-Verlag.

\bibitem{Du94}
D.~Duval.
\newblock Algebraic numbers: an example of dynamic evaluation.
\newblock {\em Journal of Symbolic Computation}, 18(5):429--445, 1994.

\bibitem{Emiris02}
Ioannis~Z Emiris and Victor~Y Pan.
\newblock Symbolic and numeric methods for exploiting structure in constructing
  resultant matrices.
\newblock {\em Journal of Symbolic Computation}, 33(4):393 -- 413, 2002.

\bibitem{FGLM93}
J.~C. Faug{\`e}re, P.~Gianni, D.~Lazard, and T.~Mora.
\newblock Efficient computation of zero-dimensional gr{\"o}bner bases by change
  of ordering.
\newblock {\em Journal of Symbolic Computation}, 16(4):329--344, 1993.

\bibitem{F4}
Jean-Charles Faug{\`e}re.
\newblock A new efficient algorithm for computing {G}r{\"o}bner bases ({F}4).
\newblock {\em Journal of pure and applied algebra}, 139(1-3):61--88, 1999.

\bibitem{F5}
Jean-Charles Faug{\`e}re.
\newblock A new efficient algorithm for computing {G}r{\"o}bner bases without
  reduction to zero ({F}5).
\newblock In {\em Proceedings of the 2002 international symposium on Symbolic
  and algebraic computation}, pages 75--83, 2002.

\bibitem{FeRaRo06}
B.~Felszeghy, B.~R{\'a}th, and L.~R{\'o}nyai.
\newblock The lex game and some applications.
\newblock {\em Journal of Symbolic Computation}, 41(6):663 -- 681, 2006.

\bibitem{GTZ88}
P.~Gianni, B.~Trager, and G.~Zacharias.
\newblock Gr\"obner bases and primary decomposition of polynomial ideals.
\newblock {\em Journal of Symbolic Computation}, 6:149--167, 1988.

\bibitem{GiLeSa01}
M.~Giusti, G.~Lecerf, and B.~Salvy.
\newblock A {G}r{\"o}bner free alternative for polynomial\ system solving.
\newblock {\em J. of Complexity}, 17(2):154--211, 2001.

\bibitem{gonzalez1996improved}
Laureano Gonz{\'a}lez-Vega and M'hammed El~Kahoui.
\newblock An improved upper complexity bound for the topology computation of a
  real algebraic plane curve.
\newblock {\em Journal of Complexity}, 12(4):527--544, 1996.

\bibitem{Hu1}
{\'E}velyne Hubert.
\newblock Notes on triangular sets and triangulation-decomposition algorithms.
  {I}. {P}olynomial systems.
\newblock In {\em Symbolic and numerical scientific computation (Hagenberg,
  2001)}, volume 2630 of {\em Lecture Notes in Comput. Sci.}, pages 1--39.
  Springer, Berlin, 2003.

\bibitem{lang2002algebra}
Serge Lang.
\newblock {\em Algebra (revised third edition)}, volume 211 of {\em Graduate
  Texts in Mathematics}.
\newblock Springer Science and Media, 2002.

\bibitem{La92}
D.~Lazard.
\newblock Solving zero-dimensional algebraic systems.
\newblock {\em Journal of Symbolic Computation}, 13:147--160, 1992.

\bibitem{Laz83}
Daniel Lazard.
\newblock Gr{\"o}bner bases, {G}aussian elimination and resolution of systems
  of algebraic equations.
\newblock In J.~A. van Hulzen, editor, {\em Computer Algebra}, pages 146--156,
  Berlin, Heidelberg, 1983. Springer Berlin Heidelberg.

\bibitem{Laz85}
Daniel Lazard.
\newblock Ideal bases and primary decomposition: case of two variables.
\newblock {\em Journal Symbolic Computation}, 1(3):261--270, 1985.

\bibitem{LaPoRo17}
Sylvain Lazard, Marc Pouget, and Fabrice Rouillier.
\newblock Bivariate triangular decompositions in the presence of asymptotes.
\newblock {\em Journal of Symbolic Computation}, 82:123--133, 2017.

\bibitem{LeMMMXi05}
F.~Lemaire, M.~Moreno Maza, and Y.~Xie.
\newblock The {\tt {r}egular{c}hains} library.
\newblock In {\em Maple conference 2005}, pages 355--368. I. Kotsireas Ed.,
  2005.

\bibitem{Li2003multiplicities}
Bang-He Li.
\newblock A method to solve algebraic equations up to multiplicities via
  {R}itt-{W}u's characteristic sets.
\newblock {\em Acta Analysis Functionalis Applicata}, 2:97--109, 2003.

\bibitem{li2011modpn}
Xin Li, Marc~Moreno Maza, Raqeeb Rasheed, and {\'E}ric Schost.
\newblock The modpn library: Bringing fast polynomial arithmetic into maple.
\newblock {\em Journal of Symbolic Computation}, 46(7):841--858, 2011.

\bibitem{Marcus2012-short}
Steffen Marcus, Marc~Moreno Maza, and Paul Vrbik.
\newblock On {F}ulton's algorithm for computing intersection multiplicities.
\newblock In {\em CASC 2012, Proceedings}, pages 198--211. Springer Berlin
  Heidelberg, Berlin, Heidelberg, 2012.

\bibitem{MaMo03}
M.~G. Marinari and T.~Mora.
\newblock A remark on a remark by {M}acaulay or enhancing {L}azard structural
  theorem.
\newblock {\em Bull. Iranian Math. Soc.}, 29(1):1--45, 85, 2003.

\bibitem{MeSc16}
Esmaeil Mehrabi and {\'E}ric Schost.
\newblock A softly optimal {M}onte {C}arlo algorithm for solving bivariate
  polynomial systems over the integers.
\newblock {\em Journal of Complexity}, 34:78--128, 2016.

\bibitem{mishra1993}
B~Mishra.
\newblock {\em Algorithmic Algebra}.
\newblock Text and monographs in Computer Science. Springer-Verlag, New
  York-Berlin-Heidelberg, 1993.

\bibitem{Mo03}
T.~Mora.
\newblock {\em Solving Polynomial Equation Systems I. The {Kronecker-Duval}
  Philosophy.}
\newblock Number~88 in Encyclopedia of Mathematics and its Applications.
  Cambridge University Press, 2003.

\bibitem{MMMOo95}
M.~{{Moreno Maza}} and R.~Rioboo.
\newblock Polynomial gcd computations over towers of algebraic extensions.
\newblock In {\em Proc. AAECC-11}, pages 365--382. Springer, 1995.

\bibitem{MoSc16}
Guillaume Moroz and {\'E}ric Schost.
\newblock A fast algorithm for computing the truncated resultant.
\newblock In {\em Proceedings of ISSAC'16}, pages 341--348, New York, NY, USA,
  2016. ACM.

\bibitem{musser1975}
David~R Musser.
\newblock Multivariate polynomial factorization.
\newblock {\em Journal of the ACM (JACM)}, 22(2):291--308, 1975.

\bibitem{Noro06}
Masayuki Noro.
\newblock Modular dynamic evaluation.
\newblock In {\em Proceedings of the 2006 International Symposium on Symbolic
  and Algebraic Computation}, ISSAC '06, page 262–268, New York, NY, USA,
  2006. Association for Computing Machinery.

\bibitem{SpaenPhD}
Spaenlehauer Pierre-Jean.
\newblock {\em R{\'e}solution de syst{\`e}mes multi-homog{\`e}nes et
  d{\'e}terminantiels}.
\newblock PhD thesis, Universit{\'e} de Pierre et Marie Curie - Paris VI, 2012.
\newblock
  \url{https://members.loria.fr/PJSpaenlehauer/data/these_spaenlehauer.pdf}.

\bibitem{pollard1974}
John~M Pollard.
\newblock Theorems on factorization and primality testing.
\newblock {\em Mathematical Proceedings of the Cambridge Philosophical
  Society}, 76(3):521--528, 1974.

\bibitem{Ri32}
Joseph-F. Ritt.
\newblock Differential equations from an algebraic standpoint.
\newblock {\em Colloquiium publications of the {AMS}}, 14, 1932.

\bibitem{Ro99}
F.~Rouillier.
\newblock Solving zero-dimensional systems through the rational univariate
  representation.
\newblock {\em Appl. Algebra Eng. Commun. Computat.}, 9(5):433--461, 1999.

\bibitem{van2020directed}
Joris Van Der~Hoeven and Gr{\'e}goire Lecerf.
\newblock Directed evaluation.
\newblock {\em Journal of Complexity}, page 101498, 2020.

\bibitem{VdH-resultant2020}
Joris {van der Hoeven} and Gr{\'e}goire Lecerf.
\newblock Fast computation of generic bivariate resultants.
\newblock {\em Journal of Complexity}, page 101499, 2020.

\bibitem{PHCpack}
Jan Verschelde.
\newblock Algorithm 795: Phcpack: A general-purpose solver for polynomial
  systems by homotopy continuation.
\newblock {\em ACM Transactions on Mathematical Software (TOMS)},
  25(2):251--276, 1999.

\bibitem{Vi-resultant-2018}
Gilles Villard.
\newblock On computing the resultant of generic bivariate polynomials.
\newblock In {\em Proceedings of the 2018 ACM International Symposium on
  Symbolic and Algebraic Computation}, ISSAC '18, page 391–398, New York, NY,
  USA, 2018. Association for Computing Machinery.

\bibitem{GaGe03}
Jocahim von~zur Gathen and J{\"u}rgen Gerhard.
\newblock {\em Modern computer algebra}.
\newblock Cambridge University Press, New York, NY, USA, 2003.
\newblock Second Edition.

\bibitem{wang2012elimination}
Dongming Wang.
\newblock {\em Elimination Methods}.
\newblock Texts \& Monographs in Symbolic Computation. Springer Vienna, 2012.

\bibitem{wang2016connection}
Dongming Wang.
\newblock On the connection between {R}itt characteristic sets and
  {B}uchberger--{G}r{\"o}bner bases.
\newblock {\em Mathematics in Computer Science}, 10(4):479--492, 2016.

\bibitem{WWT86}
Wen-Tsu Wu.
\newblock On zeros of algebraic equations - an application of {R}itt principle.
\newblock {\em Kexue Tongbao}, 5:1--5, 1986.

\end{thebibliography}

\bigskip

\noindent {\bf \Large Appendix}
\medskip

In below, is proposed  a version of the Chinese Remainder Theorem. Although probably not new,
finding a precise reference appears to be not obvious.

\begin{Lem}\label{lem:id}
Consider a family of ideals $(I_\ell)_{\ell \in L}$ of $k[x,y]$ and a polynomial $g\in k[x,y]$.
Take a monic generator $h_\ell$ of  $(I_\ell \cap k[x])$  and assume that the polynomials $h_\ell$'s are
pairwise coprime. Then  $\prod_{\ell \in L} (I_\ell + \l g\r) = (\prod_{\ell\in L} I_\ell) + \l g\r$.
\end{Lem}
\begin{proof}
Write $J = \prod_{\ell \in L} (I_\ell + \l g\r)$.
Let $L_1 \cup L_2 = L$ be a partition of the set of indices $L$ and associate to it the ideal
$(\prod_{\ell\in L_1} I_\ell) \l g\r^{\# L_2}$ where $\# L_2$ stands for the cardinal of $L_2$.
We then have the equality of ideals:
\begin{equation}\label{eq:J}
J = \sum_{{\rm partitions~} (L_1,\ L_2) {~\rm of~} L} \qquad (\prod_{\ell\in L_1} I_\ell) \l g\r^{\# L_2}.
\end{equation}
When $L_2=\emptyset$, the ideal associated to the partition $(L,\ \emptyset)$ of the set of indices $L$
is then $A:=\prod_{\ell\in L} I_\ell$.
As soon as $L_2\not=\emptyset$, the ideal associated to the partition $(L_1,\ L_2)$
is $ (\prod_{\ell\in L_1} I_\ell ) \l g \r ^{\# L_2}\subset \l g\r$.
Therefore we have:
$$
J \subset \l g\r + A.
$$
Since $A\subset J$, it suffices to prove that $g\in J$. To this end, consider
the monic polynomial $h_\ell$ generating the ideal $I_\ell \cap k[x]$. Let 
$$
H_{\ell'} = \prod_{\ell \in L,\ \ell\not=\ell'} h_\ell,\qquad \text{for all } \ell'\in L.
$$
The polynomials $H_{\ell'}$s are product of polynomials coprime with $h_{\ell'}$, hence
is coprime with $h_{\ell'}$ hence invertible modulo $h_{\ell'}$.
Write $F_{\ell'} := H_{\ell'}^{-1} \bmod h_{\ell'}$ in $k[x]$ of degree smaller than that of $h_{\ell'}$.
In particular $\deg_x(H_{\ell'} F_{\ell'}) < \deg_x(\prod_{\ell\in L} h_\ell):=D$ for all $\ell'\in L$. 
We have:
\begin{equation}\label{eq:H}
\sum_{\ell\in L} H_{\ell} F_{\ell} =1.
\end{equation}
Indeed, the polynomial $H=\sum_{\ell\in L} H_{\ell} F_{\ell}$ is of degree in $x$ smaller than 
$D:=\deg_x(\prod_{\ell\in L} h_\ell)$.
Because $H_{\ell'} \equiv  0 \bmod h_\ell$ as soon as $\ell'\not=\ell$,
we have  $H \equiv H_\ell F_\ell \bmod h_\ell$ and since $F_\ell \equiv H_\ell^{-1} \bmod h_\ell$ we have 
$H_\ell F_\ell \equiv 1 \bmod h_\ell$, whence $H\equiv 1 \bmod h_\ell$.
Since the $h_\ell$ are pairwise coprime, the Chinese remainder theorem 
implies that $H \equiv 1 \bmod (\prod_{\ell \in L} h_\ell)$.
But since $\deg_x(H)<D$ this means $H=1$. This ends the proof of Eq.~\eqref{eq:H}.

Consider the partitions $[(L\setminus \{\ell\},\ \{\ell\})]_{\ell\in L}$ of $L$. Their associated ideals
are $ (\prod_{\ell'\in L,\ \ell'\not=\ell} I_{\ell'}) \l g\r$, for $\ell \in L$.
Observe that $H_\ell g \in (\prod_{\ell'\in L,\ \ell'\not=\ell} I_{\ell'} ) \l g\r$,
hence $H_\ell F_\ell g \in (\prod_{\ell'\in L,\ \ell'\not=\ell} I_{\ell'}) \l g\r$.
Therefore $\sum_{\ell \in L} H_\ell F_\ell g \in \sum_{\ell \in L} (\prod_{\ell'\in L,\ \ell'\not=\ell} I_{\ell'}) \l g\r$.
By Eq.~\eqref{eq:H}, we obtain $g\in \sum_{\ell \in L} (\prod_{\ell'\in L,\ \ell'\not=\ell} I_{\ell'}) \l g\r$.
Finally, Eq.~\eqref{eq:J} shows that $g\in J$.
\end{proof}

\end{document}